\documentclass[3p,12pt, reqno]{amsart}
\usepackage{etex}
\usepackage[margin=1.25in]{geometry}
\usepackage[latin1]{inputenc}
\usepackage{amsmath}
\usepackage{stackengine}
\usepackage{amsthm}
\usepackage[alphabetic]{amsrefs}
\usepackage{amsfonts}
\usepackage[all,cmtip]{xy}
\xyoption{rotate}
\usepackage{amsfonts}
\usepackage{amssymb}
\usepackage{amscd}
\usepackage{bbm}
\usepackage{pictexwd,dcpic}
\usepackage[T1]{fontenc}
\usepackage{lmodern}
\usepackage{setspace}
\usepackage{mathpartir}
\usepackage{enumerate,multicol}
\usepackage[svgnames]{xcolor}
\usepackage{stmaryrd}
\usepackage[colorlinks=true, citecolor=Dark Green, linkcolor=Maroon, urlcolor=Purple]{hyperref}
\def\var{\text{FV}}
\def\FV{\text{FV}}
\def\obL{\text{ob}\L}

\def\Lrg{\mathcal{L}_{\text{rg}}}

\def\T{\mathbb{T}}

\def\C{\mathcal{C}}
\def\D{\mathcal{D}}
\def\L{\mathcal{L}}
\def\M{\mathcal{M}}

\def\Lcat{\mathcal{L}_{\text{cat}}}
\def\Lucat{\mathcal{L}_{\text{ucat}}}
\def\Lgraph{\mathcal{L}_{\text{graph}}}
\def\Lprecat{\mathcal{L}_{\text{precat}}}

\def\Tprecat{\mathbb{T}_{\text{precat}}}
\def\Tcat{\mathbb{T}_{\text{cat}}}

\def\Eqintro{(\rsort)}
\def\Jrule{\text{(J)}}
\def\wJrule{\text{(wJ)}}
\def\trrule{(\tsort\rsort)}
\def\DFOLDS{\D_{\text{FOLDS}}}
\def\DFOLDScl{\D_{\text{FOLDS}}^{\text{cl}}}
\def\Diso{\D_{\cong}}
\def\Disocl{\D_{\cong}^{\text{cl}}}
\def\Tgpd{\T_{\text{gpd}}}
\def\FOLDS{\textbf{FOLDS}}

\newcommand{\proptrunc}[1]{\vert\vert #1 \vert\vert}

\newcommand\frakfamily{\usefont{U}{yfrak}{m}{n}}
\DeclareTextFontCommand{\textfrak}{\frakfamily}

\theoremstyle{plain}\newtheorem{lemma}{Lemma}[section]
\theoremstyle{plain}\newtheorem{cor}[lemma]{Corollary}%[section]
\theoremstyle{plain}\newtheorem{theo}[lemma]{Theorem}%[section]
\theoremstyle{plain}
\theoremstyle{definition}\newtheorem{defin}[lemma]{Definition}
\theoremstyle{remark} \newtheorem{exam}[lemma]{Example}
\theoremstyle{remark} \newtheorem{remark}[lemma]{Remark}
\theoremstyle{plain}\newtheorem{prop}[lemma]{Proposition}%[section]
\theoremstyle{plain}\newtheorem{problem}[lemma]{Problem}%[section]
\theoremstyle{definition}\newtheorem{construction}[lemma]{Construction}
\theoremstyle{definition}\newtheorem*{notation}{Notation}
\theoremstyle{definition}\newtheorem*{terminology}{Terminology}
\theoremstyle{definition}\newtheorem*{convention}{Convention}

\author{Dimitris Tsementzis}
\title{First-Order Logic with Isomorphism}
\begin{document}

\date{\today}
%\begin{keyword}
%Univalent Foundations, Categorical Logic, Homotopy Type Theory
%\MSC[2010] 55P10, 03G99
%\end{keyword}

\address{Rutgers University}
\curraddr{New Brunswick, NJ 08544, USA}
\email{dt506@stat.rutgers.edu}

%%%DEFINITIONS%%%
\def\v{\:\vert\:}

%Category Modifiers
\newcommand{\core}[1]{#1_g}
\newcommand{\op}[1]{#1^{\text{op}}}
\newcommand{\ob}[1]{\text{Ob}#1}
\newcommand{\morph}[1]{\text{Mor}#1}
\newcommand{\dom}[1]{\text{dom}(#1)}
\newcommand{\cod}[1]{\text{cod}(#1)}
\newcommand{\inv}[1]{#1^{-1}}
\newcommand{\isocat}[1]{\textbf{#1}_{\text{iso}}}
\newcommand{\slice}[2]{#1 / #2}
\newcommand{\coslice}[2]{#1 / #2}
\newcommand{\discrete}[1]{s(#1)}
\newcommand{\catelem}[1]{\int #1}
\newcommand{\Nattrans}[2]{\text{Nat}(#1,#2)}
\newcommand{\constfun}[1]{\hat{#1}}
\newcommand{\prodmap}[2]{\langle #1,#2 \rangle}
\newcommand{\coprodmap}[2]{\langle #1,#2 \rangle}
\newcommand{\Logical}[1]{\text{L}(#1)}

\newcommand{\fourpartdef}[8]
{
	\left\{
		\begin{array}{lll|}
			#1 & \mbox{if } #2 \\
			#3 & \mbox{if } #4 \\
			#5 & \mbox{if } #6 \\
			#7 & \mbox{if } #8 \\
		\end{array}
	\right.
}

\newcommand{\twopartdef}[4]
{
	\left\{
		\begin{array}{ll}
			#1 & \mbox{if } #2 \\
			#3 & \mbox{if } #4 
		\end{array}
	\right.
}

\newcommand{\threepartdef}[6]
{
	\left\{
		\begin{array}{lll}
			#1 & \mbox{if } #2 \\
			#3 & \mbox{if } #4 \\
			#5 & \mbox{if } #6 
		\end{array}
	\right.
}

%Univalent Type Theory
\newcommand{\transport}[1]{#1^{\tau}}
\newcommand{\level}[1]{l(#1)}
\newcommand{\height}[1]{h(#1)}
\newcommand{\typesofhlevel}[1]{\mathcal{U}^{#1}}
\newcommand{\U}{\mathcal{U}}
\newcommand{\Pitype}[1]{\underset{#1}{\Pi}}
\newcommand{\Sigmatype}[1]{\underset{#1}{\Sigma}}
\newcommand{\Idtype}[3]{\mathtt{Id}_{#1}(#2,#3)}
\newcommand{\typed}[1]{\widehat{#1}}
\def\Idtype{\mathtt{Id}}
\def\Pitype{\Pi}
\def\Sigmatype{\Sigma}
\def\singletontype{\mathbf{1}}
\def\zerotype{\mathbf{0}}

\newcommand{\Pitypecom}[2]{\Pitype \: (#1) \:\: #2}
\newcommand{\Sigmatypecom}[2]{\Sigmatype \: (#1) \:\: #2}

\def\emptycontext{\varnothing}

\newcommand{\NL}[1]{\text{NL}(#1)}
\newcommand{\HStruc}[1]{\text{HStruc}(#1)}
\newcommand{\HStruce}[1]{\text{\emph{HStruc}}(#1)}
\newcommand{\truncated}[1]{\vert \vert #1 \vert \vert}

\newcommand{\SMod}[2]{\textbf{SatMod}_{#1}^{#2}}
\newcommand{\SMode}[2]{\textbf{\emph{SatMod}}_{#1}^{#2}}
\newcommand{\SStruc}[2]{\textbf{SatStruc}_{#1}^{#2}}
\newcommand{\Struc}[1]{\text{Struc}(#1)}
\newcommand{\Struce}[1]{\textbf{\emph{Struc}}_{#1}}
\newcommand{\Mod}[1]{\textbf{Mod} (#1)}
\newcommand{\Mode}[1]{\textbf{\emph{Mod}}(#1)}

\newcommand{\glob}[1]{#1^{=}}
\newcommand{\setdefinition}[2]{
\lbrace #1 \:\vert\: #2 \rbrace
}
\newcommand{\ts}[1]{#1^{\tau}}

%%%%% Syntactic symbols %%%%%

\def\eqdef{=_{\text{df}}}

%%%%% Mathematical Symbols %%%%%

\def\setprod{\times}
\def\catid{1}
\def\catcomp{\circ}
\def\vertcomp{\cdot}
\def\horcomp{\circ}
\def\obC{\text{ob}\:\C}
\def\obD{\text{ob}\D}
\def\catiso{\cong}
\def\Cop{\C^{\text{op}}}
\def\IdC{\text{Id}_\C}
\def\IdD{\text{Id}_\D}
\def\Cg{\C_{\text{g}}}
\def\catprod{\times}
\def\E{\mathcal{E}}
\def\commacat{\downarrow}
\def\Nat{\text{Nat}}
\def\idnat{\iota}
\def\yoneda{\mathbf{y}}
\def\coyoneda{\mathbf{y}^{\text{op}}}
\def\existsunique{\exists !}

\newcommand{\card}[1]{\vert #1 \vert}

\def\idenrule{(\text{iden})}

%Specific Definitions
\def\Ob{\text{Ob}}
\def\Mor{\text{Mor}}
\def\rsort{\rho}
\def\eqsort{\cong}
\def\tsort{\tau}
\def\rmor{r}
\newcommand{\tmorX}[1]{{#1}_1}
\newcommand{\tmorXX}[1]{{#1}_2}
\newcommand{\tmorpath}[1]{e_{#1}}
\def\Natinfty{\mathbb{N}_\infty}
\newcommand{\Heightof}[1]{H(#1)}
\def\Ltycat{\L_{\text{tycat}}}
\def\hSig{\textbf{hSig}}
\def\Gmonad{\mathcal{G}}
\def\Tmonad{\mathcal{T}}
\def\Fmonad{\mathcal{F}}
\def\FOLDSeq{\textbf{FOLDS}_=}
\def\FOLDSeqe{\textbf{\emph{FOLDS}}_=}
\def\FOLDSeqtrans{\text{FOLDS}_{=,\tau}}
\def\FOLiso{\text{FOL}_{\cong}}
\def\FOLisocat{\textbf{FOL}_{\cong}}
\newcommand{\depprod}[2]{\overset{#2}{\underset{#1}{\longrightarrow}}}
\def\refl{\mathtt{refl}}
\def\trans{\mathtt{trans}}
\newcommand{\ttype}[2]{T_{#1}}
\newcommand{\vars}[1]{\vert #1 \vert}
\newcommand{\smashcont}[2]{#1 * #2}
\newcommand{\inttype}[1]{\widehat{#1}}
\def\isnotid{\neq 1}
\def\obord{<_o}
\def\morord{<_m}
\def\fresh{\#}
\newcommand{\Fin}[1]{\textbf{Fin}(#1)}

\newcommand{\Lrghsignature}[3]{\xymatrix{
#3 &I \ar[d]^{i} \\
#2 &A \ar@/^/[d]^{d} \ar@/_/[d]_{c} \\
#1 &O
}}

\def\isacontext{\:\:\textbf{ok}}
\def\isatype{\:\:\textbf{Sort}}
\def\isatypee{\:\:\textbf{\emph{Sort}}}

\def\emptycon{(\text{con-}\varnothing)}
\def\conext{(\text{con-ext})}
\def\weakening{(\text{con-wk})}
\def\axrule{(\text{ax})}
\def\Kform{(K\text{-form})}

\newcommand\interp[2]{#1^{#2}}
\newcommand\interpretation[1]{\llbracket #1 \rrbracket}
\def\FOLDScontexts{\textbf{Con}_\L}
\def\MLTTcontexts{\textbf{Con}_{\text{MLTT}}}
\def\FOLDSsorts{\textbf{Sort}_\L}
\def\MLTTterms{\textbf{Term}_{\text{MLTT}}}
\def\FOLDSformulas{\textbf{Formulas}_\L}
\def\MLTTtypes{\textbf{Type}_{\text{MLTT}}}
\def\FOLDSjudgments{\textbf{Jud}_\L}
\def\MLTTjudgments{\textbf{Seq}_{\text{MLTT}}}
\def\FOLDSvariables{\textbf{Var}}
\def\TTisatype{\:\mathbf{Type}}

\def\conmorphism{\Rightarrow}

\def\onetype{\mathbf{1}}
\def\zerotype{\mathbf{0}}
\newcommand\TTElapp[1]{\mathtt{El}( \mathtt{app} [#1])}
\def\LTT{LTT}
\def\TTapp{\mathtt{app}}
\def\ttapp{\mathtt{app}}
\def\TTproj{\mathtt{proj}}
\def\Elements{\mathtt{El}}

%%%DEFINITIONS%%%

%\address{Department of Philosophy, Princeton University}
%\curraddr{1879 Hall, Princeton, NJ 08544, USA}
%\email{dtsement@princeton.edu}

%\maketitle

%%%% INFERRULE %%%%

\def\eqgap{.4ex}
\def\overgap{.4ex}
\def\inferrulerule{.2pt}

\newlength\rulelength
\newlength\toplength
\newlength\bottomlength

\newcommand\myinferrule[2]{%
  \stackMath%
  \setlength\bottomlength{\widthof{$#1$}}%
  \setlength\toplength{\widthof{$#2$}}%
  \ifdim\toplength>\bottomlength%
    \setlength\rulelength{\the\toplength}%
  \else%
    \setlength\rulelength{\the\bottomlength}%
  \fi%
  \mathrel{%
    \stackunder[\overgap]{%
      \stackon[\overgap]{%
        \stackanchor[\eqgap]%
          {\rule{\the\rulelength}{\inferrulerule}}%
        {\rule{\the\rulelength}{\inferrulerule}}%
      }{#2}%
    }{#1}%
  }%
}

\def\FOLeq{\text{FOL}_{=}}

\begin{abstract}
%A model theory in the framework of 
The Univalent Foundations requires a logic that allows us to define structures on homotopy types, similar to how first-order logic with equality ($\FOLeq$) allows us to define structures on sets.
We develop the syntax, semantics and deductive system for such a logic, which we call first-order logic with isomorphism ($\FOLiso$).
The syntax of $\FOLiso$ extends $\FOLeq$ in two ways. 
First, by incorporating into its signatures a notion of dependent sorts along the lines of Makkai's FOLDS as well as a notion of an $h$-level of each sort.
Second, by specifying three new logical sorts within these signatures: isomorphism sorts, reflexivity predicates and transport structure.
The semantics for $\FOLiso$ is then defined in homotopy type theory with the isomorphism sorts interpreted as identity types, reflexivity predicates as relations picking out the trivial path, and transport structure as transport along a path.
We then define a deductive system $\Diso$ for $\FOLiso$ that encodes the sense in which the inhabitants of isomorphism sorts really do behave like isomorphisms and
 prove soundness of the rules of $\Diso$ with respect to its homotopy semantics. 
Finally, as an application, we prove that precategories, strict categories and univalent categories are axiomatizable in $\FOLiso$.
\end{abstract}

\maketitle

\tableofcontents

\section*{Introduction}

The Univalent Foundations of Mathematics (UF) \cite{HTT} take their basic objects to be homotopy types.
In UF mathematical structures are therefore encoded as structured homotopy types, similar to how in set-theoretic foundations they are encoded as structured sets.
This basic picture allows us to envision a model theory in which mathematical objects are no longer described as structured sets, but rather as structured homotopy types.

In this paper we develop a logic through which this process can be carried out.
%In \cite{CK} model theory is defined as ``the branch of mathematical logic which deals with the relation between a formal language and its interpretations, or models.''
%The title ``Homotopy Model Theory'' reflects this ambition: i
%If ``set-theoretic model theory'' is understood as model theory in set-theoretic foundations then ``homotopy type-theoretic model theory'' would be the corresponding model theory in the Univalent Foundations.
%The crucial technical constraint of having the title fit in a single line forced us to compromise for ``Homotopy Model Theory''. 
%The aim of this paper is to develop a logic with which such a model theory can be carried out.
%In \cite{CK} model theory is defined as ``the branch of mathematical logic which deals with the relation between a formal language and its interpretations, or models.''
%Therefore, in order to do any kind of model theory we must at least have a formal language (\emph{syntax}) and an interpretation of that formal language (\emph{semantics}).
%This first paper will develop these two components. 
We will call it first-order logic with isomorphism ($\FOLiso$)
% (for  $-1 \leq n < \infty$)
%the syntax and semantics of the underlying logic.
%We call it ``$n$-logic'' (for  $-1 \leq n < \infty$) 
%because it is designed to cash out the 
%The title reflects 
in order to reflect the following fundamental idea of UF:
the primitive notion of equality in UF is itself a \emph{structure} that behaves like the structure of isomorphisms between structures, rather than the \emph{fact} of identity between sets.
%an ``$n$-level'' theory is a piece of syntax whose models consist of structures defined on $n$-groupoids/homotopy $n$-types. 
%Thus, $n$-logic defines a setting to carry out a model theory using the basic objects of UF directly, rather than through some ambient set theory. 
%by defining both the syntax and the semantics of (what we will call) $n$-logic.

The syntax of $\FOLiso$ will be based on an extension of the syntax of Makkai's FOLDS (First-Order Logic with Dependent Sorts) \cite{MFOLDS}. 
%A key insight of Makkai was that the syntax of ``higher-level'' theories can be presented as a category.
%In particular, t
The signatures of FOLDS can be understood as inverse categories where the arrows encode variable dependencies between the objects (understood as ``sorts'').
The key addition of $\FOLiso$ to FOLDS is to add a notion of an ``$h$-level'' to each sort as well as notions of ``isomorphism sorts'', ``reflexivity predicates'' and ``transport structure'' understood as \emph{logical} sorts (i.e. with a fixed denotation) defined on top of the non-logical sorts in accordance with the latter's $h$-level. 
%analogous to how we add equality to first-order logic to obtain first-order logic with equality. 
%Indeed, i
It is helpful to think of the syntax of $\FOLiso$ as standing to FOLDS in the same relation that  the syntax of first-order logic with equality ($\FOLeq$) stands to plain first-order logic.
To add equality to a first-order signature $\Sigma$, one simply adds a binary relation with a certain fixed denotation.
% (or, equivalently, satisfying axioms that turn it into an equivalence relation).
%What is the analogous process for a FOLDS-signature $\L$?
%The syntax of $n$-logic is obtained by providing an answer to this question 
The analogous process for a FOLDS signature $\L$ is carried out in terms of an
 ``isomorphism completion'' operation on inverse categories which attaches logical sorts to pre-existing sorts in $\L$ in a manner compatible with their $h$-level.
%Roughly speaking, the signatures of $\FOLiso$ are then the ``isomorphism completed'' FOLDS signatures of Makkai.
%This marks the first main contribution of this paper: a general definition of an ``$n$-level'' syntax.
%The main difficulty here is defining a notion of equality for FOLDS that behaves as much as possible like identity types in intensional Martin-L\"{o}f Type Theory (MLTT). 
%We achieve this by ``globularly completing'' FOLDS signatures in an appropriate sense and adding rules in our deductive system that mimic MLTT-style rules for identity types. 
%The deductive system we end up with is the main innovation of our approach.

The intended semantics for $\FOLiso$ is in UF, where in particular ``isomorphism sorts'' are interpreted as path spaces (if one thinks in terms of abstract homotopy theory) or identity types (if one thinks in terms of homotopy type theory).
% and we call them ``homotopy semantics''.
%We thus define a homotopy semantics for $n$-logic that follows naturally from the definition of the syntax.
The basic idea of the ``homotopy semantics'' of $\FOLiso$ is the following: non-logical sorts of $h$-level $m$ are interpreted as (dependent functions landing in) types of $h$-level $m$ and we make this precise by interpreting $\FOLiso$ into a version of ``book HoTT'', i.e. the formal system outlined in \cite{HTT}.
The isomorphism sorts that have been added to the syntax of $\FOLiso$ are then interpreted as identity types, the reflexivity predicates are interpreted as predicates picking out $\refl$ and transport structure is interpreted as (a relational version of) the transport function $\trans$ induced by a given path.

Finally, we present a deductive system $\Diso$ on the syntax of $\FOLiso$ that can justify its homotopy semantics.
%We present such a proof system $\Diso$ for $\FOLiso$.
It is based on a standard sequent calculus for first-order logic to which we add three rules:
a rule asserting the existence of the ``reflexivity'' isomorphisms, a rule specifying that ``transport along reflexivity'' does nothing and, crucially, a propositional version of the J-rule of MLTT. 
%Since the syntax of $n$-logic is purely relational (no closed terms) these assume an unfamiliar form,
%but they are in spirit very much related to the rules for identity types of MLTT \cite{Lof84}.
We can then prove that the deductive system is sound with respect to the above-sketched semantics, which is the main result of this paper. 
%The proofs in each case proceed very differently, and are illuminating of the general framework in different ways

%This is the first of a projected three papers on $n$-logic.
%%, culminating in the author's announced (dis)proof of the so-called Simplicial Definability Conjecture. 
%Part II will be devoted to proving suitable completeness theorems for the system outlined in Part I. 
%Part III will focus on applications to the problem of the definability of semi-simplicial types.

\vspace{0.3cm}

\textbf{Related Work.} Our work connects to several strands of ongoing work in UF.
The connection between inverse categories and homotopy theory is well-known and has been developed extensively by Shulman in \cite{shulman2015reedy, ShulmanUID} and our work builds on this connection, albeit from a more syntactic viewpoint, rather than within the categorical semantics of HoTT.
Our framework also provides a general definition of a (logical) signature for UF and could thus be used to generalize the Structure Identity Principle of \cite{HTT} as indeed has been sketched in \cite{uniFOLDS, TsemHSIP}.
Formalizations of category theory in the style of FOLDS has also been carried out by Ahrens in UniMath \cite{UniMath} and his formalization overlaps with some of the material of Section \ref{nlogExamples}.
Furthermore, our system $\FOLiso$ (as well as FOLDS) can be thought of as a dependently-typed first-order logic with specific features that make it suitable for UF, and in that regard it connects to the work of Palmgren \cite{PalmgrenFOLDS}.
Finally, although we do not (yet) adopt this point of view, $\FOLiso$ can be thought of as a proposal to add a ``logic layer'' to HoTT along the lines of so-called logic-enriched type theory \cite{GambinoAczel} and therefore is connected to the work of \cite{LuoPart}.
%Secondly, we intend  $\FOLiso$ as a tool for the study of homotopy type theories as mathematical objects themselves in such a way that can itself be formalized inside UF.
%In other words, $n$-logic can be used as a framework for doing metamathematics ``natively'' in UF, independent of any ambient set theory.
%This relates to ongoing work on the ``Initiality Conjecture'' in the series of papers by Voevodsky beginning with \cite{Csystems} (except we note that for its specific purposes this work takes place in ZF set theory).
%It also relates to work on ``internal setoids'' \cite{Palmgren2014} which is focused on implementing tools to carry out the metatheory of type theories inside type theory. 

\vspace{0.3cm}

\textbf{Outline of the Paper.} In Section \ref{prelim} we introduce the syntax of FOLDS as well as a proof system $\DFOLDS$ for this syntax and state some preliminary results. In Section \ref{FOLDSinMLTT} we define an interpretation of the syntax of FOLDS into MLTT via a direct interpretation of the raw syntax of the former system into the raw syntax of the latter. In Section \ref{nlogSyntax} we build on the syntax of FOLDS to define the syntax of $\FOLiso$. In Section \ref{nlogSemantics} we define the homotopy semantics of $\FOLiso$ by defining an interpretation of the syntax of $\FOLiso$ into HoTT, building on the interpretation of FOLDS into MLTT.
% and in Section \ref{nlogSTsemantics} we define a set-theoretic semantics for $1$-logic. 
In Section \ref{nlogicProof} we define a proof system $\Diso$ for $\FOLiso$ and prove soundness for $\Diso$ with respect to our homotopy semantics. 
%In Section \ref{nlogCompleteness} we prove completeness for the case $n=1$ and discuss the issue of obtaining a complteness proof for general $n$. 
Finally, in Section \ref{nlogExamples} we show how precategories, strict categories and univalent categories can be axiomatized in terms of $\FOLiso$.

\section{Preliminaries on FOLDS}\label{prelim}

We will assume familiarity with the basics of categorical logic as well as of intensional Martin-L\"{o}f Type Theory and its homotopy interpretation. For category-theoretic background \cite{CWM} remains the standard reference; for the basics of dependent type theory we recommend \cite{Lof84, Hof97}; for the homotopy interpretation and an introduction to the Univalent Foundations see \cite{HTT, KLV12} and references therein.
We will now present in more detail the basic syntax of FOLDS inspired by Makkai's original presentation in \cite{MFOLDS}. 

\def\nonidentity{\rightarrow}
\newcommand\toplevel[1]{\text{top}(#1)}

\begin{defin}[Inverse Category]\label{inversecategory}
An \textbf{inverse category} is a category $\L$ such that
\begin{enumerate}
\item $\L$ has no non-identity endomorphisms
\item $\L$ is skeletal 
\item For any object $K$ there is a finite number of arrows with domain $K$
\end{enumerate} 
\end{defin}

\begin{remark}
An inverse category is so-named because it can be thought of as a category in which all arrows ``go in one direction'' and in particular a category in which all arrows ``point downwards''. Condition (3) in Definition \ref{inversecategory} (which Makkai called ``finite fan-out'') ensures that every object in an inverse category is only a ``finite path'' away from an object at the ``bottom''. This allows us to carry out inductive definitions on inverse categories as indeed we will be doing throughout this paper.
\end{remark}

%\begin{example}
%
%\end{example}
%
%\begin{example}
%
%\end{example}

\newcommand{\pcosieve}[2]{#1//#2}

\begin{notation}
We write $\Natinfty$ for the set $\mathbb{N} \cup \lbrace \infty \rbrace$.
We will sometimes write $K \in \L$ as an abbreviation for $K \in \ob{\L}$ and we write $\L(K,K')$ for the morphisms from $K$ to $K'$ in $\L$. 
For any arrow $f \colon K \rightarrow K'$ in an inverse category $\L$ we will write $K_f$ for its codomain $K'$.
%We write $f \colon K \nonidentity K'$ for a non-identity arrow, i.e. for an arrow such that $K_f \neq K$ and when we do not require to make the domain and codomain explicit or when it is clear from the context we will write simply $f \isnotid$.
For any object $K \in \ob{\L}$ we write $\pcosieve{K}{\L}$ for the set (indeed, cosieve) of non-identity arrows with domain $K$.
%, or simply $\pcosieve{K}{}$ if $\L$ is clear from the context.
The objects of any inverse category $\L$ can be stratified into levels, defined (inductively) as follows:
\[
l(K) = \left\{
	\begin{array}{ll}
		0  & \mbox{if $K$ is the domain only of $1_K$} \\
		\underset{f \in \pcosieve{K}{\L}}{\text{sup}}l(K_f)+1 & \mbox{otherwise} 
	\end{array}
\right.
\]
%For $K \in \ob{\L}$ and $f \colon K \rightarrow K'$ and $\mathbf{p}=(p_1,\dots,p_n)$ the ordered $n$-tuple of arrows in $\pcosieve{K'}{\L}$ (i.e. all the non-identity arrows ``out of'' $K'$) we will denote simply by $(\mathbf{p}f)$ the $n$-tuple $(p_1f,\dots,p_nf)$.
%and the notation $A(\mathbf{p}f)$ for some type/sort expression $A$ below will denote the substitution $A[p_1f/p_1, \dots, p_nf/p_n]$ or $A[\mathbf{p}f/\mathbf{p}]$ for short. (This ``substitution will of course mainly make sense if $A$ is defined in a context that includes $p_1,\dots,p_n$, as indeed will be the case when we apply the notation in the rest of the paper.)
\end{notation}

\begin{terminology}
We will call $l(K)$ the \emph{level} of $K$. 
We will generally refer to the objects of $\L$ as \emph{sorts}. 
We call a non-identity arrow $f \colon K \rightarrow A$ \emph{top-level (for $K$)} if it does not factor through another arrow, i.e. if there is no non-identity arrows $h,k \in \morph{\L}$ such that $f=hk$ and we write $\toplevel{K}$ for the set of top-level arrows of a sort $K$.
%(We will usually reserve the letter $R$ for the former and the letter $K$ for the latter.) 
\end{terminology}

\begin{defin}\label{properorder}
A \textbf{proper order} $<$ on an inverse category is a pair $(\obord,\morord)$ where:
\begin{itemize}
\item $\obord$ is a partial order on $\ob{\L}$ such that if there is a non-identity arrow $f \colon K \nonidentity K_f$ then $K_f \obord K$
\item $\morord$ is a partial order on $\morph{\L}$ such that $f \morord g$ if $\cod{f} \obord \cod{g}$ 
%or if there is a non-identity $h$ such that $g=f \catcomp h$
 \end{itemize}
\end{defin}

\begin{notation}
For a proper order on $\L$ we will write simply $K<K'$ and $f<g$ for the ordering on objects and morphisms respectively. For every ordered tuple that is indexed by a set of arrows in $\L$ we will from now on asume that it is indexed in accordance with a given proper order on $\L$. For example, $(pf)_{f \in \pcosieve{K_f}{\L}}$ stands for the tuple $(p_1f,\dots,p_nf)$ where $p_1 < \dots <p_n$ are the morphisms of $\pcosieve{K_f}{\L}$ (i.e. all the non-identity arrows ``out of'' $K_f$). When the domain of a certain list of arrows is obvious we will denote simply by $(\mathbf{p}f)$ the $n$-tuple $(p_1f,\dots,p_nf)$. This convention will be used extensively below, especially when we have other expressions indexed by arrows of $\L$, e.g. we will write $(x_{\mathbf{p}f})$ for the list $(x_{p_1f},\dots,x_{p_mf})$.
\end{notation}

It is important for our purposes to define inverse categories over certain fixed sets of symbols for objects and arrows, since these symbols will be used to extract a formal syntax out of such inverse categories below. So now let $O,M$ be arbitrary disjoint countably infinite sets.

\begin{defin}[Inverse category over $(O,M)$]\label{inversecategoryOM}
An \textbf{inverse category over $(O,M)$} (or \textbf{$(O,M)$-inverse category}) is an inverse category $\L$ such that $\ob{\L} \subset O$, $\morph{\L} \subset M \amalg \setdefinition{1_K}{K \in O}$ and the identity on an object $K \in \ob{\L}$ is given by the symbol $1_K$.
\end{defin}

\begin{defin}[FOLDS signature over $(O,M)$]\label{FOLDSsignature}
A \textbf{FOLDS signature over $(O,M)$} (or $(O,M)$\textbf{-FOLDS-signature}) is a pair $(\L,<)$ where $\L$ is an $(O,M)$-inverse category and $<$ is a proper order on $\L$.
We write $\FOLDS(O,M)$ for the category whose objects are the FOLDS signatures over $(O,M)$ and whose arrows are the order-preserving functors between them, i.e. an arrow $I \colon (\L,<) \rightarrow (\L',<')$ is given by a functor $I \colon \L \rightarrow \L'$ such that $K < K'$ implies $I(K) < I(K')$ and $f<f'$ implies $I(f)<I(f')$.
\end{defin}

\begin{terminology}
Once fixed we will generally leave the choice of $(O,M)$ implicit and speak simply of \emph{FOLDS signatures} and write simply $\FOLDS$ for the associated category. We refer to morphisms in $\FOLDS$ as \emph{FOLDS morphisms}. We also write $\FOLDS_i$ for the important subcategory of $\FOLDS$ that consists of those FOLDS morphisms $I \colon \L \rightarrow \L'$ that are full, faithful and \emph{level-preserving} in the sense that $l(K)=l(I(K))$. We say that a FOLDS signature is \emph{finite} if its underlying inverse category is finite.
\end{terminology}

\begin{remark}
The importance of the morphisms of $\FOLDS_i$ is that they define interpretations of the syntax associated to a FOLDS signature, as we will see below.
\end{remark}

\begin{exam}
Let $\Lgraph$ denote the following FOLDS signature: 
%with the numbers on the left representing the dimension of the corresponding sorts.
\[
\xymatrix{
A \ar@/_/[d]_{d} \ar@/^/[d]^{c}\\
O
}
\]
As the name suggests this would be the signature useful to talk about graphs, where $A$ would encode the ``edges'' between previously declared ``vertices'' of sort $O$, $d$ the map picking out the ``domain'' of an arrow and $c$ the map picking out the ``codomain'' of an arrow.
A proper order on $\Lgraph$ would be given by $O<A$, $c<d$. Picking $d<c$ instead would give an isomorphic FOLDS signature.
%We will usually omit explicit mention of $d$, and write simply $\Lgraph$ for the above signature.
\end{exam}

\begin{exam}
Let $\Lrg$ denote the following FOLDS signature:
\[
\xymatrix{
I \ar[d]^{i} \\
A \ar@/^/[d]^{d} \ar@/_/[d]_{c} \\
O
}
\]
subject to the relation $di=ci$. Intuitively, this corresponds to the signature for reflexive graphs, where $I$ is a unary predicate that can only be ``asked'' of an ``arrow'' in $A$ that we already know is a loop. $\Lrg$ will serve as a running example throughout this paper.
A proper order on $\Lrg$ would be given by $O<A<I$, $c<d<di<i$. Picking $d<c$ instead would once again give an isomorphic FOLDS signature.
%for most of the concepts introduced in this paper 
%since it provides the simplest case sufficient to illustrate most of the concepts we are going to introduce.
\end{exam}

%with $\vert \text{ob} \L \vert < \aleph_0$ and $\vert \text{mor} \L \vert < \aleph_0$ together with a grading
%\[
%d \colon \obL \rightarrow \mathbb{Z}_{\geq -2}
%\]
%such that for any $f \colon K \rightarrow K_f$ with $f \neq 1_K$ we have $d(K_f) > d(K)$ (where we write $K_f$ for the codomain of $f$). We call $d(K)$ the \emph{dimension} of $K$.
%(it is to be understood as its homotopical dimension as will become clear in the semantics described in Seciton \ref{nlogSemantics}).
 %We call such a pair $(\L,d)$ a FOLDS \emph{signature}. 
% (as these terms are used in \cite{MFOLDS}). 
%We define the \emph{height} of $\L$ as
%\[
%h(\L,d) = \underset{K \in \obL}{\text{sup}} d(K)
%\]

% together with a proper order $<$ on $\L$. 
FOLDS signatures will play the role of logical signatures from which a formal syntax can be extracted.
The basic idea is that the objects of $\L$ encode sorts and the arrows in $\L$ encode sort dependencies.
We make this idea precise in the definitions that follow, in which we define certain simple dependent type theories whose rules and syntax are determined by the (assumed given) structure of $\L$.

%For the rest of this section we assume fixed a FOLDS signature $\L$ (over some fixed $(O,M)$).

\begin{notation}
We fix a countably infinite set of \emph{variables} $V$ together with a \text{fresh variable provider} $\fresh$ that takes any finite subset $S \subset V$ to a variable $y \notin S$. For any subset $S \subset V$ we write $x \fresh S$ to indicate that $x \notin S$. We will also be indexing lists of variables by arrows of $\L$ so it is worth making clear that the notation $(x_f)_{f \in \pcosieve{K}{\L}}$ refers to a list of variables of length equal to the cardinality of $\pcosieve{K}{\L}$ ordered in accordance with the order on $\L$. For example, in $\Lrg$ we have $(x_f)_{f \in \pcosieve{I}{\L}} = (x_{di}, x_i)$. We denote, when necessary, the empty list of variables by $\epsilon$ (e.g. when $\pcosieve{K}{\L} = \varnothing$).
\end{notation}

%\begin{convention}
%We make the following helpful convention that variables $f$ standing for arrows in a FOLDS signature are elements of $V$. The utility of this convention is to declutter notation in the definitions below and in the rest of this paper. (Alternatively one can identity the symbol $f$ when regarded as a variables as the image of $f$ under an injective function $\morph{\L} \hookrightarrow V$.)
%\end{convention}

\begin{defin}[Contexts, Sorts and Context Morphisms]
Let $\L$ be a FOLDS signature.
We define the \textbf{syntactic type theory $TT_\L$ of $\L$} as follows.
The syntax of $TT_\L$ is the following:
\begin{align*}
\FOLDSsorts \:\:\:K, K', ... &::= A (x_f)_{f \in \pcosieve{A}{\L}} \:\:\: (A \in \ob{\L}) \\
\FOLDSvariables \:\:\:\:\:\:\:\:\:\:\:x,y, ... &::= V \\
\FOLDScontexts \:\:\:\:\:\: \Gamma, \Delta, ... &::= \varnothing \:\:\vert\:\: \Gamma, x \colon K
\end{align*}
We write $\vars{\Gamma}$ for the set of variables that appear in a context $\Gamma$.

The judgments of $TT_\L$ are the following:
\begin{itemize}
\item $\Gamma \isacontext$
\item $\Gamma \vdash K \isatype$
\item $\Gamma \vdash x \colon K$
\end{itemize}
We write $\FOLDSjudgments$ for the judgment expressions of $TT_\L$.

The structural rules of $TT_\L$ are the following:
\[
\inferrule
{
\:
}
{
\varnothing \isacontext
}
\:\:\emptycon
\quad
\quad
\inferrule
{
\Gamma \isacontext \\ \Gamma \vdash K \isatype 
}
{
\Gamma, x \colon K \isacontext
}
\:\:\conext,  x \fresh \vars{\Gamma} \quad\quad
\]
\[
\inferrule
{
\Gamma,\Delta \isacontext \\ \Gamma \vdash K \isatype
}
{
\Gamma,\Delta \vdash K \isatype
}
\:\:\weakening
\quad
\quad
\inferrule
{
\Gamma, x \colon K, \Delta \isacontext 
}
{
\Gamma, x \colon K, \Delta \vdash x \colon K
}
\:\:\axrule
\]
For every $K \in \ob{\L}$ we have the following formation rule
\[
\inferrule
{
\big( \Gamma \vdash x_f \colon K_f (x_{\mathbf{p}f}) \big)_{f \in \pcosieve{K}{\L}}
}
{
\Gamma \vdash K (x_f)_{f \in \pcosieve{K}{\L}} \isatype
}
\:\:\Kform
\]
with the understanding that if $\pcosieve{K}{\L}  = \varnothing$ there is nothing above the line except $\Gamma \isacontext$ (which means that objects of level $0$ in $\L$ are types in the empty context), and the notation 
\[
\big( \Gamma \vdash x_f \colon K_f (x_{\mathbf{p}f}) \big)_{f \in \pcosieve{K}{\L}}
\]
indicates the list of judgments
\[
\Gamma \vdash x_{f_1} \colon K_{f_1} (x_{\mathbf{p}f_1}) \quad \dots \quad \Gamma \vdash x_{f_n} \colon K_{f_n} (x_{\mathbf{p}f_n})
\]
where $f_1, \dots, f_n$ are all the arrows in $\pcosieve{K}{\L}$.

A \textbf{context} of $\L$ (or $\L$\textbf{-context}) is a well-formed context in $TT_\L$, i.e. a context $\Gamma$ such that $\Gamma \isacontext$ is derivable in $TT_\L$.

A \textbf{sort} of $\L$ (or $\L$\textbf{-sort}) is a well-formed sort in $TT_\L$, i.e. a sort $K$ such that $\Gamma \vdash K \isatype$ is derivable in $TT_\L$.

Let $\Gamma$ and $\Delta = x_1 \colon K_1, \dots, x_n : K_n (x_1, \dots x_{n-1})$ be $\L$-contexts.
A well-formed \textbf{context morphism} $\Gamma \Rightarrow \Delta$ in $\L$ is given by the following derivable judgments
\begin{align*}
\Gamma &\vdash y_1 \colon K_1 \\
\Gamma &\vdash y_2 \colon K_2 [y_1/x_1] \\
 &\: \vdots \\
\Gamma &\vdash y_n \colon K_n [y_1/x_1,\dots,y_{n-1}/x_{n-1}]
\end{align*}
\end{defin}

\begin{notation}
We will usually write simply $O$ for a sort expression $O\:\epsilon$, i.e. for any sort $O$ in $\L$ of level $0$.
We will write $\alpha \colon \Gamma \conmorphism \Delta$ to indicate a context morphism, where $\alpha$ will be the list of variables $\mathbf{y}=(y_1,\dots,y_n)$ as in the above definition. 
Whenever convenient we will write simply $A(\mathbf{x})$ for $A(x_f)_{f \in \pcosieve{A}{\L}}$.
For a given sort $K = A (\mathbf{x}) $ with all variables appearing in the context $\Delta$ we will write $\alpha(K)$ for the sort obtained from $K$ by substituting all variable in $\mathbf{x}$ with those in $\mathbf{y}$. (That this operation preserves well-formedness of sorts is proved as Lemma \ref{subsortwf} below.)
We will use the notation $\equiv$ to refer to syntactic equality between expressions.
%, where $\alpha = (y_1,\dots,y_n)$ are the variables being substituted into $\Delta$.
\end{notation}

\begin{remark}
What we call a ``judgment'' above is usually called a ``judgment-in-context'' or a ``sequent'' but we will call it simply a judgment in order to avoid clashing with the term ``sequent'' in Definition \ref{formandseq} below.
\end{remark}

\begin{remark}\label{remarkinfinitary}
Note that if $\L$ is infinite (i.e. has an infinite number of objects), then $TT_\L$ has an infinite number of rules (since there is a $\Kform$ rule for every $K$ in $\L$). However it is still possible that $TT_\L$ is equivalent to (in an appropriate sense) a finitely presentable (in an appropriate sense) type theory even when $\L$ is infinite. We will take up this important point in Proposition \ref{FOLisofp} below.
\end{remark}

\newcommand{\FOLDSsortseg}[1]{\textbf{Sort}_{#1}}
\newcommand{\Kformeg}[1]{(#1\text{-form})}

\begin{exam}\label{TTLrgexample}
Consider $TT_{\Lrg}$. It has the following (raw) sort expressions
\[
\FOLDSsortseg{\Lrg} = \lbrace O\:\epsilon \rbrace \cup \lbrace A (x,y) \: \vdash \: x,y \in V \rbrace
\]
Its (raw) contexts therefore are of the following form (where we write simply $O$ for $O\:\epsilon$):
\[
x \colon O, f \colon A(x,y), g \colon A(z,w)
\]
Now we would expect to be able to derive that $A(x,y)$ is a well-formed sort for any $x,y \colon O$. We outline how this derivation goes, starting with $\emptycon$:
\[
%\inferrule
%{
%\:
%}
%{
\inferrule
{
\varnothing \isacontext
}
{
\inferrule
{
\varnothing \vdash O \isatype
}
{
\inferrule
{
x \colon O \isacontext
}
{
\inferrule
{
x \colon O \vdash O \isatype
}
{
\inferrule
{
x \colon O, y \colon O \isacontext \\
}
{
\inferrule
{
x \colon O, y \colon O \vdash x \colon O \\ x \colon O, y \colon O \vdash y \colon O
}
{
x \colon O, y \colon O \vdash A(x,y) \isatype
}
\Kformeg{A}
}
\axrule
}
\conext
}
\Kformeg{O}
}
\conext
}
\Kformeg{O}
%}
%\emptycon
\]
Now, as an example of a context morphism in
$\Lrg$ let $\Delta = x \colon O, y \colon O, f \colon A (x,y)$ and let $\Gamma = z \colon O, g \colon A(z,z)$. Then we have a context morphism $\alpha = (z,z,g) \colon \Gamma \conmorphism \Delta$ given by the following judgments
\begin{align*}
\Gamma &\vdash z \colon O \\
\Gamma &\vdash z \colon O [z/x]\\
\Gamma &\vdash g \colon A(x,y)[z/x][z/y] 
\end{align*}
all of which are derivable since $O [z/x] \equiv O$ and $A(x,y)[z/x][z/y] \equiv A(z,z)$.
\end{exam}

\begin{lemma}\label{appearlemma}
If $\Gamma \vdash x \colon K$ is derivable in $TT_\L$ then $x \colon K$ appears in $\Gamma$.
\end{lemma}
\begin{proof}
The only way to derive a judgment of the form $\Gamma \vdash x \colon K$ in $TT_\L$ is through $\axrule$ which means that if such a judgment is derivable then $\Gamma \equiv \Gamma' , x \colon K, \Delta$ for some contexts $\Gamma', \Delta$ which means that $x \colon K$ appears in $\Gamma$, as required.
\end{proof}

\begin{lemma}\label{subsortwf}
If $\alpha \colon \Gamma \conmorphism \Delta$ is a context morphism and $\Delta \vdash K \isatypee$ is derivable then so is $\Gamma \vdash \alpha(K) \isatypee$. In other words, the following rule is admissible in $TT_\L$:
\[
\inferrule
{
\alpha \colon \Gamma \conmorphism \Delta \ \\ \Delta \vdash K \isatypee
}
{
\Gamma \vdash  \alpha(K) \isatypee
}
\quad (\hspace{-0.17cm}\isatypee\text{\emph{-sub}})
\]
\end{lemma}
\begin{proof}
Let $\alpha (K) \equiv A (\mathbf{x})$. By Lemma \ref{appearlemma} every variable in $\mathbf{x}$ appears in $\Gamma$ and therefore the whole list of judgments needed to apply $\Kform$ will be derivable and therefore $\alpha(K)$ will be a sort in context $\Gamma$, as required.
\end{proof}

\def\isaformularule{\textbf{form}} 
\def\isaformula{\:\:\textbf{formula}} 
\def\isaformulae{\:\:\textbf{\emph{formula}}} 
\def\seqimplies{\Rightarrow}
\def\formwk{(\isaformularule\mathtt{-wk})}
\def\topform{(\top\mathtt{-form})}
\def\botform{(\bot\mathtt{-form})}
\def\starform{(*\mathtt{-form})}
\def\Qform{(Q\mathtt{-form})}

\begin{defin}[Formulas and Sequents]\label{formandseq}
We define the \textbf{logical type theory $LTT_\L$ of $\L$} as follows.
The syntax of $LLT_\L$ extends $TT_\L$ by adding a class of formulas:
\[
\FOLDSformulas \:\:\:\phi, \psi, ... ::= \bot \:\vert\: \top \:\vert\: \phi \wedge \psi \:\vert\: \phi \vee \psi \:\vert\: \phi \rightarrow \psi \:\vert\: \forall x \colon K. \phi \:\vert\: \exists x \colon K. \phi
\]
For any formula $\phi$ we write $\FV(\phi)$ for the set of variables that appear in $\phi$ unbound by $\exists$ or $\forall$, and we consider formulas only up to $\alpha$-equivalence (i.e. up to consistent renaming of their bound variables).

The judgments of $LTT_\L$ are those of $TT_\L$ together with:
\begin{itemize}
\item $\Gamma \vdash \phi \isaformula$
\end{itemize}
We once again write $\FOLDSjudgments$ for the judgment expressions of $\LTT_\L$.

The structural rules of $LTT_\L$ are those of $TT_\L$ together with:
\[
\inferrule
{
\Gamma, \Delta \isacontext \\ \Gamma \vdash \phi \isaformula
}
{
\Gamma, \Delta \vdash \phi \isaformula
}
\quad (\isaformularule\mathtt{-wk})
\]
%\[
%\inferrule
%{
%\alpha \colon \Gamma \conmorphism \Delta \ \\ \Gamma \vdash \phi \isaformula
%}
%{
%\Delta \vdash  \alpha(\phi) \isaformula
%}
%\isaformula\mathtt{-sub}
%\]
The formation rules of $LTT_\L$ are those of $TT_\L$ together with:
\[
\inferrule
{
\Gamma \isacontext
}
{
\Gamma \vdash \bot \isaformula
}
(\bot\mathtt{-form})
\quad
\quad
\inferrule
{
\Gamma \isacontext
}
{
\Gamma \vdash \top \isaformula
}
(\top\mathtt{-form})
\]
\[
\inferrule
{
\Gamma \isacontext \\ \Gamma \vdash \phi \isaformula \\ \Gamma \vdash \psi \isaformula
}
{
\Gamma \vdash \phi * \psi \isaformula
}
(*\mathtt{-form})
\]
\[
\inferrule
{
\Gamma \isacontext \\ \Gamma, x \colon K \vdash \phi \isaformula \\
}
{
\Gamma \vdash Qx \colon K.\phi  \isaformula
}
(Q\mathtt{-form})
\]
where $*=\vee,\wedge,\rightarrow$ and $Q=\exists,\forall$.
We define $\phi$ to be an \textbf{$\L$-formula in context $\Gamma$} whenever $$\Gamma \vdash \phi  \isaformula$$ is derivable in $LTT_\L$.
%$\top$ and $\bot$ are atomic formulas and $\text{FV}(\top)=\text{FV}(\bot)=\varnothing$. If $\phi$, $\psi$ are formulas then so are $\phi \wedge \psi$, $\phi \vee \psi$ and $\phi \rightarrow \psi$ with $$\FV(\phi \wedge \psi) = \FV(\phi \vee \psi) = \FV(\phi \rightarrow \psi) = \FV(\phi) \cup \FV(\psi)$$ Negation $\neg\phi$ can be defined as $\phi \rightarrow \bot$. For the quantifiers, assume that $\phi$ is a formula and that $x$ is a variable of sort $K$ in $\FV(\phi)$. Then $\forall x \colon K.\phi$ and $\exists x \colon K.\phi$ are formulas provided that $x \in \FV(\phi)^\uparrow$. In that case we have 
%\[
%\FV(\forall x \colon K.\phi ) = \FV(\exists x \colon K.\phi ) = \FV(\phi) \cup \text{dep} (x) ) \setminus \lbrace x \rbrace
%\]
A \textbf{sequent} is a syntactic entity of the form
\[
\Gamma \: \vert \: \phi \seqimplies \psi
\]
where both $\phi, \psi$ are $\L$-formulas in context $\Gamma$.
%(We follow \cite{Jacobs} in using the notation $\Gamma \: \vert \: \phi \vdash \psi$ in order to avoid overloading subscripts.)
Given a context morphism $\alpha = (y_1,\dots,y_n) \colon \Gamma \Rightarrow \Delta$ and a $\L$-formula $\phi$ in context $\Delta$ we define the formula expression $\alpha(\phi)$ inductively as follows:
\begin{itemize}
\item $\alpha( \top) = \top$
\item  $\alpha(\bot) = \bot$
\item $\alpha (\phi * \psi) = \alpha (\phi) * \alpha (\psi)$,  for $*=\vee,\wedge,\rightarrow$
\item $\alpha (Qx \colon K.\phi) = Qx \colon \alpha(K). \alpha(\phi)$,  for $Q=\exists,\forall$.
\end{itemize}
 %$\Gamma=x_1 \colon K_1,\dots, x_n \colon K_{n-1}(x_1,\dots,x_{n-1})$, then $\alpha (\phi) = \phi [y_1/x_1,\dots,y_n/x_n]$ is the substitution of $\phi$ along $\alpha$.
\end{defin}

\begin{remark}
It is reasonable to wonder whether we can generate through an inductive process every FOLDS signature $\L$ and therefore generating all $\LTT_\L$ in parallel with that process. Indeed this can be done, as has been outlined in \cite{TsemWeavFIC} and (in a slighlty different form) in \cite{PalmgrenFOLDS}.
\end{remark}

\begin{lemma}
The following rule is admissible in $\LTT_\L$:
\[
\inferrule
{
\alpha \colon \Gamma \conmorphism \Delta \ \\ \Delta \vdash \phi \isaformulae
}
{
\Gamma \vdash  \alpha(\phi) \isaformulae
}
\quad (\hspace{-0.17cm}\isaformulae\mathtt{-sub})
\]
\end{lemma}
\begin{proof}
Straightforward by structural induction over formula expressions.
\end{proof}

\begin{notation}
We will usually write a $\L$-formula $\phi$ in context $\Gamma$ as $\Gamma.\phi$.
For formulas of the form $\exists x \colon K(\textbf{y}). \top$ we will use the shorter and more standard form $K(\textbf{y})$.
\end{notation}

\begin{exam}
In $\LTT_{\Lrg}$ we can derive as in Example \ref{TTLrgexample}
\[
x \colon O, y \colon O \vdash A(x,y) \isatype
\]
from which we get
\[
x \colon O, y \colon O, f \colon A(x,y) \isacontext
\]
By applying $(\top\mathtt{-form})$ we get
\[
x \colon O, y \colon O, f \colon A(x,y) \vdash \top \isaformula
\]
and then by applying $(\exists)$ we get
\[
x \colon O, y \colon O \vdash  \exists f \colon A(x,y). \top \isaformula
\]
which we will write, as suggested above, simply as $A(x,y)$ with the understanding of course that this formula is stating that \emph{there exists} an ``inhabitant'' of the sort $A(x,y)$.
This notation thus makes more sense in situations where the given sort is best understood as a predicate, as is the case with the ``identity predicate'' $I$ in $\Lrg$.
% (which allows us to think of top-level sorts in a FOLDS signature as predicate symbols).
%In $\Lrg$, for example, $I(f,x)$ will be syntactic sugar for
Similar to what we just did we have a well-formed formula $\exists \tau \colon I(f,x). \top$ (in context $x \colon O, f \colon A(x,x)$) which we will sugar as $I(f,x)$, and which can be understood as saying that there exists a witness to the fact that $f$ is an identity ``arrow'' on the ``object'' $x$.
\end{exam}

All the definitions so far have depended on a specific choice of a proper order on $\L$.
We would like to know that this choice is irrelevant.
The following definitions aim at making this precise by defining a notion of interpretation between type theories of a very general kind. The reader familiar and comfortable with notions such as derivation, admissibility etc. may skip the next few definitions (up to Corollary \ref{isoimpliesbiint}) which play no essential role in this paper other than to make precise the sense in which isomorphic FOLDS signatures give rise to equivalent (syntactic and logical) type theories.

\begin{defin}
A \textbf{ML type theory} is a triple $TT=(S, J, R)$ where:
\begin{itemize}
\item $S=(S_1,\dots,S_k)$ is a list of sets called \textbf{syntactic components} 
\item $J = (1,J_1, \dots, J_l)$ is a list of \textbf{judgments} where $1$ is a singleton and each $J_i = S_{\alpha(1)} \times \dots \times S_{\alpha(m_i)}$ with $1 \leq \alpha(j) \leq k$ for each $1 \leq j \leq m_i$ and $m_i \geq 1$.
\item $R = (R_i)_{i \in \mathbb{N}}$ is a (possibly infinite) list of \textbf{rules} where each $R_i$ is a (possibly partial) function
\[
J_{\beta(1)} \times \dots \times J_{\beta(r_i)} \rightarrow J_{n_i}
\]
with $1 \leq \beta(j), n_1 \leq l$ and $r_i \geq 0$ with the understanding that $R_i = 1 \rightarrow J_{n_i}$ if $r_i = 0$.
\end{itemize}
%of of lists $R_i = ( (s_1, \dots, s_k) , t)$ where each $s_1, \dots, s_k, t$ is an element of $J_l$ for some $l$.
\end{defin}

\begin{exam}
For any FOLDS signature $\L$, $TT_\L$ is an ML type theory $(S_{\L}, J_\L, R_\L)$ where
\begin{itemize}
\item $S_\L = (\FOLDSsorts, \FOLDSvariables, \FOLDScontexts)$
\item $J_\L = (1, \FOLDScontexts, \FOLDScontexts\times\FOLDSsorts, \FOLDScontexts\times\FOLDSvariables\times\FOLDSsorts)$
\item $R_\L = (\emptycon, \conext, \weakening, \axrule) \amalg (\Kform_{K \in \L})$. As an example of how these rules can be defined as (partial) functions we have 
\[
\emptycon \colon 1 \rightarrow \FOLDScontexts
\]
defined by $\emptycon (*) = \varnothing$ (where $*$ is the unique element of $1$).
\end{itemize}
To obtain $\LTT_\L$ we add the syntactic class of $\FOLDSformulas$ to $S_\L$, a judgment of the form $\FOLDScontexts \times \FOLDSformulas$ to $J_\L$ and the logical rules to $R_\L$.
\end{exam}

%\begin{notation}
%For a set $X$ we write $List(X)$ for the set of finite lists on $X$ and if $X$ is a set of sets by abuse of notation we will write $x \in List(X)$ for an element of a given list (of sets) in $X$.
%\end{notation}

\begin{defin}
Let $TT=(S,J,R)$ be a ML type theory and let $\sigma \in J_{\gamma(1)} \times \dots \times J_{\gamma(m)}$ (for some $1  \leq \gamma (i) \leq l$) and $\tau \in J_m$ for some $1 \leq m \leq l$.
A \textbf{derivation} in $TT$ is a list $D= (d_1,\dots,d_n) \in R^n$ such that 
\[
d_n \circ \dots \circ d_1 (\sigma) = \tau
\]
A function $F \colon J_{\beta(1)} \times \dots \times J_{\beta(r_i)} \rightarrow J_{n_i}$ is an \textbf{admissible rule} in $TT$ (or simply \textbf{admissible in $TT$}) if for every $\sigma \in J_{\beta(1)} \times \dots \times J_{\beta(r_i)}$ for which $F$ is defined there is a derivation in $TT$ from $\sigma$ to $F(\sigma)$.
%A list $r=((s_1,\dots,s_k),t)$ is an \textbf{admissible rule} in $TT$ if there is a derivation in $TT$ from $(s_1, \dots, s_k)$ to $t$.
\end{defin}

\begin{defin}
Let $TT=(S,J,R)$ and $TT' = (S', J', R')$ be ML type theories. An \textbf{interpretation $I \colon TT \rightarrow TT'$ of ML type theories} consists of
\begin{itemize}
\item A family of functions $(I_i \colon S_i \rightarrow S'_j)_{1 \leq i \leq k, 1 \leq i \leq k'}$ 
%\item A function $I_{\text{rules}} \colon R \rightarrow R'$
\end{itemize}
such that
\begin{itemize}
\item $I(R_1)$ is admissible in $TT'$
\item $I(R_n)$ is admissible in $TT'$ given that each of $I(R_1), \dots, I(R_{n-1})$ are also admissible
\end{itemize}
where if $R_i \colon J_{\beta(1)} \times \dots \times J_{\beta(r_i)} \rightarrow J_{n_i}$ we define $I(R_i) \colon I(J_{\beta(1)}) \times \dots \times I(J_{\beta(r_i)}) \rightarrow I(J_{n_i})$ by the obvious induced action whenever the domain is in the image of $I$ and undefined otherwise.
\end{defin}

\def\MLTT{\textbf{MLTT}}

\begin{defin}
There is an obvious \textbf{identity interpretation} $1_{TT}$ given by the identity function on $S$ and given two interpretations $I \colon TT \rightarrow TT'$ and $J \colon TT' \rightarrow TT''$ we get the \textbf{composite interpretation} $J \circ I$ defined by the composite function $J \circ I \colon S \rightarrow S''$. We thus obtain a category $\MLTT$ whose objects are ML type theories and morphisms are the interpretations.
We say that two ML type theories are \textbf{bi-interpretable} if they are isomorphic in $\MLTT$.
\end{defin}

\begin{problem}
For any morphism $I \colon \L \rightarrow \L'$ in $\FOLDS_i$ to construct an interpretation $I^* \colon \LTT_\L \rightarrow \LTT_{\L'}$
\end{problem}
\begin{construction}\label{functortointerpretation}
We define $I^* \colon Sorts_\L \rightarrow Sorts_{\L'}$ by 
\begin{equation}
A(x_f)_{f \in \pcosieve{A}{\L}} \mapsto I(A) (x_{I(f)})_{f \in \pcosieve{A}{\L}}
\end{equation}
which is well-defined since $I$ is assumed full and faithful which means that $\pcosieve{A}{\L}$ is bijective to $\pcosieve{I(A)}{\L'}$. 
For the variables and contexts we define the obvious maps.
Lastly, to see that $I^*$ does indeed define an interpretation we need to show that for each rule $R$ of $\LTT_\L$ we have that $I^*(R)$ is admissible in $\LTT_{\L'}$ which is entirely straightforward and left to the reader.
\end{construction}

\begin{prop}
The assignment $\L \mapsto \LTT_{\L}$ and $I \mapsto I^*$ defines a functor 
\[
\LTT \colon \FOLDS_i \rightarrow \MLTT
\]
\end{prop}
\begin{proof}
Immediate from Construction \ref{functortointerpretation}.
\end{proof}

We can now make precise the sense in which the order on a FOLDS signature $\L$ is irrelevant to $\LTT_\L$.

\begin{cor}\label{isoimpliesbiint}
Let $I \colon \L \rightarrow \L'$ be an isomorphism of FOLDS signatures. Then $\LTT_\L$ and $\LTT_{\L'}$ are bi-interpretable.
\end{cor}

\def\existstoprule{(\exists\top)}

Over any FOLDS signature $\L$ we can now define a deductive system that gives us a notion of provability, and which is essentially an adaptation of a standard deductive system for $\FOLeq$. For added generality (i.e. to be able to work in fragments of FOLDS that do not contain all the logical connectives and quantifiers, namely the regular and coherent fragments) we define the deductive system as a sequent calculus.

\begin{defin}[$\DFOLDS$, $\DFOLDS^{\text{cl}}$]
We assume that every sequent that appears below is a well-formed $\L$-sequent and 
%Our presentation combines elements from \cite{Jacobs, MFOLDS, Elephant}. 
%As usual,
 the double lines denote a rule that goes in either direction.

%We present the ``old'' rules of the proof system $\D$ of Section \ref{nlogicProof} to which the ``new'' rules (Eq-intro), (J) and (R) are added.
%We fix a signature $\L$ and we assume that every formula that appears below is a well-formed $\L$-formula and that every sequent is well-formed. 
%Our presentation combines elements from \cite{Jacobs, MFOLDS, Elephant}. 
%As usual, the double lines denote a rule that goes in either direction.
%As usual we denote substitution by square brackets and for $\vec{x} \equiv x_1,...,x_n$ and $\vec{y} \equiv y_1,...,y_n$ contexts of variables of the same sort and length we write $\vec{x} = \vec{y}$ as an abbreviation for the conjunction of equalities $\bigwedge_{i=1}^n x_i=y_i$.

\vspace{0.3cm}

\noindent\textbf{Structural Rules}

\[
\inferrule
{
\!
}
{
\Gamma \v \phi \seqimplies \phi
}
\quad \text{(iden)}
\quad
\quad
\inferrule
{
\Delta \v \phi \seqimplies \psi
}
{
\Gamma \v \alpha (\phi) \seqimplies \alpha (\psi)
}
\quad \text{(Sub)}, \:\alpha \colon \Gamma \Rightarrow \Delta
\]
%where $s \colon \Gamma \Rightarrow \Delta$ is a context morphism.
\[
\inferrule
{
\Gamma \v \phi \seqimplies \psi \\ \Gamma \v \psi \seqimplies \chi
}
{
\Gamma \v \phi \seqimplies \chi
}
\quad \text{(Cut)}
\quad \quad
\inferrule
{
\Gamma \v \phi \seqimplies \psi
}
{
\Gamma, x \colon K \v \phi \seqimplies \psi
}
\quad \text{(Con-wk)}
\]

\[
\myinferrule
{
\Gamma, x \colon K, y \colon K', \Gamma' \v \phi \seqimplies \psi
}
{
\Gamma, y \colon K', x \colon K, \Gamma' \v \phi \seqimplies \psi
}
\text{(Con-exch)},  x \fresh K', y \fresh K
\]

%\subsection*{

\vspace{0.3cm}

\noindent\textbf{Logical Rules}

%%%TOP BOT%%%
\[
\inferrule
{
\:
}
{
\Gamma \v \phi \seqimplies \top
}
\quad (\top)
\quad \quad 
\inferrule
{
\:
}
{
\Gamma \v \bot \seqimplies \phi
}
\quad (\bot)
\]
\[
\inferrule
{
\:
}
{
\Gamma \v \exists x \colon K. \phi \seqimplies \exists x \colon K. \top
}
\existstoprule
\]

\[
 (\wedge) \myinferrule{\Gamma \v \theta \seqimplies \phi \wedge \psi}{\Gamma \v \theta \seqimplies \phi \: \: \: \: \Gamma \v \theta \seqimplies \psi}
 \quad \quad
  (\vee) \myinferrule{\Gamma \v \phi \vee \psi \seqimplies \theta}{\Gamma \v \phi \seqimplies \theta \: \: \: \: \Gamma \v \psi \seqimplies \theta}
\]
 
%\[
%\inferrule
%{
%\phipsi \\ \phi \seqimplies_{\vec{x}} \chi
%}
%{
%\phi \seqimplies_{\vec{x}} \psi \wedge \chi
%}
%\quad \text{$\wedge$-intro}
%\quad \quad 
%\inferrule
%{
%\phi \seqimplies_{\vec{x}} \chi \\ \psichi
%}
%{
%\phi \vee \psi \seqimplies_{\vec{x}} \chi
%}
%\quad \text{$\vee$-intro}
%\]
\[
 (\rightarrow) \myinferrule{\Gamma \v \theta \seqimplies \phi \rightarrow \psi}{\Gamma \v \theta \wedge \phi \seqimplies \phi}
\]

%\[
%\inferrule
%{
%\phi \seqimplies_{\vec{x},y} \psi
%}
%{
%\exists y \phi \seqimplies_{\vec{x}} \psi
%}
%\quad \text{$\exists$-intro}
%\quad \quad 
%\inferrule
%{
%\exists y \phi \seqimplies_{\vec{x}} \psi
%}
%{
%\phi \seqimplies_{\vec{x},y} \psi
%}
%\quad \text{$\exists$-elim}
%\]

\[
\myinferrule{\Gamma \v \theta \seqimplies \forall x \colon K. \phi}{\Gamma, x \colon K \v \theta \seqimplies \phi} (\forall) 
\quad \quad
\myinferrule{\Gamma \v \exists x \colon K. \phi \seqimplies \theta}{\Gamma, x \colon K \v \phi \seqimplies \theta} (\exists) 
\]
%
%\[
%\inferrule
%{
%\phi \seqimplies_{\vec{x},y} \psi
%}
%{
%\phi \seqimplies_{\vec{x}}\forall y \psi
%}
%\quad \text{$\forall$-intro}
%\quad \quad
%\inferrule
%{
%\phi \seqimplies_{\vec{x}} \forall y \psi
%}
%{
%\phi \seqimplies_{\vec{x},y} \psi
%}
%\quad \text{$\forall$-elim}
%\]

%\[
%\inferrule
%{
%\phi \wedge \psi \seqimplies_{\vec{x}} \chi
%}
%{
%\psi \seqimplies_{\vec{x}} \phi \rightarrow \chi
%}
%\quad \text{$\rightarrow$-intro}
%\quad \quad
%\inferrule
%{
%\psi \seqimplies_{\vec{x}} \phi \rightarrow \chi
%}
%{
%\phi \wedge \psi \seqimplies_{\vec{x}} \chi
%}
%\quad \text{$\rightarrow$-elim}
%\]

If we are working in the regular or coherent fragment then we also add the following rule, which is otherwise derivable:

\[
\inferrule
{
\!
}
{
\Gamma \v \phi \wedge (\exists x \colon K. \psi) \seqimplies \exists x \colon K. (\phi \wedge \psi)
}
\quad \text{(Frob)}, x \notin \var(\phi)
\]

If we are working in the coherent fragment then we also add following rule, which is otherwise derivable:

\[
\inferrule
{
\!
}
{
\Gamma \v \phi \wedge (\psi \vee  \chi) \seqimplies (\phi \wedge \psi) \vee (\phi \wedge \chi)
}
\quad \text{(Dist)}
\]
Finally, to get $\DFOLDS^{\text{cl}}$ we add the law of the excluded middle:
\[
\inferrule
{
\!
}
{
\Gamma \v \top \seqimplies \phi \vee (\phi \rightarrow \bot)
}
\quad \text{(LEM)}
\]
\end{defin}

\begin{remark}
The perhaps unfamiliar rule $\existstoprule$ is needed due to the lack of ``atomic relation symbols'' in our syntax, the role of which is played by statements of the form $\exists x \colon K. \top$ for some sort $K$. The $\existstoprule$ allows us to extract from a witness that a certain proposition holds the witness for that proposition and can therefore be thought of as encoding, semantically, the ``first projection'' of a term of a $\Sigma$-type.
\end{remark}

\def\FOLDSentails{\vdash}
\def\FOLDSentailscl{\vdash_{\text{cl}}}
\def\FOLDSentailsboth{\vdash_{(\text{cl})}}

\begin{defin}[FOLDS theory]\label{FOLDStheory}
A \textbf{FOLDS theory $\T$ over a signature $\L$} (or \textbf{FOLDS $\L$-theory}) is a set of $\L$-sequents. 
\end{defin}

\begin{defin}[Entailment]\label{entailment}
We say that a FOLDS $\L$-theory $\T$ \textbf{entails} (resp. \textbf{classically entails}) an $\L$-sequent $\tau$ if there is a derivation in $\DFOLDS$ (resp. $\DFOLDScl$) of $\tau$ from a finite subset of the sequents in $\T$. When that is the case we write $\T \FOLDSentails \tau$ (resp. $\T \FOLDSentailscl \tau$).
\end{defin}

\begin{notation}
Whenever what we want to say refers to both entailment and classical entailment then we will use the notation $\FOLDSentailsboth$.
\end{notation}

To conclude this section we present the (first-order) FOLDS axiomatization of category theory as an illustrative example that combines all the above-introduced notions.
%, and which we will then build on for our results in the final Section.

\begin{exam}\label{Tcat}
Let $\Lcat$ denote the following  FOLDS signature
\[
\xymatrix{
\circ \ar@/^5pt/[rd]^{t_0} \ar[rd]_{t_1} \ar@/_20pt/[rd]_{t_2} & I \ar[d]^{i} & =_A \ar@/_/[ld]_{s} \ar@/^/[ld]^{t} \\
 & A \ar@/_/[d]_{d} \ar@/^/[d]^{c} \\
 & O \\
}
\]
subject to the relations
\[
dt_0=dt_2, ct_1=c1t_2, dt_1=ct_0
\]
\[
di=ci
\]
\[
ds=dt, cs=ct
\]
%and with $d(O)=\infty$, $d(A) = 0$ and $d(I)=d(\circ)=-1$.
The (first-order) $\Lcat$-\emph{theory of categories} $\mathbb{T}_{\text{cat}}$ consists of the following axioms:
%in $\Lcat$ has the following $\Lcat$ sentences as axioms: 
\begin{enumerate}[(1)]
\item(Existence of identities) 

$\forall x \colon O. \exists i \colon A(x,x). \exists \sigma \colon I(i,x,x). \top$ 
%\item $\forall x,y \colon O \forall f \colon A(x,y) \exists \epsilon \colon =_A(f,f,x,y).\top$
%\item $\forall x,y \colon O \forall f,g \colon A(x,y) \forall \epsilon \colon =_A(f,g,x,y) \exists \epsilon_2 \colon =_A(g,f,x,y).\top$
%\item $\forall x,y \colon O \forall f,g,h \colon A(x,y) \forall \epsilon \colon =_A(f,g,x,y) \forall \epsilon_2 \colon =_A(g,h,x,y) \exists \epsilon_3 \colon =_A(f,h,x,y).\top$
\item(Functionality of composition-1)

$\forall x,y,z \colon O. \forall f \colon A(x,y). \forall g \colon A(y,z). \exists h \colon A(x,z). \exists \tau \colon \circ(f,g,h,x,y,z). \top$ 
\item (Functionality of Composition-2)
\begin{align*}
\forall x,y,z \colon O. &\forall f \colon A(x,y). \forall g \colon A (y,z). \forall h,h' \colon A(x,z). 
\forall \tau_1 \colon \circ (f,g,h). \forall \tau_2 \colon \circ (f,g,h'). \\ &\exists \epsilon \colon =_A(h,h',x,z)
 %&((\exists \tau_1 \colon \circ (f,g,h,x,y,z). \exists \tau_2 \colon \circ (f,g,h',x,y,z).\top) \rightarrow h=h')
\end{align*}
\item(Associativity)
\[
\begin{split}
&\forall x,y,z,w \colon O. \forall f \colon A(x,y). \forall g \colon A(y,z). \forall h \colon A(z,w). \forall i \colon A(x,z). \forall j \colon A(x,w). \\ &\forall k \colon A(y,w). \forall \tau_1 \colon \circ(f,g,i,x,y,z). \forall \tau_2 \colon \circ(i,h,j,x,z,w). \forall \tau_3 \colon \circ (g,h,k,y,z,w). \\ &\exists \tau_4 \colon \circ (f,k,j,x,y,w). \top
\end{split}
\]
\item(Uniqueness of identity)

$\forall x \colon O. \forall i,j \colon A(x,x). \forall \sigma_1 \colon I(i,x,x). \forall \sigma_2 \colon I(j,x,x). \exists \epsilon \colon =_A(i,j,x,x). \top$ 
%\item $\forall x \colon O \forall i,j \colon A(x,x) \forall \phi \colon I(i,x,x) \forall \epsilon \colon =_A(i,j,x,x) \exists \psi \colon I(j,x,x). \top$
\item (Right unit)

$\forall x,y \colon O. \forall i \colon A(x,x). \forall g \colon A(x,y). \forall \sigma \colon I(i,x,x). \exists \tau \colon \circ (i,g,g,x,x,y). \top$ 
\item (Left unit) 

$\forall x,y \colon O. \forall i \colon A(y,y). \forall f \colon A(x,y). \forall \phi \colon I(i,y,y). \exists \tau \colon \circ (f,i,f,x,y,y). \top$ 
%\item 
%\[
%\begin{split}
%&\forall x,y,z \colon O \forall f,f' \colon A(x,y) \forall \epsilon_1 \colon =_A(f,f',x,y) \forall g, g' \colon A(y,z) \forall \epsilon_2 \colon =_A(g,g',x,y) \\ &\forall h, h' \colon A (x,z) \forall \tau_1 \colon \circ (f,g,h,x,y,z) \forall \tau_2 \colon \circ (f', g' ,h', x,y,z) \exists \epsilon_3 \colon =_A(h,h',x,z)
%\end{split}
%\]
%\item \[
%\begin{split}
%&\forall x,y,z \colon O \forall f,f' \colon A(x,y) \forall \epsilon_1 \colon =_A(f,f',x,y) \forall g, g' \colon A(y,z) \forall \epsilon_2 \colon =_A(g,g',x,y) \\ &\forall h, h' \colon A (x,z) \forall \tau_1 \colon \circ (f,g,h,x,y,z)  \forall \epsilon_3 \colon =_A(h,h',x,z) \exists \tau_2 \colon \circ (f', g' ,h', x,y,z)
%\end{split}
%\]
\end{enumerate}
\end{exam}

\begin{exam}\label{Tgpd}
To obtain the (first-order) $\Lcat$-\emph{theory of groupoids} $\Tgpd$ we add to the axioms of $\Tcat$ the following axiom
\begin{enumerate}[(1)]
\setcounter{enumi}{7}
\item (Every arrow is an isomorphism)

 $\forall x,y \colon K. \forall p \colon x \cong y. \exists q \colon y \cong x. \exists u \colon x \cong x. \exists v \colon y \cong y. I(u) \wedge I(v) \wedge \circ(p,q,u) \wedge \circ (q,p, v)$
%\item $
\end{enumerate}
where we have now used the sugared form for the composition and identity relations.
\end{exam}

\section{Semantics of FOLDS in MLTT}\label{FOLDSinMLTT}

We now describe a semantics of FOLDS in MLTT, constructed as a direct interpretation of the syntax of $\LTT_\L$ into MLTT, and prove that under this interpretation the rules of $\DFOLDS$ (resp. $\DFOLDScl$) are sound in MLTT (resp. classical MLTT, which we understand as MLTT together with (an inhabitant of the type representing) LEM).
For the purposes of this section by MLTT we will understand type theory with $\Pi, \Sigma, \onetype$, a universe a-la-Tarski $\U$ and (depending on the fragment of first-order logic that we want to work with) $+, \zerotype$.

We first define for any FOLDS signature $\L$ a notion of an $\L$-structure. We do this by describing a type expression $\Struc{\L}$ (Definition \ref{FOLDSLstructure}), proving that this type expression is a well-formed type (Theorem \ref{strucwellformed}) and taking its terms to be our $\L$-structures.
We then define an interpretation of $\LTT_\L$ into MLTT which gives us (for any $\L$-structure $\M$) notions of $\M$-contexts, formulas, sequents etc. (Definition \ref{interpretationFOLDS}) and prove that this interpretation is correct (Theorem \ref{FOLDSsoundness}).
In particular, any sequent $\sigma$ is interpreted as a type $\sigma^\M$  in type theory following the usual recipe of the proposition-as-types interpretation (interpreting $\wedge$ as $\times$, $\forall$ as $\Pi$ etc.) and then we say that $\M$ satisfies $\sigma$ if $\sigma^\M$ is inhabited (Definition \ref{FOLDSsatisfaction}). Finally, we prove soundness of $\DFOLDS$ with respect to the above-defined semantics in MLTT (Theorem \ref{soundness}).
%The second step is to define for any $\L$-sequent $\sigma$ and $\L$-structure $\M$ a notion of when $\M$ satisfies $\sigma$.
%For the purposes of this section by MLTT we will understand type theory with $\Pi, \Sigma, \onetype$, a universe a-la-Tarski $\U$, and (depending on the fragment of first-order logic that we want to work with) $+, \zerotype$.

\begin{notation}
We write $\MLTTcontexts$ for the context expressions (``pre-contexts''), $\MLTTterms$ for the term expresssions (``pre-terms''), $\MLTTtypes$ for the type expressions (``pre-types'') and $\MLTTjudgments$ for the sequent expressions (``pre-sequents'') of MLTT.
We will use the notation $\equiv$ for definitional equality when it is derived and the symbol $\eqdef$, as usual, when we are defining a term by postulating a definitional equality.
We write $\TTapp(x,f)$ for function application and $\TTproj(x)$ for the projections in $\Sigma$-types, and we write $\Elements$ for the reflector of the universe $\U$.
We write $*$ for the term of $\onetype$.
We will write $\vert \vert A \vert \vert$ for the propositional truncation of a type $A$ in $\mathcal{U}$.
%For any $m \in \Natinfty$ we will write $\typesofhlevel{m}$ for the types of $h$-level $m$ in $\U$ with the convention that $\typesofhlevel{\infty} \eqdef \U$.
%We will also use the more recognizable notation $\PropU, \SetU$ and $\GpdU$ for types of $h$-level $1,2$ and $3$ respectively.
%In general, we will abuse notation and conflate a term $A \colon \typesofhlevel{m}$ with its underlying type (i.e. with its first projection).
%\textcolor{red}{For any type $A$ we will write $\mathtt{Id}^i_A$ for the $i^\text{th}$-iterated identity type. Thus, for example, $\mathtt{Id}^3_A(\alpha,\beta)(p,q)(x,y)$ stands for
%\[
%\mathtt{Id}_{\mathtt{Id}_{\mathtt{Id}(x,y)} (p,q)} (\alpha,\beta)
%\]
%}
We will usually write $\times$ for  non-dependent sums 
($\Sigmatype$-types) 
%using the Agda-style notation $(x \colon A) \times (b \colon B(x))$ where appropriate 
and $\rightarrow$ for non-dependent functions ($\Pitype$-types).
We will occasionally also use $\times$ and $\rightarrow$ in the Agda-style notation $(x \colon A)\rightarrow B$ and $(x \colon A) \times B$ for dependent functions and sums, whenever convenient.
%again using Agda-style notation $(x \colon A) \rightarrow B$ where appropriate.
For a given (ordered) set $S$ (almost exclusively of objects and arrows of a given inverse category) we will use the notation $\underset{s \in S}{\Sigmatype}$ and $\underset{s \in S}{\Pitype}$ to denote the sums and functions over all the expressions indexed in some way by $S$ (this is an ``external'' description of a type expression in MLTT which a priori may not be well-formed).
%We will use the notation ``$A \:\:(\Gamma)$'' where $A$ is a type expression and $\Gamma$ a context expression to indicate that variables in $\Gamma$ (may) appear in $A$.
%We will generally conflate a term judgment $A \colon \U$ with a typing judgment $A\:\textbf{Type}$ so as to be able to refer freely to contexts formed from types in a universe, e.g. given $A \colon \U$ we will write $x \colon A$ for the context that would more appropriately be written as $x \colon El (A)$ and similarly for any type in $\U$ that might depend on $A$.
%In all other notational matters related to type theory we will follow the notation of \cite{HTT} closely.
%Let $\L$ be an $h$-signature.
We will assume that every expression that ranges over the objects and morphisms of $\L$ is ordered according to the given proper order of $\L$.
%(in the sense of Definition \ref{logicalorder}). 
For example, if $S=\lbrace K_1, \dots, K_n \rbrace$ is a set of objects of $\L$ such that $K_1 < \dots < K_n$ according to the proper order on $\L$ and we write
\[
\underset{K \in S}{\Pitype} E(K)
\]
(for some expression $E(K)$ in which $K$ may appear) this notation will denote the expression
$E(K_1) \rightarrow \dots \rightarrow E(K_n)$ and similarly for $\Sigmatype$-types.
This notational device will help us declutter notation significantly, especially in Definitions \ref{FOLDSLstructure} and \ref{HLstructure}.
Whenever $S$ is empty then we stipulate that 
\begin{align*}
\underset{K \in S}{\Pitype} E &= \onetype \rightarrow E
\\
\underset{K \in S}{\Sigmatype} E &= \zerotype
\end{align*}
%In particular, we will write a context $\Gamma$ of $\L$ as a list of variables and sorts ordered by the given proper order. 
For any context $\Gamma$ and type $A$ in context $\Gamma$ we write $\Sigmatype \: (\Gamma) \: A$ (resp. $\Pitype \: (\Gamma) \: A$) for the iterated (closed) $\Sigmatype$-type (resp. $\Pitype$-type) which has bound everything in context $\Gamma$. We will also sometimes write simply $\mathbf{x} \colon \Gamma$ for a list of variables in a given context, e.g. as in $\underset{\mathbf{x} \colon \Gamma}{\Pitype} A \TTisatype$.

We also fix several more conventions for the interaction for the syntax of a FOLDS signature $\L$ and the syntax of MLTT.
We will allow ourselves to denote variables in type theory by morphisms of $\L$.
For example, in $\Lrg$, we may describe the interpretation of $A$ as a type in context $c \colon O, d \colon O$ (cf. Example \ref{Lrghstruc} below for a more precise illustration of how this works).
For $K \in \L$ and $\mathbf{p}=(p_1,\dots,p_m)$ a list of arrows in $\L$ the notation $\TTapp[K(\mathbf{p})]$ in type theory will now be understood as the term expression 
\[
\ttapp (p_m, \ttapp(p_{m-1}, \dots \ttapp (p_1, K) \dots))
\]
%where we leave the type of each $p_i$ implicit (but it will be deducible from the structure of $\L$ whenever we employ) 
with the understanding that 
\[
\TTapp[K\:\epsilon] = \TTapp(*,K)
\]
For example, in $\Lrg$, the notation $\TTapp[I(x,f)]$ stands for the term expression
\[
\ttapp (f, \ttapp(x, I))
\]
which of course will make sense if $I \colon (x \colon O) \rightarrow (f \colon A (x,x)) \rightarrow \U$ and $A(x,x) \eqdef \ttapp(x, (x,A))$ for $A \colon O \rightarrow O \rightarrow \U$, as indeed will be the case when the notation is used.
%For a given $K \in \ob{\L}$ and (pre-)context $\Gamma = (g \colon K_g (\mathbf{p}g))_{g \in K \downarrow \L}$ in type theory and for any $f \colon A \rightarrow K$ we will write $\Gamma f$ for the context obtained ``by $f$ acting on the right'' i.e. for the (pre-)context $(h \colon K_{h} (\mathbf{p}h))_{h \in \lbrace gf \:\vert\: g \in \pcosieve{K}{\L} \rbrace}$, which one can think of as the ``cosieve induced by $f$''. (In Lemma \ref{strucwellformedlemma0} below we prove that $\Gamma f$ is indeed well-formed in an appropriate sense if $\Gamma$ is.)
%where each distinct $fg$ appears only once (i.e. in case $h=f_ig=f_jg$ for $i \neq j$ then $h$ will appear only once in $\Gamma_Kg$.
\end{notation}

\begin{terminology}
We say that a type (expression) $A$ or context (expression) $\Gamma$ etc. is \emph{well-formed} (in type theory) if the appropriate judgment ($A \TTisatype$ or $\Gamma \isacontext$)  is derivable (in type theory). A \emph{closed} type is a well-formed type in the empty context.
We will generally refer to the relevant rules of MLTT (and, below, of HoTT) through the relevant type formers, e.g. $\Sigmatype$-\emph{formation}, $\Pitype$-\emph{elimination} etc.
%Finally we will use the HoTT terminology and call a type $A$ a \emph{mere proposition} if its identity types are contractible in the sense that we have an inhabitant of the type
%\[
%\underset{x, y \colon A}{\Pitype} \:\: \underset{q \colon \Idtype_A (x,y)}{\Sigmatype} \:\: \underset{p \colon \Idtype_A (x,y)}{\Pitype} \:\: \Idtype_{\Idtype_A(x,y)} (p,q)
%\]
%With this in mind what we will refer to as the \emph{universal property of propositional truncation} means that given a term $\eta \colon A \rightarrow P$ where $P$ is a mere proposition we obtain a term 
%\[
%\truncated{\eta} \colon \truncated{A} \rightarrow P
%\]
%and leave the details of the particular implementation of the propositional truncation implicit. 
%We will freely refer to results in \cite{HTT} 
%We say that two (well-formed) contexts are \emph{isomorphic} if they have the same length and are $\alpha$-equivalent variants, i.e. differ only in the naming of their declared variables. 
%We say that a context $\Gamma$ is a \emph{subcontext} of $\Delta$ if $\Gamma$ is an initial segment of $\Delta$.
\end{terminology}

%We now fix a finite FOLDS signature $\L$.
% together with a proper order $<$ on it.

\begin{defin}\label{FOLDSLstructure}
Let $\L$ be a finite FOLDS signature. The type of \textbf{$\L$-structures} is given by the type expression
\begin{equation}
\Struc{\L} \eqdef \underset{K \in \ob{\L}}{\Sigmatype} (K \colon T_K)
\end{equation}
where
\begin{equation}
T_K \eqdef \underset{f \in \pcosieve{K}{\L}}{\Pitype} (f \colon \TTElapp{K_f (pf)_{p \in \pcosieve{K_f}{\L}}}) \:\: \U 
\end{equation}
%and where the symbol $\inttype{K_f} (\mathbf{p}f)$ is defined as follows:
%\begin{itemize}
%%\item $\cod{f}$, if $l(K_f)=0$
%\item $K_f(\mathbf{p}f)$, if $K_f \in \NL{\L}$
%\item $\Idtype_{A_{sf} (\mathbf{p}sf)} (sf, tf)$, if $K_f \equiv \eqsort_{A}$ for some $A \in \L$
%\item $\Idtype_{\Idtype_{A_{s\rmor f}(\mathbf{p}s\rmor f})} (\rmor f, \refl_{srf})$, if $K_f \equiv \rsort_A$ for some $A \in \L$
%\item $\Idtype_{A(\mathbf{q}\tmorXX{h}f)} (\tmorXX{h}f, \trans_{e_hf}^{A[h]}(h_1f))$, if $K_f \equiv \tau_h$ for some $h \colon A \rightarrow K$ in $\L$ and where $A[h]$ denotes the expression
%$
%\lambda h \colon K. A(h,\mathbf{g}h_1f)
%$.
%% $\lambda y \colon K. A(k)[y/f] \colon \U$.
%\end{itemize}
\end{defin}

\begin{prop}\label{isoLequalstrucL}
If $\L$ and $\L'$ are isomorphic FOLDS signatures then $\Struc{\L} \equiv \Struc{\L'}$.
\end{prop}
\begin{proof}
If $\L$ and $\L'$ are isomorphic then the type expressions $\Struc{\L}$ and $\Struc{\L'}$ are $\alpha$-equivalent, and therefore denote the same type expression.
\end{proof}

Before proving that $\Struc{\L}$ is indeed a well-formed type we give several examples that illustrate the idea.

\begin{exam}\label{Lrghstruc}
For $\Lrg$ we have
\begin{align*}
\ttype{O}{} &= \onetype \rightarrow \U \quad \quad (\text{since $\pcosieve{O}{\Lrg} = \varnothing$})\\
\ttype{A}{} &= \underset{f \in \lbrace c,d \rbrace}{\Pitype} (f \colon \TTElapp{A_f (pf)_{p \in \pcosieve{A_f}{\Lrg}}}) \:\: \U\\
&= \Pitype (c \colon \TTElapp{A_c (pc)_{p \in \pcosieve{A_c}{\Lrg}}}) (d \colon \TTElapp{A_d (pd)_{p \in \pcosieve{A_d}{\Lrg}}}) \:\: \U \\
&= \Pitype (c \colon \Elements (\TTapp (*,O))) (d \colon \Elements (\TTapp (*,O))) \:\: \U \\
%(d \colon A_d)\rightarrow (c \colon A_c) \rightarrow \U = (d \colon O)\rightarrow(c \colon O) \rightarrow \typesofhlevel{2}  
&\equiv \Elements (\TTapp (*,O)) \rightarrow \Elements (\TTapp (*,O)) \rightarrow \U \\
\ttype{I}{} &= \underset{f \in \lbrace di,i \rbrace}{\Pitype} (f \colon \TTElapp{A_f (pf)_{p \in \pcosieve{A_f}{\Lrg}}}) \:\: \U\\
&= \Pitype (di \colon \Elements(\TTapp (*,I_{di})) (i \colon \TTapp{I_i (pi)_{p \in \pcosieve{I_i}{\Lrg}}}) \:\: \U \\
&= \Pitype (di \colon \Elements(\TTapp (*,O)) (i \colon \Elements( \TTapp (di, \TTapp (ci,A))) \:\: \U \\
%(di \colon I_{di}) \rightarrow (i \colon I_i (di,ci)) \rightarrow \typesofhlevel{1} 
&\equiv \Pitype (x \colon \Elements(\TTapp (*,O)) (f \colon \Elements( \TTapp (x, \TTapp (x,A))) \:\: \U \\
&\equiv \Pitype (x \colon \Elements(\TTapp (*,O)) \:\: \Elements( \TTapp (x, \TTapp (x,A))) \rightarrow \U
\end{align*}
Note that if $O \colon T_O$ then $T_A$ is well-formed and that if $A \colon T_A$ then $T_I$ is well-formed. 
Furthermore, note that in the last two lines of the unpacking of definition of $T_I$ we have given an $\alpha$-equivalent reformulation, where the repetition of the $x$ is due to the fact that $di=ci$.
If we abuse notation and write simply $O$ for $\Elements(\TTapp (*,O))$ and simply $A(x,x)$ for $\Elements( \TTapp (x, \TTapp (x,A)))$ then $T_I$ becomes the more recognizable type
\[
\underset{x \colon O}{\Pitype} A(x,x) \rightarrow \U
\]
With that in mind we get
\begin{align*}
\Struc{\Lrg} &= (O \colon T_O) \times (A \colon T_A) \times (I \colon T_I) \\
		  &= (O \colon \U) \times (A \colon O \rightarrow O \rightarrow \U) \times (I \colon \underset{x \colon O}{\Pitype} A(x,x) \rightarrow \U)		
\end{align*}
which is the data type we should expect from $\Lrg$, i.e. the type of ``reflexive graphs''.
\end{exam}

\begin{exam}
We do an example that is a little more arbitrary. Consider the following FOLDS signature $\L$
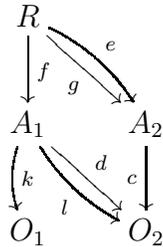
\[
\xymatrix{
R \ar[d]^{f} \ar[dr]_{g} \ar@/^/[dr]^{e} \\
A_1 \ar@/_/[d]^{k} \ar[dr]^{d} \ar@/_/[dr]_{l} &A_2  \ar[d]_{c} \\
O_1  &O_2 
}
\]
 subject to the relation $df=cg$.
Then using the obvious notational abbreviations as in the previous example we have:
%, up to the obvious equivalences:
\begin{align*}
\ttype{O_1}{} &\equiv \U \\
\ttype{O_2}{} &\equiv \U \\
\ttype{A_1}{} 
%= \ttype{A_1}{0} = \times \setdefinition{(x_f \colon K_f (x_{pf})_p)}{\dom{f}=A_1, \cod{f} \neq A_1} \rightarrow \U \\ &= \times \lbrace (x_k \colon O_1),(x_l \colon O_2),(x_d \colon O_2) \rbrace \rightarrow \U \\ &
&\equiv (k \colon O_1) \rightarrow (l \colon O_2) \rightarrow (d \colon O_2) \rightarrow \U \equiv O_1 \rightarrow O_2 \rightarrow O_2 \rightarrow \U \\
\ttype{A_2}{} 
%&= \ttype{A_2}{0} = \setdefinition{(x_f \colon K_f (x_{pf})_p)}{\dom{f}=A_2, \cod{f} \neq A_2} \rightarrow \U \\
		 &\equiv (c \colon O_2) \rightarrow \U \equiv O_2 \rightarrow \U\\
\ttype{R}{} 
%&= \ttype{R}{0} \rightarrow \ttype{R}{1} \rightarrow \U \\
%	      &= \big[ \times \setdefinition{(x_f \colon K_f (x_{pf})_p)}{\dom{f}=R, l(K_f)=0} \big] \\
%	      &\:\:\:\:\:\:\:\: \rightarrow \big[ \times \setdefinition{(x_f \colon K_f (x_{pf})_p)}{\dom{f}=R, l(K_f)=1} \big] \rightarrow \U \\
	        &\equiv (df \colon O_1) \rightarrow (lf \colon O_2) \rightarrow (df \colon O_2) \rightarrow (ce \colon O_2) \\
	        &\:\:\:\:\:\:\:\: \rightarrow (f \colon A_1 (kf,lf,df)) \rightarrow (g \colon A_2 (df)) \rightarrow (e \colon A_2 (ce)) \rightarrow \U \\
	&\equiv (x \colon O_1) \times (y \colon O_2) \times (z \colon O_2) \times (w \colon O_2) \\
	&\:\:\:\:\:\:\:\: \rightarrow (f \colon A_1 (x,y,z)) \times (g \colon A_2 (z)) \times (e \colon A_2 (w)) \rightarrow \U
\end{align*}
Thus we get
\begin{align*}
\Struc{\L} = &(O_1 \colon \U) \times (O_2 \colon \U) \times  \\
		  &(A_1 \colon O_1 \rightarrow O_2 \rightarrow O_2 \rightarrow \U) \times (A_2 \colon O_2 \rightarrow \U) \times \\
		  &(R \colon (x \colon O_1) \times (y \colon O_2) \times (z \colon O_2) \times (w \colon O_2) \\
	&\:\:\:\:\:\:\:\: \rightarrow (x_f \colon A_1 (x,y,z)) \times (x_g \colon A_2 (z)) \times (x_e \colon A_2 (w)) \rightarrow \U )
\end{align*}
which is the data type we should expect from $\L$.
\end{exam}

\begin{exam}
In $\Lcat$ we have $T_O, T_A, T_I$ just as in $\Lrg$. Then we have, using the usual abbreviations:
\begin{align*}
\ttype{=_A}{} &= \underset{f \in \lbrace ds, cs, s,t \rbrace}{\Pitype} (f \colon \TTElapp{(=_A)_f (pf)_{p \in \pcosieve{(=_A)_f}{\Lcat}}}) \:\: \U  \\ 
&= \Pitype (ds, cs \colon O) (s \colon \Elements (\TTapp (cs, \TTapp (ds, A)))) (t \colon \Elements (\TTapp (ct, \TTapp (dt, A)))) \:\: \U \\
&\equiv \underset{x,y \colon O}{\Pitype} A(x,y) \rightarrow A(x,y) \rightarrow \U \\ 
\ttype{\circ}{} &=  \underset{f \in \lbrace dt_0, ct_0, ct_2, t_0,t_1,t_2 \rbrace}{\Pitype} (f \colon \TTElapp{\circ_f (pf)_{p \in \pcosieve{\circ_f}{\Lcat}}}) \:\: \U \\
&= \Pitype (dt_0, dt_1, ct_2 \colon O) (t_0 \colon \Elements (\TTapp (ct_0, \TTapp (dt_0, A))))  \\
&\quad \quad \:\: (t_1 \colon \Elements (\TTapp (ct_1, \TTapp (dt_1, A)))) (t_2 \colon \Elements (\TTapp (ct_2, \TTapp (dt_2, A)))) \:\: \U \\
&\equiv \underset{x,y,z \colon O}{\Pitype} A(x,y) \rightarrow A(y,z) \rightarrow A(x,z) \rightarrow \U \\
\end{align*}
As we should expect $T_{=_A}$ is the type of a binary relation on arrows with the same domain and codomain and $T_{\circ}$ is the type of a tertiary relation on ``commutative'' triangles.
Thus we get
\begin{align*}
\Struc{\Lcat} = &(O \colon \U) \times (A \colon O \rightarrow O \rightarrow \U) \times (I \colon (x \colon O) \rightarrow A(x,x) \rightarrow \U) \times \\ 
&(\circ \colon (x,y,z \colon O) \rightarrow A(x,y) \rightarrow A(y,z) \rightarrow A(x,z) \rightarrow \U)
\end{align*}
which, as we shall make more precise in Section \ref{nlogExamples}, is the ``relational'' form of the signature for (pre)category theory.
\end{exam}

%\begin{lemma}\label{strucwellformedlemma0}
%If $\Gamma = (g \colon K_g (\mathbf{p}g))_{g \in \pcosieve{K}{\L}}$ is a well-formed context then $\Gamma f$ is isomorphic to a subcontext of $\Gamma$ for any $f \colon A \rightarrow K$ in $\L$.
%\end{lemma}
%\begin{proof}
%At worst $f$ can introduce relations $p_1gf=p_2gf$ which would mean that $\Gamma f$ contracts $\Gamma$ in the obvious way. We leave the details to the reader.
%\end{proof}

\begin{lemma}\label{strucwellformedlemma}
Let $\L$ be a FOLDS signature. For all $K \in \ob{\L}$ the following context is well-formed
\[
\Gamma_K \eqdef (A \colon T_A)_{A \leq K}
\]
\end{lemma}
\begin{proof}
Firstly note that by the well-foundedness of (the underlying inverse category of) $\L$ for any $K \in \ob{\L}$ there will be only finitely many $A \leq K$ and so $\Gamma_K$ will be a finite list.
We now proceed by $<$-induction on $K$. If $K$ is $<$-minimal then $\Gamma_K \equiv K \colon T_K$ and by definition $T_K \equiv \onetype \rightarrow \U$. But since we assume that our type theory contains $\onetype, \Pitype$ and $\U$ we have
\begin{equation}
\varnothing \vdash \onetype \rightarrow \U \TTisatype
\end{equation}
which implies that $K \colon \onetype \rightarrow \U \equiv \Gamma_K$ is well-formed, as required.
Now assume that $\Gamma_A$ is well-formed for all $A <K$. To show that $\Gamma_K$ is well-formed it suffices to show that $\Gamma_A \vdash T_K \TTisatype$ where $A$ is the immediately preceding object to $K$ (according to $<$). Since $T_K$ is the following $\Pi$-type
\begin{equation}
\underset{f \in \pcosieve{K}{\L}}{\Pitype} (f \colon \TTElapp{K_f (pf)_{p \in \pcosieve{K}{\L}}}) \: \U 
\end{equation}
it suffices to show that 
\begin{equation}
\Gamma_A, (f \colon \TTElapp{K_f(\mathbf{p}f)})_{f \in \pcosieve{K}{\L}} \vdash \U \TTisatype
\end{equation}
for which it suffices to show that the context
\begin{equation}\label{target}
\Gamma_A, (f \colon \TTElapp{K_f(\mathbf{p}f)})_{f \in \pcosieve{K}{\L}}
\end{equation}
is well-formed. We proceed by $<$-induction on $f$ (by which we mean that starting with $\Gamma_A$ we will derive the context (\ref{target})). If $f$ is $<$-minimal then $\level{K_f} = 0$ which means that 
\begin{equation}
\TTElapp{K_f(\mathbf{p}f)} \equiv \Elements (\TTapp (*,K_f))
\end{equation}
But since $l(K_f)=0$ and $K_f < A$, then $K_f$ appears in $\Gamma_A$ and its type is $\onetype \rightarrow \U$. Hence we can derive
\[
\Gamma_A \vdash \Elements (\TTapp (*,K_f)) \TTisatype
\]
which implies that $\Gamma_A, (f \colon K_f)$ is a well-formed context.
Now assume that we know that for some $f \in \pcosieve{K}{\L}$ the following context
\begin{equation}
\Delta \eqdef \Gamma_A, (g \colon \TTElapp{K_g(\mathbf{p}g)})_{g<f}
\end{equation}
is well formed. We need to show that
\begin{equation}\label{deltants}
\Delta \vdash \TTElapp{K_f(\mathbf{q}f)} \TTisatype
\end{equation}
is derivable. Since $K_f < K$ we know that $K_f$ appears in $\Gamma_A$ as
\begin{equation}\label{Kftype}
K_f \colon \underset{h \in \pcosieve{K_f}{\L}}{\Pitype} (h \colon \TTElapp{K_h (\mathbf{q}h)}) \:\: \U
\end{equation}
and also for each $q \in \pcosieve{K_f}{\L}$ we have $qf<f$ and hence 
$
(qf \colon \TTElapp{K_{qf}(\mathbf{p}qf))} \in \Delta
$
which means that for each $q \in \pcosieve{K_f}{\L}$ we can derive
\begin{equation}\label{qfvar}
\Delta \vdash qf \colon \TTElapp{K_{qf}(\mathbf{p}qf)}
\end{equation}
where, crucially, some of these judgments may be declaring the same variable since it may happen that $qf=q'f$ even if $q \neq q'$.
But then by ($\pcosieve{K_f}{\L}$-many) successive applications of $\Pitype$-elimination using (\ref{Kftype}) and (\ref{qfvar}) we obtain 
\begin{equation}
\Delta \vdash \TTapp[K_f(\mathbf{q}f)] \colon \U
\end{equation}
from which we immediately obtain (\ref{deltants}) as required.
%And by Lemma \ref{strucwellformedlemma0} we know that 
%\begin{equation}
%\Gamma_A, (gf \colon \TTElapp{K_{gf}(\mathbf{p}gf)})_{g<f}
%\end{equation}
%is a subcontext (up to isomorphism, i.e. renaming of variables) of $\Delta$.
%Since for each $q$ in (\ref{deltants}) we have $qf<f$ so it appears in $\Delta'$ and therefore $\Delta$ and by the inductive hypothesis on $f$ is of the right type to derive (\ref{deltants}), as required. 
Hence $\Gamma_A \vdash T_K \TTisatype$ and we are done.
\end{proof}

\begin{theo}\label{strucwellformed}
$\Struc{\L}$ is a closed type.
\end{theo}
\begin{proof}
Since $\L$ is assumed finite, by Lemma \ref{strucwellformedlemma} we get immediately that
\[
\Gamma_\L \eqdef (K \colon T_K)_{K \in \ob{\L}}
\]
is well-formed. 
And then we obtain $\Struc{\L}$ (in the empty context) by successive applications of $\Sigma$-formation.
\end{proof}

\begin{defin}
An \textbf{$\L$-structure} $\M$ is a term of $\Struc{\L}$.
\end{defin}

We will now define an interpretation $\interpretation{-}$ of $\LTT_\L$ into MLTT.
% for any given $\L$-structure $\M$. 
To be clear, what we go on to describe in Definition \ref{interpretationFOLDS} below is a mapping from the raw syntax of $\LTT_\L$ into the raw syntax of MLTT.
%Althernatively, one can think of our interpretation as mapping into the raw syntax of MLTT except by a ``context'' in MLTT we now mean an expression which begins with a declaration $\M \colon \Struc{\L}$.
% given a term $\M \colon \Struc{\L}$. 
We will then prove in Theorem \ref{FOLDSsoundness} that this mapping preserves the rules of $\LTT_\L$ and therefore deserves to be called an interpretation.
In particular, this means that we automatically obtain an interpretation of $\LTT_\L$ into any categorical model of MLTT (e.g. CwFs or $C$-systems with the appropriate extra structure).

\begin{notation}
For any $\L$-structure $\M$ and any sort $K$ in $\L$ we denote by $K^\M$ the ``$K^{\text{th}}$ projection of $\M$''. More precisely if $K_1 < \dots K_m$ are the (ordered) objects of $\L$ then 
\begin{equation}
(K_i)^\M \eqdef \TTproj_i (\M)
\end{equation}
\end{notation}

\begin{convention}
We will assume that the variables of MLTT are given by a set  that contains the variables $V$ of $\LTT_\L$. This allows us to simply re-use the same symbol for a variable and its interpretation, and as such we will not include the mapping of variables of $\LTT_\L$ to variables of MLTT in the definition below.
\end{convention}

\newcommand{\FOLDSdepth}[1]{d(#1)}

\begin{defin}[Interpretation of $\LTT_\L$ in MLTT]\label{interpretationFOLDS}
The \textbf{depth} of the expressions in the raw syntax of $\LTT_\L$ is defined as follows:
\begin{align*}
\FOLDSdepth{\varnothing} &= 0 \\
\FOLDSdepth{\Gamma, x \colon K} &= \FOLDSdepth{\Gamma} + \FOLDSdepth{K}+1 \\
\FOLDSdepth{A (\mathbf{x})} &= \level{A}+1 \\
\FOLDSdepth{\top} &= \FOLDSdepth{\bot} = 1  \\
\FOLDSdepth{\phi * \psi} &= \FOLDSdepth{\phi} + \FOLDSdepth{\psi} \quad \quad \quad \: (*= \wedge, \vee, \rightarrow) \\
\FOLDSdepth{Qx \colon K. \phi} &= \FOLDSdepth{K}+\FOLDSdepth{\phi} \quad \quad \quad (Q= \exists, \forall) \\
\FOLDSdepth{\Gamma \isacontext} &= \FOLDSdepth{\Gamma} \\
\FOLDSdepth{\Gamma \vdash K \isatype} &= \FOLDSdepth{\Gamma} + \FOLDSdepth{K} \\
\FOLDSdepth{\Gamma \vdash x \colon K} &= \FOLDSdepth{\Gamma} +1 \\
\FOLDSdepth{\Gamma \vdash \phi \isaformula} &= \FOLDSdepth{\Gamma}+\FOLDSdepth{\phi} \\
\end{align*}
The \textbf{interpretation of $\LTT_\L$ into MLTT}
 %into an $\L$-structure $\M$} 
 consists of the following functions:
\begin{align*}
\interpretation{-}_c \colon \FOLDScontexts &\rightarrow \MLTTcontexts \\
                                \varnothing          &\mapsto     \M \colon \Struc{\L}     \\
                                \Gamma, x \colon K &\mapsto \interpretation{\Gamma}_c, x \colon \interpretation{K}_s
\end{align*}
\begin{align*}
\interpretation{-}_s \colon \FOLDSsorts &\rightarrow \MLTTtypes \\
			       %K \epsilon &\mapsto \mathtt{El}(\TTapp(*,K^\M)) \\
                                K (x_{\mathbf{p}f})          &\mapsto     \TTElapp{K^\M (x_{\mathbf{p}f})}      \\                  
\end{align*}
\begin{align*}
\interpretation{-}_f \colon \FOLDSformulas &\rightarrow \MLTTtypes \\
                                \bot         &\mapsto     \zerotype                 \\ 
                                \top          &\mapsto     \onetype                 \\
                                \phi \wedge \psi         &\mapsto   \interpretation{\phi}_f \times \interpretation{\psi}_f        \\
                                \phi \vee \psi         &\mapsto   \interpretation{\phi}_f + \interpretation{\psi}_f   \\ 
                               \phi \rightarrow \psi         &\mapsto \interpretation{\phi}_f \rightarrow \interpretation{\psi}_f    \\
                                \exists x \colon K. \phi &\mapsto \Sigmatype \:(x \colon \interpretation{K}_s)\: \interpretation{\phi}_f \\
                                \forall x \colon K. \phi &\mapsto \Pitype \:(x \colon \interpretation{K}_s) \: \interpretation{\phi}_f
\end{align*}
We will drop the subscripts of each separate interpretation function and write simply $\interpretation{\Gamma}, \interpretation{K}, \interpretation{\phi}$. 
With this in mind the interpretation of the judgments of $\LTT_\L$ is given by the following function:
\begin{align*}
\interpretation{-} \colon \FOLDSjudgments &\rightarrow \MLTTjudgments \\
                                \Gamma \isacontext          &\mapsto     \interpretation{\Gamma} \isacontext      \\
                                \Gamma \vdash K \isatype &\mapsto \interpretation{\Gamma} \vdash  \interpretation{K} \TTisatype  \\
     \Gamma \vdash x \colon K &\mapsto \interpretation{\Gamma} \vdash x \colon  \interpretation{K} \\  
     \Gamma \vdash \phi \isaformula &\mapsto \interpretation{\Gamma} \vdash \interpretation{\phi} \TTisatype                   
\end{align*}
For a judgment $\mathcal{S}$ in $\LTT_\L$ we write $\interpretation{S}$ for its interpretation in MLTT.
\end{defin}

\begin{remark}
Note that instead of interpreting the empty $\L$-context as $\M \colon \Struc{\L}$ we could have instead interpreted it as the context $\Gamma_\L = (K \colon T_K)_{K \in \ob{\L}}$ as defined in the proof of Theorem \ref{strucwellformed}. The difference is entirely inessential with respect to anything we have to say in the rest of this paper.
\end{remark}

%\begin{notation}
%For a judgment $\mathcal{S}$ in $\LTT_\L$ we write $\interpretation{S}$ for its interpretation in MLTT.
%\end{notation}

\begin{theo}[Correctness of the Interpretation]\label{FOLDSsoundness}
If $\mathcal{S}$ is a derivable judgment in $\LTT_\L$ then $\interpretation{S}$ is a derivable judgment in MLTT. 
\end{theo}
\begin{proof}
We proceed by induction on the depth of the expressions of $\LTT_\L$. By inspection we can see that the depth function in Definition \ref{interpretationFOLDS} is correct in the sense that for every rule of $\LTT_\L$ the depth of the expression below the line is strictly greater than the depth of any of the expressions above the line. 
The proof thus can proceed by induction on the complexity of derivations of $\LTT_\L$ which here means that we need to show that every rule of $\LTT_\L$ is valid in MLTT under the interpretation in Definition \ref{interpretationFOLDS}. We take them in turn:

\bigskip
\noindent $\emptycon$: We have 
\begin{equation}
\interpretation{\varnothing \isacontext} = \M \colon \Struc{\L} \isacontext
\end{equation}
But by Theorem \ref{strucwellformed} $\Struc{\L}$ is a well-formed type in the empty context which means that $\M \colon \Struc{\L}$ is a well-formed context.

\bigskip
\noindent $\conext$: Given that $\interpretation{\Gamma}$ is a well-formed context and that $\interpretation{\Gamma} \vdash \interpretation{K} \TTisatype$ is derivable, we get immediately that
\begin{equation}
\interpretation{\Gamma}, x \colon \interpretation{K} \isacontext
\end{equation}
is derivable. But $\interpretation{\Gamma, x \colon K} = \interpretation{\Gamma}, x \colon \interpretation{K}$ and so the above sequent is exactly the interpretation of the sequent below the line in $\conext$.

\bigskip 
\noindent  $\weakening$: Follows immediately from the analogous weakening rule in MLTT.

\bigskip 
\noindent  $\axrule$: Follows immediately from the analogous rule in MLTT.

\bigskip
\noindent $\Kform$: By Lemma \ref{appearlemma} we know that if $\Gamma \vdash x_f \colon K_f (x_{\mathbf{p}f})$ is derivable in $TT_\L$ (and hence $\LTT_\L$) then $x_f \colon K_f (x_{\mathbf{p}f})$ must appear in $\Gamma$. Therefore $x_f \colon \interpretation{K_f (x_{\mathbf{p}f})}$ must appear in $\interpretation{\Gamma}$ for every $f \in \pcosieve{K}{\L}$. By the definition of the interpretation $\interpretation{-}$ we know that 
\begin{equation}
\interpretation{K(x_f)_{f \in \pcosieve{K}{\L}}} = \TTElapp{K^\M(x_f)_{f \in \pcosieve{K}{\L}}} 
\end{equation}
Since $K^\M$ is a dependent function into $\U$ of the appropriate type (by the definition of $\Struc{\L}$) and each $x_f$ appears in the context $\interpretation{\Gamma}$ then we know that 
$
\TTapp [K^\M(x_f)_{f \in \pcosieve{K}{\L}}]
$
is a derivable term of $\U$ and therefore that
\begin{equation}
\interpretation{\Gamma} \vdash  \TTElapp{K^\M(x_f)_{f \in \pcosieve{K}{\L}}} \TTisatype
\end{equation}
is derivable. But this last sequent is exactly the interpretation of the bottom line of $\Kform$, as required.

\bigskip
\noindent $\formwk$: Follows immediately from the analogous weakening rule in MLTT.

\bigskip
\noindent $\topform, \botform$: Since $\onetype, \zerotype$ are included in our type theory these follow immediately from the corresponding formation rules since $\interpretation{\Gamma \vdash \top \isaformula} = \interpretation{\Gamma} \vdash \onetype \: \TTisatype$ and similarly for $\bot$.

\bigskip
\noindent $\starform$: For each of $*= \wedge, \vee, \rightarrow$ the $\starform$ rule follows from the corresponding formation rule for the relevant type former. 
For example, given that $\interpretation{\Gamma}$ is a well-formed context and $\interpretation{\phi}$ and $\interpretation{\psi}$ are types in that context we have
\begin{equation}
\interpretation{\Gamma}, x \colon \interpretation{\phi} \vdash \interpretation{\psi} \TTisatype
\end{equation} 
Then by $\Pitype$-formation we get 
\begin{equation}
\interpretation{\Gamma} \vdash \interpretation{\phi} \rightarrow \interpretation{\psi} \TTisatype
\end{equation}
as required. The cases for $\vee$ and $\wedge$ follow analogously using formation rules for $+$ and $\Sigmatype$ respectively.
%%%HOTTmove
%We outline the case of $\vee$ since it involves an application of truncation. Given that $\interpretation{\Gamma}$ is a well-formed context and $\interpretation{\phi}$ and $\interpretation{\psi}$ are types in that context, then by $+$-formation we get 
%\begin{equation}
%\interpretation{\Gamma} \vdash \interpretation{\phi} + \interpretation{\psi} \TTisatype
%\end{equation}
%and then by an application of truncation we get
%\begin{equation}
%\interpretation{\Gamma} \vdash \truncated{\: \interpretation{\phi} + \interpretation{\psi} \:} \TTisatype
%\end{equation}
%which is exactly the interpretation of the bottom line of the formation rule for $\vee$.
%%%HOTTmove

\bigskip
\noindent $\Qform$: For each of $Q= \exists, \forall$ the $\Qform$ rule follows from the corresponding formation rule for the relevant type former. 
For example, given that $\interpretation{\Gamma}, x \colon \interpretation{K}$ is a well-formed context and $\interpretation{\phi}$ and is a type in that context we have by $\Sigmatype$-formation that
\begin{equation}
\interpretation{\Gamma} \vdash \Sigmatype \: (x \colon \interpretation{K}) \: \interpretation{\phi}\TTisatype
\end{equation} 
%Then by $\Pitype$-formation we get 
%\begin{equation}
%\interpretation{\Gamma} \vdash \interpretation{\phi} \rightarrow \interpretation{\psi} \TTisatype
%\end{equation}
as required. The case for $\forall$ and $\wedge$ follows analogously using $\Pitype$-formation.
%%%HOTTmove
%We once again outline the case of $\exists$ since it involves an application of truncation.
%Given that $\interpretation{\Gamma, x \colon K} =  \interpretation{\Gamma}, x \colon \interpretation{K}$ is a well-formed context and $\interpretation{\phi}$ is a type in that context, then by $\Sigmatype$-formation we get 
%\begin{equation}
%\interpretation{\Gamma} \vdash \Sigmatype \: (x \colon \interpretation{K}) \: \interpretation{\phi} \TTisatype
%\end{equation}
%and then by an application of truncation we get
%\begin{equation}
%\interpretation{\Gamma} \vdash \truncated{ \: \Sigmatype \: (x \colon \interpretation{K}) \: \interpretation{\phi} \:} \TTisatype
%\end{equation}
%which is exactly the interpretation of the bottom line of the formation rule for $\exists$.
%%%HOTTmove
\end{proof}

\begin{notation}
We will usually want to assume that we are given a specific $\L$-structure $\M$ and speak of contexts and formulas in $\M$. To express this we will use the notation $\Gamma^\M$ and $\phi^\M$ for contexts and formulas. 
We will also use the notation $\interp{A(\mathbf{x})}{\M}$ for sorts, noting that it must be distinguished from the notation $K^\M$ for an object $K \in \L$ which we have already used. The former denotes a type (in context) whereas the latter denotes a term of a certain type (a dependent function into a universe). 
%For sorts we will abuse notation and also use the notation $^\M$ not just for the declaration of the type of the sort $K$ in $\M$ but rather for a specific instance of $K^\M$ in some context $\Gamma^\M$.
\end{notation}

The following result is now essentially a tautological restatement of Theorem \ref{FOLDSsoundness}.

\begin{cor}\label{FOLDSsoundnesscor}
For any $\L$-structure $\M$ we have:
\begin{enumerate}
\item If $\Gamma$ is an $\L$-context then $\interp{\Gamma}{\M}$ is a well-formed context.
\item If $\phi$ is an $\L$-formula in context $\Gamma$ then $\interp{\phi}{\M}$ is a type in context $\interp{\Gamma}{\M}$.
\item If $A(\mathbf{x})$ is an $\L$-sort in context $\Gamma$ then $\interp{A(\mathbf{x})}{\M}$ is a type in context $\interp{\Gamma}{\M}$.
\end{enumerate}
\end{cor}

\begin{cor}\label{FOLDSsoundnesscor2}
If $\alpha \colon \Gamma \conmorphism \Delta$ is a well-formed $\L$-context morphism then $\alpha^\M \colon \Gamma^\M \conmorphism \Delta^\M$ is a context morphism where $\alpha^\M$ consists of the interpretations of the sequents defining $\alpha$.
\end{cor}
\begin{proof}
If $\alpha$ is given by
\begin{align*}
\Gamma &\vdash y_1 \colon K_1 \\
\Gamma &\vdash y_2 \colon K_2 [y_1/x_1] \\
 &\: \vdots \\
\Gamma &\vdash y_n \colon K_n [y_1/x_1,\dots,y_{n-1}/x_{n-1}]
\end{align*}
then $\alpha^\M$ is given by
\begin{align*}
\Gamma^\M &\vdash y_1 \colon K_1^\M \\
\Gamma^\M &\vdash y_2 \colon K_2^\M [y_1/x_1] \\
 &\: \vdots \\
\Gamma^\M &\vdash y_n \colon K_n^\M [y_1/x_1,\dots,y_{n-1}/x_{n-1}]
\end{align*}
which means exactly that we have a context morphism $\Gamma^\M \conmorphism \Delta^\M$ in MLTT.
\end{proof}

%\begin{cor}\label{formulasaremerepropositions}
%For any formula-in-context $\Gamma.\phi$ and any $\L$-structure $\M$ we have that $\phi^\M$ is a mere proposition (in context $\Gamma^\M$).
%\end{cor}
%\begin{proof}
%Immediate by the definition of the interpretation of formulas, since the only type-formers that do not preserve $h$-level are $\Sigmatype, +$ and we truncate them.
%\end{proof}

%\begin{corr}
%Let $C$ be a $C$-system with $\Pi, \Sigma, 1, \U$-structure. Then there is an interpretation of $\LTT_\L$ 
%\end{corr}

\newcommand{\Extension}[2]{\text{Ext}_{#1}^{#2}}

\begin{defin}[Interpretation of a Sequent in $\M$]
The \textbf{interpretation of a sequent $\Gamma \:\vert\: \phi \seqimplies \psi$} in an $\L$-structure $\M$ is given by 
\begin{equation}
\interp{(\Gamma \:\vert\: \phi \seqimplies \psi)}{\M} \eqdef \Pitype \:(\interp{\Gamma}{\M})\: \interp{\phi}{\M} \rightarrow \interp{\psi}{\M}
\end{equation}
which is a closed type by Corollary \ref{FOLDSsoundnesscor}.
\end{defin}

\begin{defin}[Extension of a Formula]\label{FOLDSextension}
The \textbf{extension} of an $\L$-formula $\Gamma.\phi$ in an $\L$-structure $\M$ is the type 
\[
\Extension{\Gamma}{\M} (\phi) \eqdef \Sigmatype \:(\Gamma^\M)\: \phi^\M
\]
We call a derivable term $\mathbf{a} \colon \Extension{\Gamma}{\M} (\phi)$ a \textbf{realization} of $\Gamma.\phi$ in $\M$.
\end{defin}

\begin{defin}[Satisfaction of a Formula]\label{FOLDSsatisfaction}
Let $\phi$ be a $\L$-formula in context $\Gamma$ and $\M$ an $\L$-structure. We define the \textbf{satisfaction of $\Gamma.\phi$ by $\mathbf{a}$ in $\M$} as follows:
\[
\M \models \phi [\mathbf{a}/\Gamma] \text{ iff  \:$\mathbf{a}$ is a realization of $\Gamma.\phi$ in $\M$}
\]
%where the notation ``$[\vec{a}/\Gamma]$'' on the right-hand side expresses the substitution into $\phi^\M$ of each variable in $\Gamma$ to its corresponding term in $\vec{a}$. 
The case where $\phi$ has no free variables is a special case of the above definition, in which case we write $\M \models \phi$ and say that $\M$ is a \textbf{model} of $\phi$.
Satisfaction for sequents can be defined similarly as follows:
%We define:
\[
\M \models \Gamma \: \vert \: \phi \seqimplies \psi \text{ iff } \Pitype \: (\Gamma^\M) \: \phi^\M \rightarrow \psi^\M \text{ is inhabited}
\]
We say that $\M$ \textbf{satisfies} a sequent $\sigma$ if $\M \models \sigma$.
%where $p,p'$ are variables in MLTT.
We say that $\M$ is a \textbf{model} of a FOLDS $\L$-theory $\T$ if $\M$ satisfies every sequent in $\T$.
\end{defin} 

\def\modelscl{\models_{\text{cl}}}
\def\modelsboth{\models_{(\text{cl})}}

\begin{notation}
If we want to indicate that our semantics is taken in classical MLTT we will write $\modelscl$ and if we want to indicate that what we say applies to both MLTT and classical MLTT we will write $\modelsboth$
\end{notation}

\begin{remark}
Our notion of satisfaction is defined in terms of the notion of the derivability of a certain judgment in type theory. In terms of the categorical semantics of type theory, our notion of satisfaction depends on the (mere) existence or non-existence of a certain section to a certain canonical projection.
\end{remark}

\begin{defin}[Semantic Consequence]\label{FOLDSsemanticconsequence}
For any FOLDS $\L$-theory $\T$ and $\L$-sequent $\tau$ we write $\T \modelsboth \tau$ for the statement that for every (classical) $\L$-structure $\M$ if $\M \modelsboth \sigma$ for all $\sigma$ in $\T$ then $\M \modelsboth \tau$, in which case we say that $\tau$ is a \textbf{semantic consequence} of $\sigma$.
\end{defin}

%\begin{remark}
%Equivalently, we may define semantic consequence as the statement that the type
%\[
%\Sigma \modelsboth \tau \eqdef \Pitype (\interpretation{\Gamma}) \interpretation{\Sigma} \rightarrow \interpretation{\psi}
%\]
%is inhabited in (classical) MLTT, where $\interpretation{\Sigma}$ is defined by 
%\[
%\underset{\sigma \in \Sigma}{\times} \interpretation{\sigma}
%\]
%\end{remark}

\begin{theo}[Soundness]\label{soundness}
Let $\T$ be a FOLDS $\L$-theory. If $\T \FOLDSentailsboth \tau$ then $\T \modelsboth \tau$.
\end{theo}
\begin{proof}
The proof proceeds without difficulties by induction on the complexity of $\DFOLDS$-derivations since we already know that the interpretation of an $\L$-sequent is a closed type in MLTT. 
%This means that we take each particular rule in the deductive system and show that there is a derivation from the translation of the top line to the translation of the bottom line. There is one minor subtlety. Axioms in $\D$ are stated by starting from an empty line and then producing a formula. 
For example, for
\[
\inferrule
{
\!
}
{
\Gamma \: \vert \: \phi \seqimplies \phi
}
\quad \idenrule
\]
%There is a certain amount of information suppressed in stating such an axiom, namely that the given sequent is well-formed. In translating an axiom like (iden) we will therefore translate the ``empty'' set of formulas above as the judgement in HoTT stating that the types involved in the translation of the formula below are well-formed. 
%This can be understood as codifying the information of the well-formedness of the formula which is implicit in the statement of (iden). 
%(This is essentially the price we pay when we interpret a proof-irrelevant system into a proof-relevant one.)
%So in translating (iden)
we know that $(\Gamma \: \vert \: \phi \seqimplies \phi)^\M \equiv \Pitype \:(\Gamma^\M) \: \phi^\M \rightarrow \phi^\M$ and
\[
\vdash \Pitype \:(\Gamma^\M) \: \phi^\M \rightarrow \phi^\M \TTisatype
\]
%which captures the well-formedness of $\phi$ (which is implicit in the statement of (iden)) to the judgment
%\[
%\vdash s \colon \underset{\mathbf{x}\colon \Gamma^\M}{\Pitype} \phi^\M \rightarrow \phi^\M
%\]
%for some $s$ and for any $\M$. 
To show that (iden) is sound we therefore need to show that $\Pitype \:(\Gamma^\M) \: \phi^\M \rightarrow \phi^\M$ is inhabited, i.e. that there is a derivation of a term of that type. In a somewhat abbreviated form this derivation goes as follows
\[
\inferrule
{
\inferrule{\Gamma^\M \vdash \phi^\M \TTisatype}
	{\inferrule{\Gamma^\M \vdash \phi^\M \rightarrow \phi^\M \TTisatype}
		{\vdash \Pitype \: (\Gamma^\M) \: \phi^\M \rightarrow \phi^\M  \TTisatype}
		\quad (\text{$\Pitype$-form})
	}
	\quad (\text{$\Pitype$-form), (wkg)}
}
{
\vdash \lambda \mathbf{x} \colon \Gamma^\M. (\lambda y. y) \colon \Pitype \: (\Gamma^\M) \: \phi^\M \rightarrow \phi^\M
}
\quad (\text{$\Pi$-intro})
\]
%where the term $\lambda \mathbf{y}. \mathbf{y}$ has been produced by applying $\Pi\text{-intro}$ to 
%\[
%\Delta^\M, \Gamma^\M, y \colon \phi^\M \vdash y \colon \phi^\M
%\]
Similar derivations work for the rest of the structural rules of $\DFOLDS$. 
%Since the languages we are considering are purely relational and there are no closed terms, the substitution rule follows immediately without any complications. 
%As for the cut rule, given terms
%\[
% \vdash \eta \colon \underset{\mathbf{x}\colon \Gamma^\M}{\Pi} \phi^\M \rightarrow \psi^\M 
%\]
%and
%\[
% \vdash \xi \colon \underset{\mathbf{x}\colon \Gamma^\M}{\Pi} \psi^\M \rightarrow \chi^\M
%\]
%we can define a term
%\[
%\Delta^\M \vdash \lambda \mathbf{x}. (\lambda y \colon \phi^\M. \xi(\mathbf{x})(\eta(\mathbf{x})(y))) \colon \underset{\mathbf{x}\colon \Gamma^\M}{\Pi} \phi^\M \rightarrow \chi^\M
%\]
%The rest of the structural rules follow simlarly.
The soundness of the logical rules of $\DFOLDS$ follow straightforwardly by the introduction rules of the relevant type formers. We omit the (well-known) details.
\end{proof}

Finally we want to be able to define the ``type of models'' of an $\L$-theory. But to do so we need to be able to list the sequents that comprise the theory within MLTT, a qualification made precise by the following definition.
%\begin{defin}
%A \textbf{theory} over $\L$ (or $\L$-theory) is a set of $\L$-sequents.
%\end{defin}

\begin{defin}
A FOLDS $\L$-theory $\T = \lbrace \sigma_i \rbrace_{i\in \mathbb{N}}$ is \textbf{internalizable} if there is a family of types
\[
i \colon \mathbb{N}, \M \colon \Struc{\L} \vdash \T_i (\M) \TTisatype
%\T \colon \Pitypecom{i \colon \mathbb{N}}{\Struc{\L} \rightarrow \U}
\] 
such that $\T_i (\M) \equiv (\sigma_i)^\M$.
\end{defin}

\begin{remark}
Every finite theory will be internalizable in MLTT since we can write out the list of sequents ``by hand''. Infinite theories will be internalizable to the extent that $\LTT_\L$ can be internalized in MLTT.
\end{remark}

\begin{defin}[Type of Models]\label{FOLDStypeofmodels}
For an internalizable theory $\T$, we define the \textbf{type of models} of $\T$ as
\[
\Mod{\T} \eqdef \Sigmatypecom{\M \colon \Struc{\L}}{\Pitypecom{i \colon \mathbb{N}}{\sigma_i^\M}}
\]
\end{defin}

%\begin{terminology}
%We say that a dependent type theory $TT$ is an \emph{extension} of MLTT (as we have understood the latter in this section) it contains all the syntax,  judgments and rules of MLTT (but has possibly more type constructors). Let $C(TT)$ a categorical semantics for $TT$, i.e. $C(TT)$ is a (usually essentially algebraic) first-order theory with the right structure to interpret the syntax and rules of $TT$.
%We say that \emph{$C(TT)$ is an initial categorical semantics for $TT$} if the initiality conjecture holds for $C(TT)$, i.e. the term model $M_{TT}$ for $TT$ is a model of $C(TT)$ and is initial among all other models of $C(TT)$. 
%We can say that a given FOLDS signature $\L$ \emph{interprets into} a type theory $TT$ if $\LTT_\L$ interprets into $TT$.
%\end{terminology}
%
%\begin{cor}
%Let $TT$ be an extension of MLTT, $C(TT)$ an initial categorical semantics for $TT$, and $M$ a model of $C(TT)$. Then
%\end{cor}

\section{Syntax of $\FOLiso$}\label{nlogSyntax}

We will now introduce the syntax of $\FOLiso$ as an extension of the syntax of FOLDS described in Section \ref{prelim}. 
The signatures of $\FOLiso$ will be given by inverse categories $\L$ with extra structure.
This extra structure consists of an assignment of a number to each object of $\L$ corresponding intuitively to its $h$-level, as well as allowing for certain specified objects and arrows encoding equality, reflexivity and transport (along equalities).
Given these signatures, the contexts, formulas, sequents etc. of $\FOLiso$ can be defined just as in Section \ref{prelim} by slight modifications that take into account the extra structure on the relevant signatures.
%The syntax of $n$-logic is thus a formalization of the idea that we ``add equalities'' (globular completion) and then by extending add predicates and relations that talk \emph{about} these equalities.

%\begin{notation}
%We write $\Natinfty$ for the set $\mathbb{N} \cup \lbrace \infty \rbrace$.
%We will sometimes write $K \in \L$ as an abbreviation for $K \in \ob{\L}$ and we write $\L(K,K')$ for the morphisms from $K$ to $K'$ in $\L$. 
%\end{notation}

As before, we assume fixed countably infinite disjoint sets $O$ and $M$.

\begin{defin}[Logical Sorts and Arrows]\label{symbols}
We define by mutual induction two disjoint sets of symbols:
\begin{align*}
\Ob (O,M) \quad \quad K &::=  A  \:\vert\: \eqsort_K \:\vert\: \rsort_K \:\vert\: \tsort_f  \quad \quad \quad \quad  \quad \quad \quad \:\:\:\:  (A \in O)\\
\Mor (O,M) \quad \quad f & ::= h \:\vert\: s_K \:\vert\: t_K \:\vert\: \rmor_K \:\vert\: \tmorX{f} \:\vert\: \tmorXX{f} \:\vert\: \tmorpath{f} \quad \quad \quad \quad (h \in M)
\end{align*}
%satisfying the following equations
%\begin{itemize}
%\item $\tsort_{1_K} \equiv \:\eqsort_K$
%\item $\tmorX{(1_K)} \equiv s_K$
%\item $\tmorXX{(1_K)} \equiv t_K$
%\item $\tmorpath{1_K} \equiv 1_{=_K}$
%\end{itemize}
\end{defin}

\begin{notation}
We will drop explicit mention of $O,M$ and write simply $\Ob$ and $\Mor$ for $\Ob (O,M)$ and $\Mor (O,M)$, with the understanding, as before, that everything we say below is parametrized by our choice of $O$ and $M$.
\end{notation}

%\begin{remark}
%%As above, we will usually suppress indices for $r$ when it is clear from the context.
%We will usually suppress explicit mention of $K$ in the $s$, $t$ and $r$.
%\end{remark}

\begin{defin}[$h$-signatures]\label{hsignatures}
An \textbf{$h$-signature} is a pair $(\L,h)$ where:
\begin{itemize}
\item $\L$ is a FOLDS signature over $(\Ob, \Mor)$ 
%\textcolor{red}{with identities given by the symbols $1_K$}
\item $h$ is a function $h \colon \ob{\L} \rightarrow \Natinfty$ 
\end{itemize}
such that:
\begin{enumerate}
\item If $\eqsort_K \in \L$ then we have: 
\begin{enumerate}
\item $K \in \L$ with $h(K)\geq 2$
\item $\toplevel{\eqsort_K} = \lbrace s_K, t_K \rbrace = \L(\eqsort_K,K)$ 
\item $f \catcomp s_K = f \catcomp t_K$ for any $f \colon K \rightarrow K'$
\item $h(\eqsort_K) = h(K) - 1$
\end{enumerate}
\item If $\rsort_K \in \L$ then we have: 
\begin{enumerate}
\item $K \in \L$ with $h(K) \geq 2$ 
\item $\toplevel{\rsort_K} = \lbrace \rmor_K \rbrace = \L (\rsort_K, \eqsort_K)$ 
\item $s_K \catcomp \rmor_K = t_K \catcomp \rmor_K$
\item $h(\rsort_K)=h(K)-2$
\end{enumerate}
\item If $\tsort_f \in \L$ then we have:
\begin{enumerate}
%\item $K \in \L$ with $h(K) \geq 2$
\item $f \in \L(A,K)$ for some $A,K \in \L$ with $h(K) \geq 3$ and $f \in \toplevel{A}$
\item $\toplevel{\tsort_f}= \lbrace \tmorpath{f}, \tmorX{f}, \tmorXX{f} \rbrace$ with $\L(\tsort_f, \eqsort_K)=\lbrace \tmorpath{f} \rbrace$ and $\L (\tsort_f, A)= \lbrace \tmorX{f}, \tmorXX{f} \rbrace$
\item $s_K \catcomp \tmorpath{f} = f \catcomp \tmorX{f}$ and $t_K \catcomp \tmorpath{f} = f \catcomp \tmorXX{f}$
\item $g \catcomp \tmorX{f} = g \catcomp \tmorXX{f}$, for all $g \in \pcosieve{A}{\L} \setminus \lbrace f \rbrace$
% such that $g \neq f$
%\item $pf\tmorX{f}=pf\tmorX{f}$, for all non-identity $p \colon K \rightarrow K'$
\item $q \catcomp f \catcomp f_1 = q \catcomp f \catcomp f_2$, for all $q \in \pcosieve{K}{\L}$
\item $h(\tau_f) = h (A) -1$
\end{enumerate}
\item If $h(K)=0$ then $K \equiv \rsort_A$ for some $A$ in $\L$
%\item Whenever $\eqsort_K, \rsort_K$, or $\tsort_f \in \L$ we have:
%\begin{itemize}
%\item $h(\eqsort_K) = h(K) +1$
%\item $h(\rsort_K)=h(K)+2$
%\item $h(\tau_f) = h (\dom{f}) +1$
%\end{itemize}

In all the above conditions involving $h$ we stipulate that $\infty - m = \infty$.
\end{enumerate}
\end{defin}

\begin{terminology}
For an $h$-signature $(\L,h)$ we call $h(K)$ the \emph{$h$-level} of $K$. The \emph{height} $\Heightof{\L,h}$ of $(\L,h)$ is the maximum $h$-level in $\L$, i.e.
\[
\Heightof{\L,h} = \underset{K \in \obL}{\text{sup}} h(K)
\]
We call $\eqsort_K$ the \emph{isomorphism sort} of $K$, $\rsort_K$ the \emph{reflexivity sort} (or \emph{predicate}) on $\eqsort_K$ and for $f \colon A \rightarrow K$ we call $\tsort_f$ the \emph{transport structure of $A$ in position $f$ (along the path picked out by $\tmorpath{f}$}). 
We call $s_K$ and $t_K$ the \emph{source} and \emph{target} maps (for equality ``paths''). 
We call all these symbols \emph{logical} and any symbol that is not of this kind \emph{non-logical}.
For a given inverse category $\L$ over $(\Ob, \Mor)$ we will write $\NL{\L}$ for the set of non-logical sorts of $\L$ and we will refer to such an $\L$ with a finite $\NL{\L}$ as \emph{essentially finite}.
%Analogously we will refer to the objects of the form $\eqsort_K, \rsort_K$ or $\tsort_f$ as \emph{logical sorts} and those not of this form as 
\end{terminology}

\begin{remark}\label{hsigexplanation}
We now illustrate the idea behind each of the conditions in Definition \ref{hsignatures} at some length, anticipating the semantics in Section \ref{nlogSemantics}.

The function $h \colon \ob{\L} \rightarrow \mathbb{N}_\infty$ is to be understood as picking out the $h$-level of the dependent sort it is attached to, in the sense that a given sort with $h(K)=m$ is to be understood as a family of types of $h$-level $m$ dependent on the sorts in $\pcosieve{K}{\L}$ in the appropriate way.

The isomorphism sorts are to be understood as the identity types/path spaces of (the family defined by) a sort (at each of its instances). 
%They are best understood as \emph{isomorphism sorts}, rather than isomorphism sorts, and it is indeed this fact that gives $\FOLiso$ its name.
The conditions on an isomorphism sort $\eqsort_K$ ensure that it appears in an $h$-signature only in the following form
\[
\xymatrix{
m-1 &\eqsort_K \ar@/_/[d]_{s_K} \ar@/^/[d]^{t_K} \\
m    &K 
}
\]
where $m \geq 2$ because we only want to have an isomorphism sort over types that are at least sets (i.e. at least of $h$-level 2). 
The idea is that the identity type on a type $K$ is always of $h$-level one less than the $h$-level of $K$ (if that $h$-level is known to be finite) and that there can be no more arrows with domain $\eqsort_K$ since the identity type depends only on $K$ (this is what the condition $\toplevel{\eqsort_K}=\lbrace s_K, t_K \rbrace$ achieves). Finally given any arrow $f \in \pcosieve{K}{\L}$ as in
\[
\xymatrix{
m-1 &\eqsort_K \ar@/_/[d]_{s_K} \ar@/^/[d]^{t_K} \\
m    &K \ar@{-->}[d]^{f} \\
n       &K_f
}
\]
the condition $fs_K=ft_K$ ensures that to declare a variable $p \colon x \eqsort_K y$ then $x$ and $y$ must have the same dependency, as is to be expected.

The reflexivity predicates are to be understood as a type family picking out the reflexivity path from the identity type ``below'' them, i.e. as the identity type of paths identical to reflexivity. 
%Note that even though we call $\rsort_K$ a ``predicate'' for ease of understanding, it is not to be understood as a mere proposition since its $h$-level can be arbitrarily large, even though, being an identity type on the identity type of $K$, it will always be of $h$-level $h(K)-2$.
The conditions on a reflexivity predicate $\rsort_K$ ensure that it appears in an $h$-signature only in the following form
\[
\xymatrix{
m-2 &\rsort_K \ar[d]^{r_K} \\
m-1 &\eqsort_K \ar@/_/[d]_{s_K} \ar@/^/[d]^{t_K} \\
m    &K 
}
\]
with the condition $s_Kr_K=t_Kr_K$ ensuring (as in $\Lrg$) that we can only ask of a path if it is reflexivity if we know that it is a loop.

The transport structure is to be understood as a (functional) relation relating a term of a type to its transport along a path.
The conditions on a transport relation $\tau_f$ ensure that it appears in an $h$-signature only in the following form
\[
\xymatrix{
m-1 & & &\tsort_f \ar@/_/[ld]_{\tmorX{f}} \ar@/^/[ld]^{\tmorXX{f}} \ar[rd]^{\tmorpath{f}} \\
m  & &A \ar@{-->}[ld]_g \ar[rd]_{f}  & &\eqsort_O \ar@/_/[ld]_{s_K} \ar@/^/[ld]^{t_K}\\
n   &K &   &O
}
\]
with the conditions ensuring that the ``path'' $e_f$ has as source the ``point'' $f\tmorX{f}$ and as target the ``point'' $f\tmorXX{f}$ (condition 3.(c)), that the ``points'' $\tmorX{f}$ and $\tmorXX{f}$ belong to types  differing only in ``position'' $f$ (condition 3.(d)) and that the ``points'' $f\tmorX{f}$, $f\tmorXX{f}$ are ``points'' of the same type (condition 3.(e)). Example \ref{Lrgtransport} offers a more detailed illustration of these conditions. Semantically they will allow us to define each $\tau_f$ as the transport function induced by a path on $O$ for a type family $A$ over $O$.

Finally, the fact that we do not allow non-logical sorts to be contractible (i.e. of $h$-level $0$) is because we have no use for families of contractible types (over some other type). 
The fact that we do not allow isomorphism sorts and transport structure to be contractible is because we have no use for contractible identity types since, as we said above, we want to have isomorphism sorts and transport structure over types that are at least of $h$-level $2$.
On the other hand, the fact that we \emph{do} allow reflexivity sorts to be of $h$-level 0 is, first, because we do have a use for stating the inhabitation of an isomorphism sort of $h$-level 1 and indeed because the reflexivity predicate for such an isomorphism sort will, semantically, correspond to a contractible type and, second, because making this (somewhat artificial expression) allows for a more uniform presentation of the deductive system $\Diso$ below since it will allow us to state a single ``$\Idtype$-introduction'' rule.
\end{remark}

\begin{notation}
We will usually suppress explicit mention of $K$ in the $s$, $t$ and $\rmor$ maps.
We will also often drop explicit mention of $h$ in $(\L,h)$ and refer to an $h$-signature simply by its associated  inverse category $\L$. 
When displaying an $h$-signature we will usually write the $h$-level of each sort on the left of the displayed category, as illustrated by the examples below, and we will usually omit writing out reflexivity predicates of $h$-level $0$.
\end{notation}

%We now offer several examples of $h$-signatures which we will also use to illustrate the constructions in the rest of the paper.

\begin{exam}
The $h$-signature $\L_{2,1} = (\Lgraph, h_{2,1})$ is defined as follows
\[
\xymatrix{
1 &A \ar@/_/[d] \ar@/^/[d]\\
2  &O
}
\]
Its underlying FOLDS signature is $\Lgraph$. Note that $\L_{2,1}$ is not, for example, isomorphic (as we make more precise below) to the $h$-signature $\L_{\infty,1} = (\Lgraph, h_{\infty,1})$ given by
\[
\xymatrix{
1 &A \ar@/_/[d] \ar@/^/[d]\\
\infty  &O
}
\]
The $h$-level, in other words, is extra structure on FOLDS signatures. 
There will be non-isomorphic $h$-signatures whose underlying FOLDS signatures are isomorphic.
%There will be isomorphic finite inverse categories that are not isomorphic as $h$-signatures.
\end{exam}

\begin{exam}
Similarly, $\Lrg$ is the underlying FOLDS signature of several $h$-signatures. The following $h$-signature 
\[
\xymatrix{
1 &I \ar[d]^{i} \\
2 &A \ar@/^/[d]^{d} \ar@/_/[d]_{c} \\
\infty &O
}
\]
would be appropriate (semantically) for talking about reflexive graphs on general types (i.e. types of any $h$-level). If we want to talk about reflexive graphs on $h$-sets then it would be appropriate to let $h(O)= 2$, i.e. to let $O$ be of $h$-level $2$.
\end{exam}

\begin{exam}\label{globLrg}
We can define $\glob{\Lrg}$ as the following $h$-signature
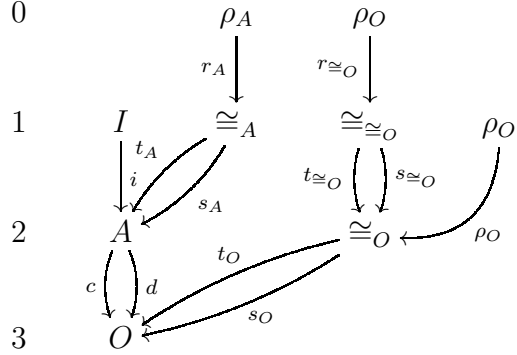
\[
\xymatrix{
0 & & \rsort_A \ar[d]_{\rmor_A} &\rsort_O \ar[d]_{\rmor_{\eqsort_O}} \\
1 &I \ar[d]^{i} &\eqsort_A \ar@/^/[ld]^{s_A} \ar@/_/[ld]_{t_A} &\eqsort_{\eqsort_O} \ar@/^/[d]^{s_{\eqsort_O}} \ar@/_/[d]_{t_{\eqsort_O}} &\rsort_O \ar@/^20pt/[ld]^{\rsort_O}\\
2 &A \ar@/^/[d]^{d} \ar@/_/[d]_{c} & &\eqsort_O \ar@/^/[lld]^{s_O} \ar@/_/[lld]_{t_O} \\
3 &O
}
\]
subject to the usual relation on $\Lrg$.
By definition it is also subject to the following relations:
\begin{align*}
s_O \circ s_{\eqsort_O} = s_O \circ t_{\eqsort_O} &, t_O \circ s_{\eqsort_O} = t_O \circ t_{\eqsort_O} \\
t_O \circ \rmor_O = s_O \circ \rmor_O &, c \circ t_A = c \circ s_A \\
t_{\eqsort_O} \circ \rho_O = s_{\eqsort_O} \circ \rmor_{\eqsort_O}&, t_A \circ \rmor_A = s_A \circ \rmor_A \\
d \circ t_A &= d \circ s_A \\
\end{align*}
This signature can be thought of as $\Lrg$ ``completed with respect to equality'' according to its $h$-level, i.e. we have added those equality and reflexivity sorts that the $h$-level of the non-logical sorts determines is non-trivial.
We will make this ``completion'' process more precise below, through the ``globular completion monad'' in Definition \ref{globcompmonad}.
%One can check that these relations and the way in which the isomorphism sorts and reflexivity sorts appear in $\Lrg^=$ make it a well-formed $h$-signature.
\end{exam}

\begin{exam}
A crucial feature of our formalism is that it allows us to express properties and impose structure on the isomorphism sorts themselves.
Semantically this will allow us to be able to express axioms and therefore define theories satisfying conditions relating to the inhabitants of isomorphism sorts (e.g. that they are in bijective correspondence with ``isomorphisms'' as defined over $\Lcat$).
For example, consider the following $h$-signature, extending $\Lrg^{\cong}$
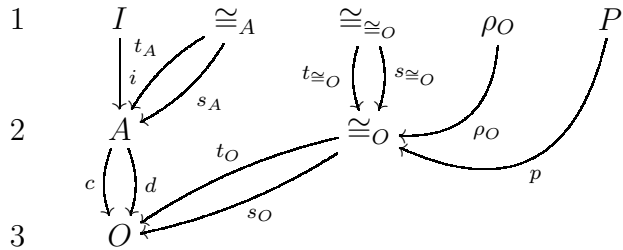
\[
\xymatrix{
%0 & & \rsort_A \ar[d]_{\rmor_A} &\rsort_O \ar[d]_{\rmor_{\eqsort_O}} \\
1 &I \ar[d]^{i} &\eqsort_A \ar@/^/[ld]^{s_A} \ar@/_/[ld]_{t_A} &\eqsort_{\eqsort_O} \ar@/^/[d]^{s_{\eqsort_O}} \ar@/_/[d]_{t_{\eqsort_O}} &\rsort_O \ar@/^20pt/[ld]^{\rsort_O} & P \ar@/^35pt/[lld]^{p}\\
2 &A \ar@/^/[d]^{d} \ar@/_/[d]_{c} & &\eqsort_O \ar@/^/[lld]^{s_O} \ar@/_/[lld]_{t_O} \\
3 &O
}
\]
%We will put this idea to use in axiomatizing univalent categories (Proposition \ref{unicatprop}).
Intuitively, we have added a one-place predicate on the (first-level) isomorphism sort of $O$. Extensions of $h$-signatures that already contain isomorphism sorts thus allow us to express properties and impose structure on the isomorphism sorts themselves. 
%This is the reason we consider them.
\end{exam}

\begin{exam}\label{Lrgtransport}
The following $h$-signature is useful for expressing properties of a ``reflexive graph'' on a groupoid $O$ together with properties of the ``transported'' sets of edges on any two points of the groupoid connected by a path:
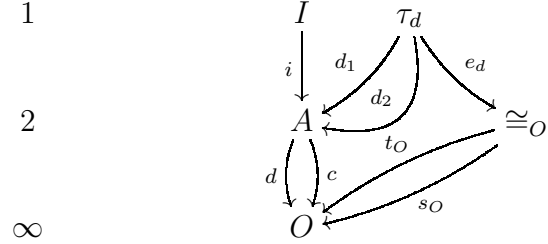
\[
\xymatrix{
1 & & &I \ar[d]_{i}   &\tsort_d \ar@/^/[ld]_{d_1} \ar@/^25pt/[ld]_{d_2} \ar@/_/[rd]^{\tmorpath{d}}  \\
2& & & A \ar@/_/[d]_{d} \ar@/^/[d]^{c} & &\eqsort_O \ar@/^/[lld]^{s_O} \ar@/_/[lld]_{t_O} \\
\infty & & & O \\
}
\]
In short, this is the signature for reflexive graphs with transport. 
As described in Remark \ref{hsigexplanation}, note that the relation $cd_1=cd_2$ ensures that we are ``transporting'' between ``types'' that \emph{only} (possibly) differ in the ``domain'' position $d$ but do not differ in the ``codomain'' position $c$, i.e. we can think of $\tsort_d$ as defining a function(al relation)
\[
A(x,z) \rightarrow A(y,z)
\]
for the ``type family''
$
A(-,z) \colon O \rightarrow \U
$
\end{exam}

\begin{exam}\label{Lprecateg}
The following $h$-signature $\Lprecat$ is useful for formalizing the theory of precategories:
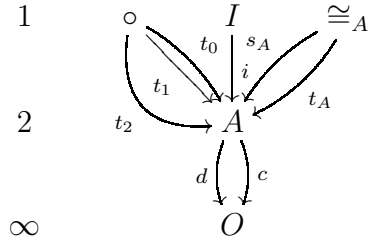
\[
\xymatrix{
1 &\circ \ar@/^5pt/[rd]^{t_0} \ar[rd]_{t_1} \ar@/_20pt/[rd]_{t_2} & I \ar[d]^{i} & \eqsort_A \ar@/_/[ld]_{s_A} \ar@/^/[ld]^{t_A} \\
 2 & & A \ar@/_/[d]_{d} \ar@/^/[d]^{c} \\
 \infty & & O \\
}
\]
where we have the relations of $\Lcat$.
% as well as those demanded from the presence of the isomorphism sort, together with
%\[
%d\circ t_0 = d\circ t_2, c\circ t_0 = d\circ t_1, c\circ t_1 = c\circ t_2 
%\]
\end{exam}

\begin{exam}\label{Lucatcanonical}
The following $h$-signature $\Lucat$ is useful for formalizing the theory of univalent categories:
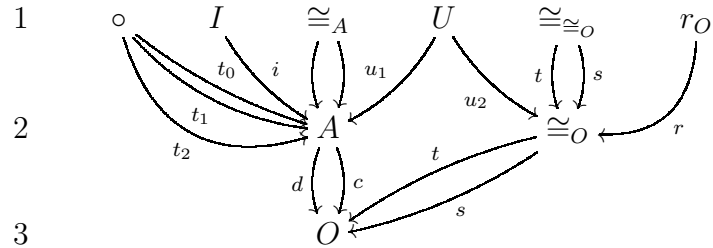
\[
\xymatrix{
1 &\circ \ar@/_5pt/[rrd]^{t_0} \ar@/_10pt/[rrd]_{t_1} \ar@/_25pt/[rrd]_{t_2} & I \ar@/_5pt/[rd]^{i}  &\eqsort_A \ar@/^/[d] \ar@/_/[d] &U \ar@/^/[ld]_{u_1} \ar@/_/[rd]_{u_2} &\eqsort_{\eqsort_O} \ar@/^/[d]^{s} \ar@/_/[d]_{t}  &r_O \ar@/^20pt/[ld]^{\rmor} \\
2& & & A \ar@/_/[d]_{d} \ar@/^/[d]^{c} & &\eqsort_O \ar@/^/[lld]^{s} \ar@/_/[lld]_{t} \\
 3 & & & O \\
}
\]
where we have the same relations as $\Lprecat$ together with the relation
\[
t \circ u_1 = s \circ u_2
\]
The sort $U$ can be thought of as allowing us to express a function between paths and arrows, and indeed to express the defining axiom for univalent categories.
(The same could also be accomplished through the transport relation $\tau_{s_O}$.)
\end{exam}

%\begin{exam}
%
%$\Ltycat$
%\end{exam}

\begin{defin}[Category of $h$-signatures]\label{hsigcategory}
We let $\hSig$ be the category whose objects are the $h$-signatures and whose morphisms $I \colon (\L,h) \rightarrow (\L',h')$ are given by $\FOLDS(\Ob,\Mor)$-morphisms $I \colon \L \rightarrow \L'$ that  satisfy the additional condition $h(K) \leq h'(I(K))$.
\end{defin}

\begin{terminology}
We will refer to morphisms in $\hSig$ as \emph{$h$-morphisms}.
\end{terminology}

\begin{remark}
The condition $h(K) \leq h'(I(K))$ is imposed with the semantic fact that the $h$-level is ``upwards cumulative'' (i.e. that a type of $h$-level $n$ is also a type of $h$-level $m \geq n$) but not the other way around (i.e. a type of $h$-level $n$ is not necessarily a type of $h$-level $m \leq n$). 
%This condition will also allow us to prove, later on, that isomorphic $\FOLiso$-signatures have identical types of (what we will call) homotopy $\L$-structures.
\end{remark}

%\begin{notation}
%We will conflate an $h$-morphism $I \colon (\L,h) \rightarrow (\L',h')$ with its underlying functor $I \colon \L \rightarrow \L'$.
%%, although of course not every functor $\L \rightarrow \L'$ will be a morphism of $h$-signatures.
%\end{notation}

\begin{defin}\label{Gfunctor}
We define $G \colon \hSig \rightarrow \hSig$ to be the following functor: 
\begin{itemize}
\item On objects $G$ takes $(\L,h) \in \hSig$ to the $h$-signature $G(\L,h) = (G(\L),G(h))$ with 
\[
\ob{G(\L)} = \ob{\L} \cup \setdefinition{\eqsort_K}{K \in \L, h(K) \geq 2} \cup \setdefinition{\rsort_K}{K \in \L, h(K) \geq 2}
\]
\item On arrows, given $I \colon (\L,h) \rightarrow (\L',h')$ we define $G(I) \colon G(\L) \rightarrow G(\L')$ as the following functor:
\begin{itemize}
\item On objects we have $G(I)\vert_{\ob{\L}} = I$ and whenever $\eqsort_K, \rsort_K \in \ob{G(\L)} \setminus \ob{\L}$ we set $G(I)(\eqsort_K) \eqdef \eqsort_{I(K)}$ and $G(I)(\rsort_K)=\rsort_{I(K)}$.
\item On arrows we once again have $G(I)\vert_{\ob{\L}} = I$ and whenever $s_K, t_K, \rmor_K \in \morph{G(\L)} \setminus \morph{\L}$ we set $G(I)(s_K) = s_{I(K)}$, $G(I)(t_K) = t_{I(K)}$  and $G(I)(\rsort_K)=\rsort_{I(K)}$. 
\end{itemize}
\end{itemize}
\end{defin}

\begin{lemma}\label{Gwelldefined}
$G$ is well-defined.
\end{lemma}
\begin{proof}
Firstly, we note that $G(\L,h)$ is uniquely determined for any $(\L,h)$ since by Definition \ref{hsignatures} the objects $\eqsort_K$ and $\rsort_K$ can only appear in an $h$-signature in a unique way and $G(\L)$ contains no more non-logical sorts than $\L$ and their $h$-level is determined by the $h$-level of their associated $K$.
It remains to check that for $I \colon (\L,h) \rightarrow (\L',h')$ we have that $G(I) \colon G(\L) \rightarrow G(\L')$ is indeed an $h$-morphism, i.e. that $I$ does not decrease $h$-level. 
For any $K \in \L$ this is immediate since $I$ is assumed to be an $h$-morphism.
Now let $\eqsort_K \in \ob{G(\L)} \setminus \ob{\L}$. Then we have
\[
\glob{h} (\eqsort_K) = h(K) -1 \leq h'(I(K)) -1 = \glob{h'}(\eqsort_{I(K)})
\]
The same argument works for $\rsort_K$.
\end{proof}

\begin{exam}
Let $\Lrg$ be as in Example \ref{Lrgcanonical}. Then the action of $G$ on $\Lrg$ is given by
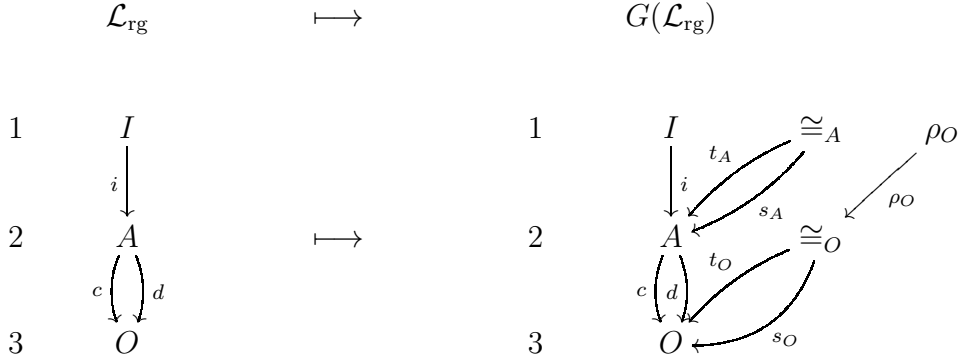
\[
\xymatrix{
&\Lrg & &\longmapsto & & &G(\Lrg) \\
1 &I \ar[d]_{i} & & & &1 &I \ar[d]^{i} &\eqsort_A \ar@/^/[ld]^{s_A} \ar@/_/[ld]_{t_A} &\rsort_O \ar[ld]^{\rsort_O} \\
2 &A \ar@/^/[d]^{d} \ar@/_/[d]_{c} & &\longmapsto & &2 &A \ar@/^/[d]_{d} \ar@/_/[d]_{c} &\eqsort_O \ar@/^15pt/[ld]^{s_O} \ar@/_/[ld]_{t_O} \\
3 &O & & & &3 &O
}
\]
\end{exam}

\begin{defin}[Globular Completion Monad]\label{globcompmonad}
We define the \textbf{globular completion monad} as the triple $\Gmonad = \langle G, \mu_G, \eta_G \rangle$ where:
\begin{itemize}
\item $G$ is the functor defined in Definition \ref{Gfunctor}
\item $\mu_G (\L,h)$ is the ``contraction'' $h$-morphism that takes 
\begin{align*}
\rsort_{\rsort_K} \:\:  &\mapsto \:\: \rsort_K \\
\eqsort_{\rsort_K} \:\: &\mapsto \:\: \rsort_K \\
\rsort_{\eqsort_K} \:\: &\mapsto \:\: \eqsort_K \\
\eqsort_{\eqsort_K} \:\: &\mapsto \:\: \eqsort_K
\end{align*}
\item $\eta_G (\L,h)$ is the inclusion $(\L,h) \hookrightarrow G(\L,h)$
\end{itemize}
\end{defin}

\begin{lemma}
$\Gmonad$ is a monad.
\end{lemma}
\begin{proof}
Straightforward from the definitions.
\end{proof}

\begin{defin}[$\FOLDSeq$]
We write $\FOLDSeq$ for the category of free $\Gmonad$-algebras, i.e. the Kleisli category for the monad $\Gmonad$.
The objects of $\FOLDSeq$ are the \textbf{$h$-signatures with equality}.
We write $$\glob{(-)} \colon \hSig \rightarrow \FOLDSeq$$ for the free algebra functor associated to the monad $\Gmonad$.
\end{defin}

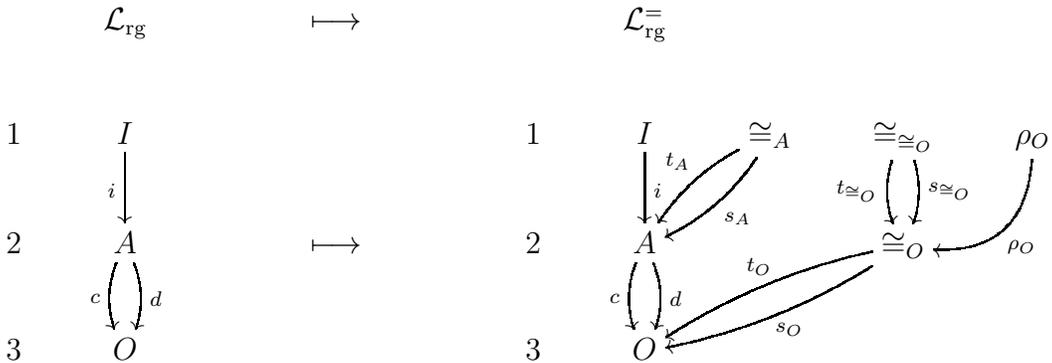
\begin{exam}
%Let $\Lrg$ be as in Example \ref{Lrgcanonical}. Then the action of $\glob{(-)}$ on $\Lrg$ is given by
%1 &I \ar[d]^{i} &\eqsort_A \ar@/^/[ld]^{s_A} \ar@/_/[ld]_{t_A} &\eqsort_{\eqsort_O} \ar@/^/[d]^{s_{\eqsort_O}} \ar@/_/[d]_{t_{\eqsort_O}} &\rsort_O \ar@/^20pt/[ld]^{\rsort_O}\\
%2 &A \ar@/^/[d]^{d} \ar@/_/[d]_{c} & &\eqsort_O \ar@/^/[lld]^{s_O} \ar@/_/[lld]_{t_O} \\
%3 &O
\[
\xymatrix{
&\Lrg & &\longmapsto & & &\glob{\Lrg} \\
1 &I \ar[d]_{i} & & & &1 &I \ar[d]^{i} &\eqsort_A \ar@/^/[ld]^{s_A} \ar@/_/[ld]_{t_A} &\eqsort_{\eqsort_O} \ar@/^/[d]^{s_{\eqsort_O}} \ar@/_/[d]_{t_{\eqsort_O}} &\rsort_O \ar@/^20pt/[ld]^{\rsort_O}\\2 &A \ar@/^/[d]^{d} \ar@/_/[d]_{c} & &\longmapsto & &2 &A \ar@/^/[d]^{d} \ar@/_/[d]_{c} & &\eqsort_O \ar@/^/[lld]^{s_O} \ar@/_/[lld]_{t_O} \\
3 &O & & & &3 &O
}
\]
This explains the notation in Example \ref{globLrg}.
\end{exam}

\newcommand\taufunctor[1]{\ts{#1}}

\begin{defin}\label{taufunctor}
We define a functor $\taufunctor{(-)} \colon \FOLDSeq \rightarrow \hSig$ as follows: 
\begin{itemize}
\item On objects $T$ takes $(\L,h)$ to the $h$-signature $\taufunctor{(\L,h)} = (\ts{\L},\ts{h})$ where $$\ob{\ts{\L}} = \ob{\L} 
\cup \setdefinition{\tsort_f}{f\colon A \rightarrow K, f \in \toplevel{A}, h(K) \geq 2}$$ 
\item On arrows, given $I \colon (\L,h) \rightarrow (\L',h')$ we define $T(I) \colon \ts{\L} \rightarrow \ts{\L'}$ as the following functor:
\begin{itemize}
\item On objects we have $T(I)\vert_{\L} = I$ and if $\tsort_f \in \ob{\ts{\L}} \setminus \ob{\L}$ we set $T(I)(\tsort_f) = \tsort_{I(f)}$.
\item On arrows we once again have $T(I)\vert_{\L} = I$ and whenever $\tmorX{f}, \tmorXX{f}, \tmorpath{f} \in \morph{\ts{\L}} \setminus \morph{\L}$ we set $T(I)(\tmorX{f}) = \tmorX{I(f)}$, $T(I)(\tmorXX{f}) = \tmorXX{I(f)}$  and $T(I)(\tmorpath{f})=\tmorpath{I(f)}$. 
\end{itemize}
\end{itemize}
\end{defin}

\begin{lemma}
$\taufunctor{(-)}$ is well-defined.
\end{lemma}
\begin{proof}
Just as in the proof of Lemma \ref{Gwelldefined} the choice of the objects of $(\ts{\L},\ts{h})$ determines a unique $h$-signature and that each $T(I)$ is an $h$-homomorphism follows easily.
\end{proof}

\begin{exam}
We have
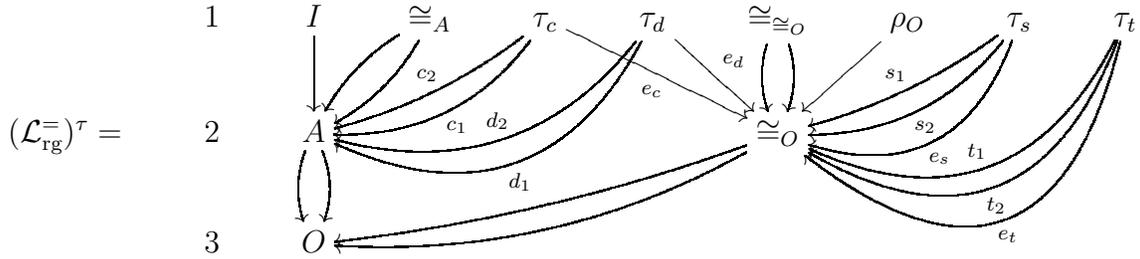
\[
\xymatrix{
 &1 &I \ar[d] &\eqsort_A \ar@/^/[ld] \ar@/_/[ld]
 %%Transport Structure
 &\tsort_c \ar@/^15pt/[lld]^{\tmorX{c}} \ar@/^/[lld]_{\tmorXX{c}} \ar[rrd]_{\tmorpath{c}} &\tsort_d \ar@/^35pt/[llld]^{\tmorX{d}} \ar@/^25pt/[llld]_{\tmorXX{d}} \ar[rd]^{\tmorpath{d}} 
 &\eqsort_{\eqsort_O} \ar@/^/[d] \ar@/_/[d] &\rsort_O \ar[ld] 
%%Transport Structure
&\tsort_s \ar@/^/[lld]_{\tmorX{s}} \ar@/^15pt/[lld]^{\tmorXX{s}}  \ar@/^28pt/[lld]^{\tmorpath{s}} &\tsort_t \ar@/^37pt/[llld]_{\tmorX{t}} \ar@/^45pt/[llld]^{\tmorXX{t}} \ar@/^58pt/[llld]^{\tmorpath{t}}  \\
\ts{(\glob{\Lrg})}= &2 &A \ar@/^/[d] \ar@/_/[d] & & & &\eqsort_O \ar@/^/[lllld] \ar@/^15pt/[lllld] \\
&3 &O
}
\]
where we have omitted the names of those arrows that are already in $\glob{\Lrg}$ for readability.
\end{exam}

\begin{defin}[Transport Structure Endofunctor]\label{Tfunctor}
We define the \textbf{transport structure endofunctor} $T \colon \FOLDSeq \rightarrow \FOLDSeq$ as the composite $\glob{(-)} \circ \taufunctor{(-)}$.
\end{defin}

%\begin{defin}[Transport Structure Monad]\label{globcompmonad}
%We define the \textbf{transport structure monad} as the triple $\Tmonad = \langle T, \mu_T, \eta_T \rangle$ where:
%\begin{itemize}
%\item $T$ is the functor defined in Definition \ref{Tfunctor}
%\item $\mu_T (\L,h)$ is the ``contraction'' $h$-morphism that for every top-level $f \colon A \rightarrow K$ takes 
%\begin{align*}
%\tsort_{gf_2}=\tsort_{gf_1} \:\:  &\mapsto \:\: \tsort_f \:\:\:\:\: (\forall g \neq f) \\
%\tsort_{s_K\tmorpath{f}}=\tsort_{ff_1} \:\: &\mapsto \:\: \tsort_f \\
%\tsort_{t_K\tmorpath{f}}=\tsort_{ff_2} \:\: &\mapsto \:\: \tsort_f \\
%%\eqsort_{\eqsort_K} \:\: &\mapsto \:\: \eqsort_K
%\end{align*}
%with the obvious action on the corresponding arrows.
%\item $\eta_T$ is the inclusion $(\L,h) \hookrightarrow T(\L,h)$
%\end{itemize}
%\end{defin}

%\begin{lemma}
%$\Tmonad$ is a monad.
%\end{lemma}
%\begin{proof}
%Straightforward from the definitions.
%\end{proof}

%\begin{defin}[$\FOLDSeq$]
%We write $\FOLDSeqtrans$ for the category of free $\Tmonad$-algebras.
%The objects of $\FOLDSeqtrans$ are the \textbf{$h$-signatures with transport structure}.
%\end{defin}
%
%\begin{defin}[$\Fmonad$]
%We define $\Fmonad$ to be the triple $\langle TG, \mu_T\mu_G, \eta_T\eta_G \rangle$, i.e. the composite of the monads $\Gmonad$ and $\Tmonad$.
%\end{defin}

\newcommand\FOLisom[1]{\FOLiso^{#1}}

\begin{terminology}
For a category $\C$ and an endofunctor $F \colon \C \rightarrow \C$ an object $a$ of $\C$ is a \emph{fixed point} of $F$ if $F(a)=a$.
\end{terminology}

\begin{defin}\label{FOLisofinite}
We write $\FOLisom{\text{fin}}$ for the full subcategory of $\FOLDSeq$ consisting of the fixed points of the endofunctor $T$.
The objects of $\FOLisom{\text{fin}}$ are the \textbf{signatures of first-order logic with isomorphism of finite height} or \textbf{finite $\FOLiso$-signatures}.
Similarly we write $\FOLisom{m}$ for the for the full subcategory of $\FOLDSeq^m$ consisting of the fixed points of the endofunctor $T$.
The objects of $\FOLisom{m}$ are the \textbf{signatures of first-order logic with isomorphism of height $m$} or \textbf{$\FOLiso$-$m$-signatures}.
\end{defin}

\begin{exam}
We have $TT(\glob{\Lrg})=T(\glob{\Lrg})$ and therefore $\glob{\Lrg}$ is a $\FOLiso$-3-signature.
\end{exam}

We now wish to extend Definition \ref{FOLisofinite} to $h$-signatures of possibly infinite height, and thus finally arrive at the full definition of a $\FOLiso$ signature.

\begin{defin}
For any $\L$ in $\FOLDSeq$ we let let $J_\L$ be the functor
\[
\xymatrix{
\omega \ar[rrrr]^{J_{\L}} & & & &\FOLDSeq \\
n   & &\mapsto & & T^n \L \\
m \ar[u]^{\leq} & &\mapsto & & T^m\L \ar@{>->}[u]
}
\]
where $\omega$ is the usual poset category on the ordinal $\omega$, $T^i\L$ is the $i$-fold application of the functor $T$ to $\L$ and each map $T^m \L \hookrightarrow T^n \L$ is the obvious inclusion.
\end{defin}

\newcommand\repletion[1]{#1^{\cong}}
\def\colim{\underset{\longrightarrow}{\text{lim}}}
\def\colime{\underset{\longrightarrow}{\text{\emph{lim}}}}

\begin{lemma}
The sequential colimit $\colime J_\L$ of $\L$ exists in $\FOLDSeqe$ for any $\FOLDSeqe$-signature $\L$.
\end{lemma}
\begin{proof}
We construct $\colim J_\L$ as the union of all the $T^i\L$ over $i \in \mathbb{N}$, i.e. we set
\begin{align*}
\ob{(\colim J_\L)} &\eqdef \underset{i \in \mathbb{N}}{\bigcup} \ob{(T^i\L)} \\
\morph{(\colim J_\L)} &\eqdef \underset{i \in \mathbb{N}}{\bigcup} \morph{(T^i\L)}
\end{align*}
It is then clear by the definition of $h$-signatures that this data defines a unique $h$-signature since each logical sort can occur in a unique way and all the $T^i\L$ contain the same non-logical sorts, namely $\NL{\L}$. Clearly $\colim J_\L$ is a free $\Gmonad$-algebra and each $T^i\L$ embeds into it in the obvious way, thus giving us a cocone which is immediately seen to be universal.
\end{proof}

\begin{defin}[$\FOLiso$-signatures]\label{FOLisosignatures}
For any $\L$ in $\FOLDSeq$ we define its \textbf{associated $\FOLiso$-signature $\repletion{\L}$} by 
\begin{equation}
\repletion{\L} \eqdef \colim J_\L
\end{equation}
The assignment $\L \mapsto \repletion{\L}$ defines an endofunctor
\[
\mathcal{T} \colon \FOLDSeq \longrightarrow \FOLDSeq
\]
and we write $\FOLisocat$ for the category of algebras of the endofunctor $\mathcal{T}$.
The objects of $\FOLisocat$ are the \textbf{signatures of first-order logic with isomorphism} or $\FOLiso$\textbf{-signatures}. 
The morphisms of $\FOLisocat$ are the \textbf{$\FOLiso$-morphisms} or \textbf{logical $h$-morphisms}.
\end{defin}

\newcommand\forget[1]{\vert #1 \vert}

We can now summarize the situation in the following diagram of functors, with $U_1, U_2, U_3$ the obvious forgetful functors.
\[
\xymatrix{
 & & &\FOLDSeq \ar[ld]_{\ts{(-)}} \ar[rd]^{T} \\
\FOLDS (\Ob, \Mor) &  &\hSig \ar@(ul,ur)^G  \ar@/^/[rr]^{\glob{(-)}} \ar[ll]^{\:\:\:\: \quad \quad U_1} & &\FOLDSeq \ar@(ul,ur)^{\mathcal{T}} \ar@{=}[d] \ar@/^/[ll]^{U_2} & &\FOLisocat \ar@{=}[d] \ar[ll]^{U_3} \\
 & & & &\textbf{Kl}(\Gmonad) & &\mathcal{T}\textbf{-Alg} \\
}
\]
% and $\gl$ is the left adjoint to $U_2$ taking an $h$-signature to its associated free $\Gmonad$-algebra.

\begin{notation}
For a given $\FOLiso$-signature $\L$ we will write $\forget{\L}$ for its image under the composite $U_1U_2U_3$ of the corresponding forgetful functors.
We write $TT_\L$ and $\LTT_\L$ for $TT_{\forget{\L}}$ and $\LTT_{\forget{\L}}$. 
\end{notation}

\begin{terminology}
We call $\forget{\L}$ the \emph{underlying FOLDS signature} of a $\FOLiso$-signature $\L$.
We say that a $\FOLiso$-signature is \emph{essentially finite} if its underlying FOLDS signature is essentially finite. 
\end{terminology}

%\begin{notation}
%For a given $h$-signature $\L = (\L,h)$ we will write $\L^{\cong}$ for the $\FOLiso$-signature obtained as the free $\mathcal{F}$-algebra on $\L$.
%\end{notation}
%
%\begin{terminology}
%From now on we will refer even to plain $h$-signatures as $\FOLiso$-signatures, and will generally assume that any $h$-signature can be replaced by its associated $\FOLiso$-signature, i.e. its image under the free algebra functor of the monad $\Fmonad$. This terminology is justified by the fact that none of the constructions in the rest of this paper separate between $h$-signatures and $\FOLiso$-signatures, as indeed is made precise in Proposition \ref{hsigsameasfolisosig} below. We will call the signatures $\L^{\cong}$ originating from an $h$-signature $\L$ \emph{replete}.
%\end{terminology}

%\begin{remark}
%The only difference between the syntax of $\FOLiso$ and of FOLDS at the level of signatures are the extra logical sorts $\eqsort_K, \rsort_K$ and $\tsort_f$ and their associated logical arrows. That these symbols play a different role from non-logical sorts and arrows will become clear in the following section where they will receive a fixed denotation. 
%For the time being we can use the fact that 
%Since the definition of the syntax of FOLDS for a given FOLDS signature $\L$ relies only on the (properly ordered) inverse category (over $(\Ob, \Mor)$) structure of $\L$ in order to define the syntax of $\FOLiso$.
%\end{remark}
%(A mechanism by which this can be achieved formally is introduced in the next section.)

With this in mind we can use the fact that every $\FOLiso$-signature $\L$ has an underlying FOLDS$(\Ob,\Mor)$-signature to define the syntax of $\FOLiso$ as the syntax of its underlying FOLDS$(\Ob,\Mor)$-signature. We thus arrive at the desired definition of the syntax of $\FOLiso$.

\begin{defin}[Syntax of $\FOLiso$]
Let $\L$ be a $\FOLiso$-signature. The \textbf{contexts, sorts, context morphisms, formulas and sequents of $\L$} are defined as the contexts, sorts, context morphisms, formulas and sequents of $\LTT_{\L}$.
\end{defin}

\newcommand\FOLdepth[1]{d(#1)}
%
%\begin{defin}[Depth]\label{depth}
%We define the \textbf{depth} of $K \in \L$ as follows:
%\[
%\FOLdepth{K} = \left\{
%	\begin{array}{ll}
%		0  & \mbox{if $K \in \NL{\L}$ and $\not\exists f \in K \downarrow \L$ such that $K_f \in \Logical{\L}$} \\
%		d(O)  & \mbox{if $K \equiv \eqsort_{O}$ or $K \equiv \rsort_{O}$ or $K \equiv \tsort_{f}$ with $O=K_f$} \\
%		%0 & \mbox{if $K \in \Logical{\L}$ and $\underset{\NL{\L} \ni K' < K}{\text{sup}}d(K') = 0$} \\
%		\underset{f \in K \downarrow \L, K_f \in \Logical{\L}}{\text{sup}}d(K_f) +1  & \mbox{otherwise}
%	\end{array}
%\right.
%\]
%% $\FOLdepth{\L}$ of $\L$ inductively as follows:
%%\begin{itemize}
%%\item $\FOLdepth{\L} = 0$ if there is no non-logical arrow whose codomain is a logical sort.
%%\item $\FOLdepth{\L} = \underset{\textcolor{red}{\L' < \L}\:\:\:\:\:\:\:\:\:\:\:\:\:\:\:\:\:\:\:\:}{\text{sup}(\FOLdepth{\L'})} + 1$ if $\L$ has a non-logical arrow whose codomain is a logical sort.
%%\end{itemize}
%The \textbf{depth} $\FOLdepth{\L}$ of $\L$ is then defined as $\underset{K \in \L}{\text{sup}}\FOLdepth{K}$.
%\end{defin}
%
\newcommand{\depth}[1]{d(#1)}
\newcommand{\FOLisodh}[2]{\FOLiso^{#1,#2}}
%
%\begin{remark}
%The depth of a sort $K$ in $\L$ can be understood as measuring to what extent $K$ defines structure on a logical sort. In other words, $K$ has non-zero depth only if semantically it is understood as depending on certain logical sorts. For example if $K$ is a predicate on an identity type, picking out certain paths over others (e.g. the path $\mathtt{loop}$ in Example \ref{Lcircleexample} below). The following examples illustrate the notion.
%\end{remark}
%
%\begin{exam}
%The depth of $\Lrg$ is $0$. Similarly, the depth of $\Lrg^{\cong}$ is also $0$.
%\end{exam}
%
%\begin{exam}
%The depth of $\Lucat$ is $1$. This is because the non-logical sort $U$ has an arrow $u_2$ coming out of it whose codomain is a logical sort of (depth $0$).
%\end{exam}
%
%\begin{notation}
%It is useful to understand $\FOLiso$-signatures as organized according to depth and height. We write $\FOLisodh{n}{m}$ for the $\FOLiso$ signatures of height $n$ and depth $m$.
%We will refer to concepts and constructions over signatures in $\FOLisodh{n}{m}$ as \emph{$(n,m)$-logic}. 
%\end{notation}
%

\section{Homotopy Semantics of $\FOLiso$}\label{nlogSemantics}

We will define the semantics for $\FOLiso$ as an interpretation directly into the syntax of Homotopy Type Theory (HoTT). 
HoTT is here understood as intensional MLTT with $\Pitype$, $\Sigmatype$, $\Idtype$, $\singletontype$, $\zerotype$, $+$-types, a univalent universe $\U$ and propositional truncation $\truncated{-}$. 
%More generally, our semantics could be defined in any dependent type theory which supports a notion of $h$-levels and a univalent universe, or indeed directly in any categorical model of such a type theory carrying the appropriate extra structure. 
As usual we will loosely refer to HoTT as \emph{type theory}.

We will follow the general pattern that we followed when defining the interpretation of FOLDS into MLTT.
%We will follow the notation of \cite{HTT} closely, but let us be clear on certain notational points
%To define the semantics $\FOLiso$ requires two steps. 
We first define for any essentially finite $\FOLiso$-signature $(\L,h)$ a notion of an $\L$-structure by describing a type expression $\HStruc{\L,h}$ of ``homotopy $\L$-structures'' (Definition \ref{HLstructure}), proving that this type expression is  a well-formed type (Theorem \ref{hstrucwellformed}), and taking the terms of this type to be our notion of $\L$-structure (Definition \ref{hstructure}). The notions of interpretation, satisfaction, model, theory etc. are then defined similar to Section \ref{prelim} with some modifications to account for the newly-introduced logical symbols. 
%The main difference is with extending the proof system $\DFOLDS$ to do justice to this semantics, which we take up in the next Section.

%TODO Remark concerning the setting in which this section takes place.

%In addition, we will assume that $\L$ has a fixed logical ordering, in the sense of the following definition.

%\begin{defin}[Logical Order]\label{logicalorder}
%Let $\L$ be a $\FOLiso$-signature. A \textbf{logical order} on $\L$ is a proper order such that for any $K,K' \in \L$, $\depth{K} \leq \depth{K'} \Rightarrow K <_o K'$.
%\end{defin}

\def\PropU{\textbf{Prop}_{\mathcal{U}}}
\def\SetU{\textbf{Set}_{\mathcal{U}}}
\def\GpdU{\textbf{Gpd}_{\mathcal{U}}}

\def\transportrule{(\text{trans})}

\begin{notation}
%We will write write $\vert \vert A \vert\tran \vert$ for the propositional truncation of a type $A$ in $\mathcal{U}$.
For any $m \in \Natinfty$ we will write $\typesofhlevel{m}$ for the types of $h$-level $m$ in $\U$ with the convention that $\typesofhlevel{\infty} \eqdef \U$.
We will also use the more recognizable notation $\PropU, \SetU$ and $\GpdU$ for types of $h$-level $1,2$ and $3$ respectively.
In general, we will abuse notation and conflate a term $A \colon \typesofhlevel{m}$ with its underlying type (i.e. with its first projection).
%\textcolor{red}{For any type $A$ we will write $\mathtt{Id}^i_A$ for the $i^\text{th}$-iterated identity type. Thus, for example, $\mathtt{Id}^3_A(\alpha,\beta)(p,q)(x,y)$ stands for
%\[
%\mathtt{Id}_{\mathtt{Id}_{\mathtt{Id}(x,y)} (p,q)} (\alpha,\beta)
%\]
%}
%We will write $\times$ for a dependent sums ($\Sigmatype$-types) using the Agda-style notation $(x \colon A) \times (b \colon B(x))$ where appropriate and we will write $\rightarrow$ for dependent functions ($\Pitype$-types) again using Agda-style notation $(x \colon A) \rightarrow B$ where appropriate.
%For a given (ordered) set $S$ (almost exclusively of objects and arrows of a given $h$-signature) we will use the notation $\underset{s \in S}{\times}$ and $\underset{s \in S}{\rightarrow}$ to denote the sums and functions over all the expressions indexed in some way by $S$.
%We will use the notation ``$A \:\:(\Gamma)$'' where $A$ is a type expression and $\Gamma$ a context expression to indicate that variables in $\Gamma$ (may) appear in $A$.
%We will generally conflate a term judgment $A \colon \U$ with a typing judgment $A\:\textbf{Type}$ so as to be able to refer freely to contexts formed from types in a universe, e.g. given $A \colon \U$ we will write $x \colon A$ for the context that would more appropriately be written as $x \colon El (A)$ and similarly for any type in $\U$ that might depend on $A$.
In all other notational matters related to type theory we will follow the notation of \cite{HTT} closely.
In particular, we write $\Idtype_A (a,b)$ for the identity type of terms $a,b$ of $A$ and we write $\trans_p^{x.P} (t)$ for the (``covariant'') transport along $p$ of a(n appropriate) type family $P$ bound at some variable $x$ for a given term $t$. More precisely, $\trans_p^{x.P} (t)$ can be understood as being controlled by the following rule
\[
\inferrule
{
\Gamma \vdash t \colon P[a/x] \\\\
\Gamma \vdash p \colon \Idtype_A(a,b) \\\\
\Gamma, x \colon A \vdash P \: \TTisatype
}
{
\Gamma, x \colon A \vdash \trans_p^{x.P} (t) \colon P[b/x]
}
\transportrule
\]
which we will have ocassion to refer to in the proof of Lemma \ref{hstrucwellformedlemma} below.
\end{notation}

\begin{terminology}
We will use the HoTT terminology and call a type $A$ a \emph{mere proposition} if its identity types are contractible in the sense that we have an inhabitant of the type.
%\[
%\underset{x, y \colon A}{\Pitype} \:\: \underset{q \colon \Idtype_A (x,y)}{\Sigmatype} \:\: \underset{p \colon \Idtype_A (x,y)}{\Pitype} \:\: \Idtype_{\Idtype_A(x,y)} (p,q)
%\]
With this in mind what we will refer to as the \emph{universal property of propositional truncation} means that given a term $\eta \colon A \rightarrow P$ where $P$ is a mere proposition we obtain a term 
\[
\truncated{\eta} \colon \truncated{A} \rightarrow P
\]
and leave the details of the particular implementation of the propositional truncation operator $\truncated{-}$ implicit. 
\end{terminology}

%We now fix an essentially finite $h$-signature $\L = (\L,h)$. 
%All notions will be defined relative to such an $h$-signature. 

\newcommand{\varmodifier}[1]{\widetilde{#1}}

\begin{defin}[Homotopy $(\L,h)$-structure]\label{HLstructure}
Let $\L = (\L,h)$ be an essentially finite $h$-signature.
The type of \textbf{homotopy $(\L,h)$-structures} is given by the type expression
\begin{equation}
\HStruc{\L,h} \eqdef \underset{K \in \NL{\L}}{\Sigmatype} (K \colon T_K)
\end{equation}
where
\begin{equation}
T_K \eqdef \underset{f \in \pcosieve{K}{\L}}{\Pitype} (f \colon \inttype{K_f}) \:\: \typesofhlevel{h(K)} 
\end{equation}
and where the symbol $\inttype{K_f}$ is defined by induction on the level of $K_f$ as follows:
\begin{itemize}
%\item $\cod{f}$, if $l(K_f)=0$
\item $\Elements(\TTapp(*, K_f))$ if $\level{K_f}=0$
\item $\TTElapp{K_f(pf)_{p \in \pcosieve{K}{\L}}}$, if $K_f \in \NL{\L}$
\item $\Idtype_{\inttype{K_{sf}}} (sf, tf)$, if $K_f = \: \eqsort_{A}$ for some $A \in \L$
\item $\Idtype_{\inttype{K_{\rmor f}}} (\rmor f, \refl_{s\rmor f})$, if $K_f = \rsort_A$ for some $A \in \L$
\item $\Idtype_{\inttype{K_{\tmorXX{h}f}}} (\tmorXX{h}f, \trans_{e_hf}^{\varmodifier{h\tmorX{h}f}.\inttype{K_{\tmorX{h}f}}[\varmodifier{h\tmorX{h}f}/h\tmorX{h}f]} (\tmorXX{h}f))$, if $K_f \equiv \tau_h$ for some (top-level) $h \colon A \rightarrow K$ in $\L$ and where $\varmodifier{h\tmorX{h}f}$ denotes a variable distinct from $h\tmorX{h}f$ and $h\tmorXX{h}f$.
%where $A[h]$ denotes the  type expression
%$
%\TTElapp{\widehat{A}(gh_1f)_{g \in \pcosieve{A}{\L}}}
%$.
% $\lambda y \colon K. A(k)[y/f] \colon \U$.
\end{itemize}
\end{defin}

\begin{remark}
Note that the logical sorts of a $\FOLiso$-signature $\L$ will appear in $\HStruc{\L,h}$ only if they are the codomain of a non-logical arrow in $\L$. Since $\L$ is assumed essentially finite it will contain only a finite number of non-logical symbols, and therefore the type expression $\HStruc{\L,h}$ will always be of finite length (even if it contains non-logical sorts of $h$-level $\infty$).
\end{remark}

\begin{remark}
In the definition of $T_K$ in Definition \ref{HLstructure} we are suppressing  the information that the symbol $\inttype{K_f} (\mathbf{p}f)$ actually consists both of a type expression (displayed above) as well as a proof (expression) that that type is of a certain $h$-level. So, for example, strictly speaking we would have
\[
\widehat{\eqsort_A} (\mathbf{p}f) \eqdef \langle \Idtype_{\inttype{A_{sf}} (\mathbf{p}sf)} (sf, tf), \mathtt{isofhlevel}(h(A)-1) \rangle
\]
However, in order to not clutter the notation further, we will keep this information implicit, noting of course that it is important in the proofs of Theorem \ref{hstrucwellformed} below.
\end{remark}

\begin{notation}
For a given essentially finite $\FOLiso$ signature $\L$ we will write simply $\HStruc{\L}$ for the type of homotopy $\L$-structures of its underlying $h$-signature.
\end{notation}

\begin{prop}\label{Hstrucinvariant}
If $\L \cong \L'$ as $\FOLiso$-signatures then $\HStruc{\L} \equiv \HStruc{\L'}$.
\end{prop}
\begin{proof}
If $\L \cong \L'$ as $\FOLiso$-signatures then there is an isomorphism between the underlying $h$-signatures that preserves logical sorts and $h$-level, which implies exactly that the expressions $\HStruc{\L}$ and $\HStruc{\L'}$ are $\alpha$-equivalent.
\end{proof}

\begin{remark}
Note that Proposition \ref{Hstrucinvariant} is not true if in its statement we replace ``as $\FOLiso$-signatures'' with ``as $h$-signatures'' since we require an isomorphism that preserves the logical sorts.
\end{remark}

\begin{prop}\label{hstrucsameasfolisostruc}
$\HStruce{\L} \equiv \HStruce{\repletion{\L}}$
\end{prop}
\begin{proof}
Logical sorts appear in $\HStruc{\L,h}$ only if they are the codomain of non-logical sorts. But the objects in $\L^{\cong}$ that are not in $\L$ are those that are not the codomain of any non-logical sort.
\end{proof}

\begin{remark}
As Proposition \ref{hstrucsameasfolisostruc} makes precise, our semantics does not distinguish between an $h$-signature and (the underlying $h$-signature of) its associated $\FOLiso$-signature. This is to be expected since $\FOLiso$-signatures add to an $h$-signature logical sorts that one gets for free in type theory. The importance for our purposes of the full $\FOLiso$ structure is that it allows us to reason about constructs in HoTT (e.g. using the deductive system $\Diso$ defined below) externally to HoTT and in particular externally to any specific HoTT (e.g. cubical type theory, UniMath etc.). 
\end{remark}

\begin{exam}\label{Lrgcanonical}
Let $\Lrg^{321}$ denote the following $h$-signature:
\[
\Lrghsignature{3}{2}{1}
\]
Then by exactly the same reasoning as in Example \ref{Lrghstruc} we get
\[
\HStruc{\Lrg^{321}} = (O \colon \typesofhlevel{3}) \times (A \colon O \rightarrow O \rightarrow \typesofhlevel{2}) \times (I \colon (x \colon O) \rightarrow A(x,x) \rightarrow \typesofhlevel{1})
\]
The only thing that has changed from Example \ref{Lrghstruc} is that $O,A,I$ are now (functions landing in) the subuniverse of types of a certain $h$-level.
\end{exam}

\begin{exam}
Consider the $h$-signature $\Lucat$ from Example \ref{Lucatcanonical}. $\Lucat$ includes an arrow $u_2$ whose codomain is the isomorphism sort $\eqsort_O$.
This means that an identity type will appear in $\HStruc{\Lucat}$. To see how this works let us calculate the type expression $T_U$ associated to the sort $U$, recalling that $h(U)=1$.
\begin{align*}
T_U &= \underset{f \in \pcosieve{U}{\L}}{\Pitype} (f \colon \widehat{K_f}(\mathbf{p}f)) \rightarrow \typesofhlevel{h(U)} \\
&= (du_1, cu_1 \colon O) \rightarrow (u_1 \colon A (du_1, cu_1)) \rightarrow (u_2 \colon tu_2 \eqsort_O su_2) \rightarrow \typesofhlevel{1} \\
&\equiv (x, y \colon O) \rightarrow (f \colon A (x, y)) \rightarrow (p \colon x \eqsort_O y) \rightarrow \typesofhlevel{1} 
\end{align*}
Thus, $T_U$ can be understood as a propositional family (or relation) over a choice of any two points $x,y \colon O$, an ``arrow'' $f \colon A(x,y)$ and a ``path'' $p \colon x=_Oy$ from $x$ to $y$ in $O$. 
With additional axioms this relation can express e.g. that there is a bijection between ``arrows'' and ``paths'', as indeed we will do when we axiomatize univalent categories as a $\Lucat$-theory in the final section.
Overall, up to the obvious equivalence, we get:
\[
\HStruc{\Lucat} \simeq \HStruc{\Lprecat} \times (U \colon (x, y \colon O) \rightarrow (f \colon A (x, y)) \rightarrow (p \colon x \eqsort_O y) \rightarrow \typesofhlevel{1})
\]
\end{exam}

\begin{exam}
We now do an example that involves transport structure. Consider the following $h$-signature $\L_t$:
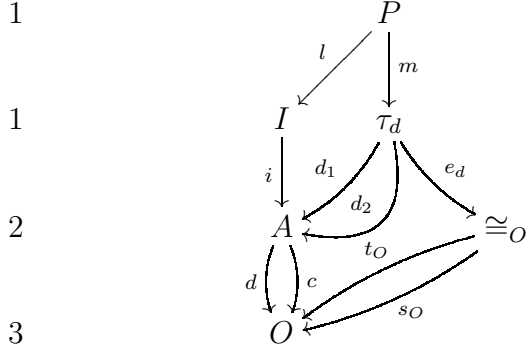
\[
\xymatrix{
1        & & &                   &P \ar[ld]_{l} \ar[d]^{m} \\
1 & & &I \ar[d]_{i}   &\tsort_d \ar@/^/[ld]_{d_1} \ar@/^25pt/[ld]_{d_2} \ar@/_/[rd]^{\tmorpath{d}}  \\
2& & & A \ar@/_/[d]_{d} \ar@/^/[d]^{c} & &\eqsort_O \ar@/^/[lld]^{s_O} \ar@/_/[lld]_{t_O} \\
3 & & & O \\
}
\]
The non-logical sorts of $\L_t$ other than $P$ are exactly as in $\Lrg^{321}$ above. So it remains to determine $T_P$. We have:
\begin{align*}
T_P  &= \underset{f \in \pcosieve{P}{\L}}{\Pitype} (f \colon K_f (\mathbf{p}f)) \:\: \typesofhlevel{1}  \\
	&= \Pitype (dil, c\tmorXX{d}m, s\tmorpath{d}m, t\tmorpath{d}m \colon O) (il \colon A(dil,cil)) (\tmorXX{d}m \colon A (d\tmorXX{d}m), c\tmorXX{d}m) (\tmorX{d}m \colon A (d\tmorX{d}m, c\tmorX{d}m) \\ & \quad \quad \: (\tmorpath{d}m \colon \Idtype_O (s\tmorpath{d}m, t\tmorpath{d}m)) (l \colon I(il,dil)) \\ 
	& \quad \quad \: (m \colon \Idtype_{A(d\tmorXX{d}m, c\tmorX{d}m)} (\tmorXX{d}m,  \trans_{\tmorpath{d}m}^{\varmodifier{d\tmorX{d}m}.A(\varmodifier{d\tmorX{d}m}, c\tmorX{d}m)} (\tmorX{d}m)) \:\: \typesofhlevel{1} \\
	&\equiv \Pitype (x, y, z, w \colon O) (f \colon A(x,x)) (h \colon A (w, y) (g \colon A (z, y)) \\ & \quad \quad \: (p \colon \Idtype_O (z, w)) (l \colon I(x,f)) \\ 
	& \quad \quad \: (m \colon \Idtype_{A(w, y)} (h,  \trans_{p}^{ v.A(v, y)} (g)) \:\: \typesofhlevel{1} 
\end{align*}
Intuitively a term of the type $T_P$ may be thought of as a relation that compares an ``identity arrow'' $f$ to the transport $h$ of an ``arrow'' $g$ along a path $p$. Overall, we get:
\[
\HStruc{\L_t} \equiv \HStruc{\Lrg^{321}} \times \underset{
\begin{subarray}  
\:x, y, z, w \colon O \\
f \colon A(x,x) \\
h \colon A (w, y) \\ 
g \colon A (z, y) 
\end{subarray}
}{\Pitype}
\quad
\underset{
\begin{subarray}
\:p \colon \Idtype_O (z, w) \\
l \colon I(x,f)
\end{subarray}
}{\Pitype}
\Idtype_{A(w, y)} (h,  \trans_{p}^{\lambda v.A(v, y)} (g)) \rightarrow \typesofhlevel{1}
\]
\end{exam}

\def\Lcircle{\L_{\text{circle}}}
\def\Tcircle{\T_{\text{circle}}}

\begin{exam}\label{Lcircleexample}
Consider the following $h$-signature $\Lcircle$:
\[
\xymatrix{
1   & &\mathtt{loop} \ar[ld]^m \ar[dd]^l  \\
1   &\mathtt{base} \ar[dd]^b \\
\infty &     &\eqsort_O \ar@/_/[ld]_s \ar@/^/[ld]^t \\
\infty &O
}
\]
with the relations $sl=tl=bm$. This $h$-signature can be thought of as encoding the type of circles, i.e. of types together with a preferred point (picked out by the ``predicate'' $\mathtt{base}$) and a preferred loop on that point (picked out by the ``predicate'' $\mathtt{loop})$. 
The reason for the arrow $m \colon \mathtt{loop} \rightarrow \mathtt{base}$ is in order to ensure that the source (and target) $x$ of the path $p$ that we pick out by $\mathtt{loop}$ is also such that it satisfies $\mathtt{base}(x)$.
The corresponding type of homotopy $\Lcircle$-structures is given by:
\begin{align*}
\HStruc{\Lcircle} = &(O \colon \U) \times (\mathtt{base} \colon O \rightarrow \typesofhlevel{1}) \times \\&(\mathtt{loop} \colon (x \colon O) \rightarrow \mathtt{base}(x) \rightarrow \mathtt{Idtype}_O(x,x) \rightarrow \typesofhlevel{1}) \\
%\simeq &(O \colon \U) \times (\mathtt{base} \colon O) \times (\mathtt{loop} \colon \Idtype_{O} (\mathtt{base},\mathtt{base}))
\end{align*}
\end{exam}

%Although it is easy to see for specific examples,
%we would now like to know that the type expressions $\HStruc{\L,h}$ are indeed well-formed types for arbitrary $h$-signatures $(\L,h)$.

\newcommand{\El}[1]{\textbf{El}(#1)}

\begin{lemma}\label{hstrucwellformedlemma0}
Let $\L$ be an $h$-signature and for any $f \in \morph{\L}$ let $\inttype{K_f}$ be as in Definition \ref{HLstructure}. Then for $f,g$ such that $\level{K_f}=\level{K_g}=n$ we have
$\inttype{K_f} \equiv \inttype{K_g}$ iff $K_f=K_g$ and $(pf)_{p \in \pcosieve{K_f}{\L}} = (pg)_{p \in \pcosieve{K_g}{\L}}$.
\end{lemma}
\begin{proof}
We do the right-to-left direction and leave the other direction (which we do not require below) to the reader.
We proceed by induction on the level $n$ of $K_f$ and $K_g$. Assume $\level{K_f}=\level{K_g}=0$. 
Then if 
\begin{equation}\label{blurgy}
K_f=K_g
\end{equation}
we get
\begin{equation}
\inttype{K_f} \equiv \Elements(\TTapp(*,K_f)) \overset{(\ref{blurgy})}{\equiv} \Elements(\TTapp(*,K_g)) \equiv \inttype{K_g}
\end{equation}
Now assume the left-to-right implication holds for all $m<n$ and that $\level{K_f}=\level{K_g}=n$.
Then we take cases.
If $K_f=K_g \in \NL{\L}$ then we get $\inttype{K_f} \equiv \inttype{K_g}$ by the essentially the same argument as above.
If $K_f=K_g=\eqsort_A$ then we get $K_{sf} = K_{sg}$ and if $(\mathbf{p}f)=(\mathbf{p}g)$ then in particular  we get
\begin{equation}\label{IH1}
(\mathbf{q}sf) = (\mathbf{q}sg) \quad\quad sf=sg \quad\quad  tf=tg
\end{equation}
Hence, by the inductive hypothesis we get 
\begin{equation}\label{IH2}
\inttype{K_{sf}} \equiv \inttype{K_{sg}}
\end{equation}
and therefore
\begin{equation}
\inttype{K_f} \equiv \Idtype_{\inttype{K_{sf}}} (sf,tf) \overset{(\ref{IH1}),(\ref{IH2})}{\equiv} \Idtype_{\inttype{K_{sg}}} (sg,tg) \equiv \inttype{K_g}
\end{equation}
Analogous arguments work for the $\rsort_A$ and $\tsort_h$ cases.
\end{proof}

\begin{lemma}\label{hstrucwellformedlemma}
For all $K \in \NL{\L}$ the following context is well-formed
\[
\Gamma_K \eqdef (A \colon T_A)_{A \leq K, A \in \NL{\L}}
\]
\end{lemma}
\begin{proof}
We proceed by $<$-induction on $K$, as in Lemma \ref{strucwellformedlemma}.
Indeed, the proof proceeds exactly as in Lemma \ref{strucwellformedlemma} until we reduce the induction to the step of proving (\ref{deltants}), which now becomes more involved.
So assume that we know that the following context
\begin{equation}
\Delta \eqdef \Gamma_A, (g \colon \inttype{K_g} )_{g<f}
\end{equation}
is well-formed, and therefore in particular that each $\inttype{K_g}$ is a well-formed type in the context preceding it. We need to show that
\begin{equation}\label{deltants2}
\Delta \vdash \inttype{K_f} \: \TTisatype
\end{equation}
is derivable.
To do so we need to consider all the cases for $\inttype{K_f}(\mathbf{q}f)$.
If $K_f \in \NL{\L}$ then the proof proceeds just as in Lemma \ref{strucwellformedlemma}.
%(even in the context {$\Gamma_{K_f}f$}) if $K_f$ is well-formed since the worst that could happen is a repetition of a variable that was not there before (e.g. if $p_1f = p_2f$ for some $p_1 \neq p_2$).
 So it remains to check the cases where $K_f$ is a logical sort.
 
If $K_f \equiv\: \eqsort_A$ for some $A \in \L$ 
then we need to show that
\begin{equation}\label{eqnts}
\Delta \vdash \Idtype_{\inttype{K_sf}} (sf,tf) \TTisatype
\end{equation}
is derivable.
Firstly, note that $sf,tf < f$ and (by condition 1.(b) in Definition \ref{hsignatures}) we have $\toplevel{\eqsort_A}= \lbrace s,t \rbrace$. Hence, we have
\begin{equation}
\Delta \equiv \Gamma_A, (g \colon \inttype{K_g})_{g<sf}, sf \colon \inttype{K_{sf}}, tf \colon \inttype{K_{tf}}
\end{equation}
By condition 1.(c) in Definition \ref{hsignatures} we know that $qsf=qtf$ for all $q \in \pcosieve{A}{\L}$ which by Lemma \ref{hstrucwellformedlemma0} implies $\inttype{K_{tf}} \equiv \inttype{K_{sf}}$. Hence, by applying $\Idtype$-formation to
\begin{equation}
\Delta \vdash sf \colon \inttype{K_{sf}} \quad \quad \Delta \vdash tf \colon \inttype{K_{sf}}
\end{equation}
we obtain exactly (\ref{eqnts}).

If $K_f \equiv \rsort_A$ for some $A \in \L$ 
then we need to show that
\begin{equation}\label{rsortnts}
\Delta \vdash \Idtype_{\inttype{K_{rf}}} (rf,\refl_{srf}) \TTisatype
\end{equation}
is derivable.
Firstly, note that $rf< f$ and (by condition 2.(b) in Definition \ref{hsignatures}) $\toplevel{\rsort_A}= \lbrace r \rbrace$. Hence, we have
\begin{equation}
\Delta \equiv \Gamma_A, (g \colon \inttype{K_g})_{g<srf}, srf \colon \inttype{K_{srf}}, rf \colon \inttype{K_{rf}}
\end{equation}
But we know that $\inttype{K_{rf}} \equiv \: \eqsort_A$ and $srf=trf$ (both by Definition \ref{hsignatures}) and therefore
\begin{equation}
\inttype{K_{rf}} \equiv \Idtype_{\inttype{K_{srf}}} (srf,trf) \equiv \Idtype_{\inttype{K_{srf}}} (srf,srf)
\end{equation}
This implies that
\begin{equation}
\Delta \vdash \refl_{srf} \colon \inttype{K_{rf}}
\end{equation}
is derivable.
Hence, by applying $\Idtype$-formation to
\begin{equation}
\Delta \vdash rf \colon \inttype{K_{rf}} \quad \quad \Delta \vdash \refl_{srf} \colon \inttype{K_{rf}}
\end{equation}
we obtain exactly (\ref{rsortnts}).

Finally, if $K_f \equiv \tsort_h$ for some $h \in \morph{\L}$ 
then we need to show that
\begin{equation}\label{tsortnts}
\Delta \vdash \Idtype_{\inttype{K_{\tmorXX{h}f}}} (\tmorXX{h}f, \trans_{e_hf}^{\varmodifier{h\tmorX{h}f}.\inttype{K_{\tmorX{h}f}}[\varmodifier{h\tmorX{h}f}/h\tmorX{h}f]} (\tmorXX{h}f)) \:\TTisatype
\end{equation}
is derivable.
As above, by the relevant conditions in Definition \ref{hsignatures} we have
\begin{equation}
\Delta \equiv \Gamma_A, (g \colon \inttype{K_g})_{g < h\tmorX{h}f}, h\tmorX{h}f \colon \inttype{K_{h\tmorX{h}f}},  h\tmorXX{h}f \colon \inttype{K_{h\tmorXX{h}f}}, \tmorX{h}f \colon \inttype{K_{\tmorX{h}f}}, \tmorXX{h}f \colon \inttype{K_{\tmorXX{h}f}}, e_hf \colon \inttype{K_{e_hf}}
\end{equation}
By Lemma \ref{hstrucwellformedlemma0} we have 
\begin{equation}
\inttype{K_{h\tmorX{h}f}} \equiv \inttype{K_{h\tmorXX{h}f}}
\end{equation}
and since by Definition \ref{hsignatures} we know  
\begin{equation}\label{blurgyy}
se_hf = h\tmorX{h}f \quad \quad te_hf = h\tmorXX{h}f
\end{equation}
we obtain
\begin{equation}\label{pretrans1}
\inttype{K_{e_hf}} \equiv \Idtype_{\inttype{K_{se_hf}}} (se_hf, te_hf) \overset{(\ref{blurgyy})}{\equiv}\Idtype_{\inttype{K_{h\tmorX{h}f}}} (h\tmorX{h}f, h\tmorXX{h}f) 
\end{equation}
Now, since $\tmorX{h}f<f$, by the inductive hypothesis we know that $\Delta \vdash \inttype{K_{\tmorX{h}f}} \TTisatype$ is derivable (by weakening the smaller context in which we already know $\inttype{K_{\tmorX{h}f}}$ to be a type). Hence, since $h$ is top-level for $(\tsort_h)_{\tmorX{h}}$ (i.e. the codomain of $\tmorX{h}$) we can derive
\begin{equation}\label{trans1}
\Delta, \varmodifier{h\tmorX{h}f} \colon \inttype{K_{h\tmorX{h}f}} \vdash \inttype{K_{\tmorX{h}f}} [\varmodifier{h\tmorX{h}f}/h\tmorX{h}f] \TTisatype
\end{equation}
Furthermore, by (\ref{pretrans1}) above we can derive
\begin{equation}\label{trans2}
\Delta \vdash e_hf \colon \Idtype_{\inttype{K_{h\tmorX{h}f}}} (h\tmorX{h}f,h\tmorXX{h}f)
\end{equation}
Finally, since $\tmorX{h}f \colon \inttype{K_{\tmorX{h}f}}$ appears in $\Delta$ we have
\begin{equation}\label{trans3}
\Delta \vdash \tmorX{h}f \colon \inttype{K_{\tmorX{h}f}} [h \tmorX{h}f/\varmodifier{h \tmorX{h}f}]
\end{equation}
But now from (\ref{trans1}),(\ref{trans2}) and(\ref{trans3}) we can apply the transport rule $\transportrule$ to obtain
\begin{equation}
\Delta \vdash \trans_{e_hf}^{\varmodifier{h\tmorX{h}f}.\inttype{K_{\tmorX{h}f}}[\varmodifier{h\tmorX{h}f}/h\tmorX{h}f]} (\tmorXX{h}f) \colon \inttype{K_{\tmorX{h}f}} [\varmodifier{h\tmorX{h}f}/h\tmorX{h}f] [h\tmorXX{h}f/\varmodifier{h\tmorXX{h}f}]
\end{equation}
But now observe that
\begin{align}
\inttype{K_{\tmorXX{h}f}} &\equiv \inttype{K_{\tmorX{h}f}} [h\tmorXX{h}f/h\tmorX{h}f] \\
				       &\equiv \inttype{K_{\tmorX{h}f}} [\varmodifier{h\tmorX{h}f}/h\tmorX{h}f][h\tmorXX{h}f/\varmodifier{h\tmorX{h}f}]
\end{align}
since $\inttype{K_{\tmorXX{h}f}}$ differs from $\inttype{K_{\tmorX{h}f}}$ only in $h\tmorX{h}f$ (by the conditions in Definition \ref{hsignatures}).
Hence, we have
\begin{equation}\label{transfinal1}
\Delta \vdash \trans_{e_hf}^{\varmodifier{h\tmorX{h}f}.\inttype{K_{\tmorX{h}f}}[\varmodifier{h\tmorX{h}f}/h\tmorX{h}f]} (\tmorXX{h}f) \colon \inttype{K_{\tmorXX{h}f}}
\end{equation}
Since we also have
\begin{equation}\label{transfinal2}
\Delta \vdash \tmorXX{h}f \colon \inttype{K_{\tmorXX{h}f}}
\end{equation}
by $\Idtype$-formation on (\ref{transfinal1}) and (\ref{transfinal2}) we obtain exactly (\ref{tsortnts}), and we are done.
\end{proof}

\begin{theo}\label{hstrucwellformed}
$\HStruc{\L}$ is a closed type.
\end{theo}
\begin{proof}
Since $\L$ is assumed essentially finite, by Lemma \ref{hstrucwellformedlemma} we get immediately that
\[
\Gamma_\L \eqdef (K \colon T_K)_{K \in \NL{\L}}
\]
is well-formed. 
And then we obtain $\HStruc{\L}$ (in the empty context) by successive applications of $\Sigma$-formation.
\end{proof}

\begin{defin}[$\L$-structure]\label{hstructure}
A \textbf{homotopy $\L$-structure} is a term $\M$ of type $\HStruc{\L}$.
\end{defin}

\begin{remark}\label{functoriality}
It is reasonable to now wonder, as the notation suggests, whether $\HStruc{-}$ can be thought of as a ``functor'' $\hSig \rightarrow \U$ (for some universe $\U$).
While it is easy to see that that for full and faithful logical $h$-morphisms $I \colon \L \hookrightarrow \L'$ we will get induced functions $\HStruc{\L'} \rightarrow \HStruc{\L}$ using the projections of the $\Sigmatype$-types, to promote this assignment to a functor would lead us into the usual coherence problems with $\U$ (cf. \cite{ACKinfinite, HoTTeat}). We discuss the prospects for fixing this issue in our concluding remarks.
\end{remark}

We would now like to define for any $\FOLiso$-signature $\L$ the interpretation of $\LTT_{\L}$ into HoTT in a manner analogous to Definition \ref{interpretationFOLDS}. However, we are here faced with the problem, alluded to in Remark \ref{remarkinfinitary}, that even if $\L$ is essentially finite, its underlying FOLDS signature may not be, which means that $\LTT_{\L}$ would contain an infinite number of rules (one rule $\Kform$ for each sort $K$ in $\L$). In particular, for an essentially finite $\L$, this will be the case precisely when $\L$ contains a sort of $h$-level $\infty$. However, in those cases we can still show that $\LTT_{\L}$ is ``finitely presentable'' in a sense that we make precise now.

\begin{defin}
An ML type theory $TT=(S,J,R)$ is \textbf{finitary} if $R$ is finite (as a list). An ML type theory is \textbf{finitely presentable} if it is bi-interpretable with a \textbf{finitary} ML type theory.
\end{defin}

\newcommand{\finpres}[1]{#1^{\text{fp}}}

\begin{prop}\label{FOLisofp}
For any essentially finite $\FOLiso$-signature $\L$, $TT_{\L}$ is finitely presentable.
\end{prop}
\begin{proof}
For a given $\FOLiso$ signature $\L$ we define the finitary ML type theory $\finpres{TT_\L}$ as follows.
Its syntax and judgments are exactly the same as $TT_\L$ and it has all the structural rules as well as all the $\Kform$ rules for non-logical $K$. However, instead of $\Kform$ rules for logical $K$ it instead contains the following three rules (in the order in which they are listed since there is dependence between them):
\[
\inferrule
{
\Gamma \vdash x \colon A (\mathbf{z}) \\ \Gamma \vdash y \colon A (\mathbf{z})
}
{
\Gamma \vdash x \eqsort_{A (\mathbf{z})} y \isatype
}
\quad (\eqsort\text{-form})
\]
\[
\inferrule
{
\Gamma \vdash x \colon A (\mathbf{z}) \\ \Gamma \vdash q \colon x\eqsort_{A (\mathbf{z})} x
}
{
\Gamma \vdash \rsort_{A (\mathbf{z})} (q,x) \isatype
}
\quad (\rsort\:\text{-form})
\]
\[
\inferrule
{
\Gamma \vdash y \colon O (\mathbf{z}) \\
\Gamma \vdash y' \colon O (\mathbf{z}) \\\\
%\Gamma \vdash p \colon y=_{A (\mathbf{z})}y' \\\\
\Gamma \vdash \alpha \colon A(\mathbf{w}) \\
\Gamma \vdash \beta \colon A(\mathbf{w})[y'/y] \\\\
\Gamma \vdash p \colon y=_{O (\mathbf{z})}y' \\\\
\Gamma, w_h \colon O (\mathbf{z}) \vdash A (\mathbf{w})[w_h/y] \isatype \\ 
}
{
\Gamma  \vdash \tau_h (p, \alpha, \beta) \isatype
}
\quad (\tsort\text{-form}), \: h \in \L(A, O)
\]
It is then immediately checked that $\finpres{TT_\L}$ and $TT_\L$ are bi-interpretable.
\end{proof}

\begin{cor}
For any essentially finite $\FOLiso$-signature $\L$, $\LTT_{\L}$ is finitely presentable.
\end{cor}
\begin{proof}
$\LTT_\L$ only adds a finite number of rules to $TT_\L$ regardless of whether or not $\L$ is essentially finite. So we can just define $\finpres{\LTT_\L}$ just as $\LTT_\L$ but extending $\finpres{TT_\L}$ instead of $TT_\L$.
\end{proof}

\begin{notation}
For any $\FOLiso$-signature $\L$ we will now write $TT_\L$ and $\LTT_\L$ for $\finpres{TT_\L}$ and $\LTT_\L$ for $\finpres{\LTT_\L}$ and treat them as the same type theory.
\end{notation}

With this in mind the interpretation of $\FOLiso$-contexts, context morphisms, formulas and sequents for a given homotopy $\L$-structure $\M$ proceeds as in Definition \ref{interpretationFOLDS} 
%and the analogues of Theorem \ref{FOLDSsoundness}, 
with extra stipulations for the logical sorts. 

\begin{defin}[Interpretation of $\LTT_\L$ in HoTT]\label{FOLisointerpretation}
Let $\L$ be a $\FOLiso$-signature.
The \textbf{depth} of the expressions in the raw syntax of $\LTT_\L$ is defined as in Definition \ref{interpretationFOLDS} with the following extra clauses:
\begin{align*}
\FOLDSdepth{\eqsort_K (\mathbf{x})} &= \FOLDSdepth{K}+1 \\
\FOLDSdepth{\rsort_K (\mathbf{x})} &= \FOLDSdepth{K}+2 \\
\FOLDSdepth{\tsort_h (\mathbf{x})} &= \FOLDSdepth{K}+3 \\
%\FOLDSdepth{\top} &= \FOLDSdepth{\bot} = 1  \\
%\FOLDSdepth{\phi * \psi} &= \FOLDSdepth{\phi} + \FOLDSdepth{\psi} \quad \quad \quad \: (*= \wedge, \vee, \rightarrow) \\
%\FOLDSdepth{Qx \colon K. \phi} &= \FOLDSdepth{K}+\FOLDSdepth{\phi} \quad \quad \quad (Q= \exists, \forall) \\
%\FOLDSdepth{\Gamma \isacontext} &= \FOLDSdepth{\Gamma} \\
%\FOLDSdepth{\Gamma \vdash K \isatype} &= \FOLDSdepth{\Gamma} + \FOLDSdepth{K} \\
%\FOLDSdepth{\Gamma \vdash x \colon K} &= \FOLDSdepth{\Gamma} +1 \\
%\FOLDSdepth{\Gamma \vdash \phi \isaformula} &= \FOLDSdepth{\Gamma}+\FOLDSdepth{\phi} \\
\end{align*}
The \textbf{interpretation of $\LTT_\L$ into HoTT}
 %into an $\L$-structure $\M$} 
 consists of the function $\interpretation{-}_c$ defined just as in Definition \ref{interpretationFOLDS} and the new functions $\interpretation{-}_s$ and $\interpretation{-}_f$ defined as follows:
%\begin{align*}
%\interpretation{-}_c \colon \FOLDScontexts &\rightarrow \MLTTcontexts \\
%                                \varnothing          &\mapsto     \M \colon \Struc{\L}     \\
%                                \Gamma, x \colon K &\mapsto \interpretation{\Gamma}_c, x \colon \interpretation{K}_s
%\end{align*}
\begin{align*}
\interpretation{-}_s \colon \FOLDSsorts &\rightarrow \MLTTtypes \\
			      % K\: \epsilon &\mapsto \mathtt{El}(\TTapp(*,K^\M))  \\
                                K (\mathbf{x})          &\mapsto     \TTElapp{K^\M (\mathbf{x})}  \quad \quad \quad \quad (K \in \NL{\L})    \\   
                                \eqsort_K (\mathbf{x})      &\mapsto  \Idtype_{\interpretation{K(x_{\mathbf{p}s})}_s} (x_s,x_t)          \\
                                \rsort_K (\mathbf{x})         &\mapsto  \Idtype_{ \Idtype_{\interpretation{K(x_{\mathbf{p}sr})}_s} (x_{sr},x_{tr})} (x_r, \refl_{x_{sr}})          \\  
                                \tsort_h (\mathbf{x})          &\mapsto  \Idtype_{\interpretation{\dom{h} (x_{\mathbf{p}\tmorXX{h}})}_s} (x_{\tmorXX{h}}, \trans_{\tmorpath{h}}^{\varmodifier{x_{h\tmorX{h}}}. \interpretation{\dom{h}(x_{\mathbf{p}\tmorX{h}})}_s [\varmodifier{x_{h\tmorX{h}}}/x_{h\tmorX{h}}]} (x_{\tmorX{h}})          \\   
\end{align*}
\begin{align*}
\interpretation{-}_f \colon \FOLDSformulas &\rightarrow \MLTTtypes \\
                                \bot         &\mapsto     \zerotype                 \\ 
                                \top          &\mapsto     \onetype                 \\
                                \phi \wedge \psi         &\mapsto   \interpretation{\phi}_f \times \interpretation{\psi}_f        \\
                                \phi \vee \psi         &\mapsto   \truncated{\interpretation{\phi}_f + \interpretation{\psi}_f } \\ 
                               \phi \rightarrow \psi         &\mapsto \interpretation{\phi}_f \rightarrow \interpretation{\psi}_f    \\
                                \exists x \colon K. \phi &\mapsto \truncated{\Sigmatype \:(x \colon \interpretation{K}_s)\: \interpretation{\phi}_f} \\
                                \forall x \colon K. \phi &\mapsto \Pitype \:(x \colon \interpretation{K}_s) \: \interpretation{\phi}_f
\end{align*}
%\begin{align*}
%\interpretation{-}_f \colon \FOLDSformulas &\rightarrow \MLTTtypes \\
%                                \bot         &\mapsto     \zerotype                 \\ 
%                                \top          &\mapsto     \onetype                 \\
%                                \phi \wedge \psi         &\mapsto   \interpretation{\phi}_f \times \interpretation{\psi}_f        \\
%                                \phi \vee \psi         &\mapsto   \interpretation{\phi}_f + \interpretation{\psi}_f   \\ 
%                               \phi \rightarrow \psi         &\mapsto \interpretation{\phi}_f \rightarrow \interpretation{\psi}_f    \\
%                                \exists x \colon K. \phi &\mapsto \truncated{\Sigmatype \:(x \colon \interpretation{K}_s)\: \interpretation{\phi}_f} \\
%                                \forall x \colon K. \phi &\mapsto \Pitype \:(x \colon \interpretation{K}_s) \: \interpretation{\phi}_f
%\end{align*}
%We will drop the subscripts of each separate interpretation function and write simply $\interpretation{\Gamma}, \interpretation{K}, \interpretation{\phi}$. 
%With this in mind t
The interpretation of the judgments of $\LTT_\L$ is then just as in Definition \ref{interpretationFOLDS}.
%\begin{align*}
%\interpretation{-} \colon \FOLDSjudgments &\rightarrow \MLTTjudgments \\
%                                \Gamma \isacontext          &\mapsto     \interpretation{\Gamma} \isacontext      \\
%                                \Gamma \vdash K \isatype &\mapsto \interpretation{\Gamma} \vdash  \interpretation{K} \TTisatype  \\
%     \Gamma \vdash x \colon K &\mapsto \interpretation{\Gamma} \vdash x \colon  \interpretation{K} \\  
%     \Gamma \vdash \phi \isaformula &\mapsto \interpretation{\Gamma} \vdash \interpretation{\phi} \TTisatype                   
%\end{align*}
%For a judgment $\mathcal{S}$ in $\LTT_\L$ we write $\interpretation{S}$ for its interpretation in MLTT.
\end{defin}

\begin{remark}
To make it clear, the two changes in Definition \ref{FOLisointerpretation} over Definition \ref{interpretationFOLDS} is that $\interpretation{-}_s$ is defined by induction on the level of $K$ taking into account the logical sorts and that in $\interpretation{-}_f$ we are now truncating $\vee$ and $\exists$.
\end{remark}

\begin{theo}[Correctness of the HoTT Interpretation]\label{FOLisocorrectness}
If $\mathcal{S}$ is a derivable judgment in $\LTT_\L$ then $\interpretation{S}$ is a derivable judgment in MLTT. 
\end{theo}
\begin{proof}
We proceed by induction on the depth of the expressions of $\LTT_\L$. By inspection we can see that the depth function in Definition \ref{FOLisointerpretation} is once again correct in the sense that for every rule of $\LTT_\L$ the depth of the expression below the line is strictly greater than the depth of any of the expressions above the line. 
The proof thus can proceed by induction on the complexity of derivations of $\LTT_\L$. The only rules that we need to check whose proof is not the same as in Theorem \ref{FOLDSsoundness} are the formation rules for the logical sorts and the formula formation rules for $\vee$ and $\exists$ which now involve truncation. We take them in turn:

\bigskip
\noindent $(\eqsort\text{-form})$: 
We have the following:
\begin{align*}
\interpretation{(\eqsort\text{-form})} &= \inferrule{\interpretation{\Gamma \vdash x \colon A (\mathbf{z})} \\ \interpretation{\Gamma \vdash y \colon A (\mathbf{z})}}{\interpretation{\Gamma \vdash x \eqsort_{A (\mathbf{z})} y \isatype}} \\ 
&=\inferrule{\interpretation{\Gamma} \vdash x \colon \interpretation{A (\mathbf{z})} \\ \interpretation{\Gamma} \vdash y \colon \interpretation{A (\mathbf{z})}}{\interpretation{\Gamma} \vdash  \Idtype_{\interpretation{A(\mathbf{z})}} (x,y) \TTisatype}
\end{align*}
But this is exactly $\Idtype$-formation.

\bigskip
\noindent $(\rsort\text{-form})$: 
We have the following:
\begin{align*}
\interpretation{(\rsort\text{-form})} &= \inferrule{\interpretation{\Gamma \vdash x \colon A (\mathbf{z})} \\ \interpretation{\Gamma \vdash q \colon x \eqsort_{A(\mathbf{z})} x}}{\interpretation{\Gamma \vdash \rsort_{A (\mathbf{z})}(q,x) \isatype}} \\ 
&=\inferrule{\interpretation{\Gamma} \vdash x \colon \interpretation{A (\mathbf{z})} \\ \interpretation{\Gamma} \vdash q \colon \interpretation{x \eqsort_{A(\mathbf{z})} x}}{\interpretation{\Gamma} \vdash  \Idtype_{ \Idtype_{\interpretation{A(\mathbf{z})}} (x,x)} (q, \refl_{x}) \TTisatype} \\
&=\inferrule{\interpretation{\Gamma} \vdash x \colon \interpretation{A (\mathbf{z})} \\ \interpretation{\Gamma} \vdash q \colon \Idtype_{\interpretation{A(\mathbf{z})}} (x,x)}{\interpretation{\Gamma} \vdash  \Idtype_{ \Idtype_{\interpretation{A(\mathbf{z})}} (x,x)} (q, \refl_{x}) \TTisatype}
\end{align*}
But this follows by an application of $\Idtype$-introduction on $\interpretation{\Gamma} \vdash x \colon \interpretation{A (\mathbf{z})}$ followed by $\Idtype$-formation.

\bigskip
\noindent $(\tsort\text{-form})$: 
We have the following:
\begin{align*}
\interpretation{(\tsort\text{-form})} &= 
\inferrule
{
\interpretation{\Gamma \vdash y \colon O (\mathbf{z})} \\
\interpretation{\Gamma \vdash y' \colon O (\mathbf{z})} \\
%\Gamma \vdash p \colon y=_{A (\mathbf{z})}y' \\\\
\interpretation{\Gamma \vdash \alpha \colon A(\mathbf{w})} \\
\interpretation{\Gamma \vdash \beta \colon A(\mathbf{w})[y'/y]} \\\\
\interpretation{\Gamma \vdash p \colon y\eqsort_{O (\mathbf{z})}y'} \\
\interpretation{\Gamma, w_h \colon O (\mathbf{z}) \vdash A (\mathbf{w})[w_h/y] \isatype}
}
{
\interpretation{\Gamma  \vdash \tau_h (p, \alpha, \beta) \isatype}
}
\\ 
&=
\inferrule
{
\interpretation{\Gamma} \vdash y \colon \interpretation{O (\mathbf{z})} \\
\interpretation{\Gamma} \vdash y' \colon \interpretation{O (\mathbf{z})} \\
%\Gamma \vdash p \colon y=_{A (\mathbf{z})}y' \\\\
\interpretation{\Gamma} \vdash \alpha \colon \interpretation{A(\mathbf{w})} \\
\interpretation{\Gamma} \vdash \beta \colon \interpretation{A(\mathbf{w})[y'/y]} \\\\
\interpretation{\Gamma} \vdash p \colon \interpretation{y\eqsort_{O (\mathbf{z})}y'} \\
\interpretation{\Gamma}, w_h \colon \interpretation{O (\mathbf{z})} \vdash \interpretation{A (\mathbf{w})[w_h/y] \isatype}
}
{
\interpretation{\Gamma}  \vdash  
\Idtype_{\interpretation{A (\mathbf{w})}} (\beta, \trans_{p}^{w_h. \interpretation{A(\mathbf{w}) } [w_h/y]} (\alpha))   
 \TTisatype
}
%\\ 
%&=\inferrule{\interpretation{\Gamma} \vdash x \colon \interpretation{A (\mathbf{z})} \\ \interpretation{\Gamma} \vdash q \colon \Idtype_{\interpretation{A(\mathbf{z})}} (x,x)}{\interpretation{\Gamma} \vdash  \Idtype_{ \Idtype_{\interpretation{A(\mathbf{z})}} (x,x)} (q, \refl_{x}) \TTisatype}
\end{align*}
But this follows by an application of $\transportrule$ and $\Idtype$-formation.

%%%HOTTmoveto
\bigskip
\noindent $(\vee\texttt{-form})$: 
Given that $\interpretation{\Gamma}$ is a well-formed context and $\interpretation{\phi}$ and $\interpretation{\psi}$ are types in that context, then by $+$-formation we get 
\begin{equation}
\interpretation{\Gamma} \vdash \interpretation{\phi} + \interpretation{\psi} \TTisatype
\end{equation}
and then by an application of truncation we get
\begin{equation}
\interpretation{\Gamma} \vdash \truncated{\: \interpretation{\phi} + \interpretation{\psi} \:} \TTisatype
\end{equation}
as required.

\bigskip
\noindent $(\exists\texttt{-form})$: 
Given that $\interpretation{\Gamma, x \colon K} =  \interpretation{\Gamma}, x \colon \interpretation{K}$ is a well-formed context and $\interpretation{\phi}$ is a type in that context, then by $\Sigmatype$-formation we get 
\begin{equation}
\interpretation{\Gamma} \vdash \Sigmatype \: (x \colon \interpretation{K}) \: \interpretation{\phi} \TTisatype
\end{equation}
and then by an application of truncation we get
\begin{equation}
\interpretation{\Gamma} \vdash \truncated{ \: \Sigmatype \: (x \colon \interpretation{K}) \: \interpretation{\phi} \:} \TTisatype
\end{equation}
as required.

\end{proof}

\begin{remark}
Theorem \ref{FOLisocorrectness} guarantees that for any categorical semantics of (any type theory HoTT$^+$ extending) HoTT for which the initiality conjecture has been established there will be a model of $\LTT_\L$ in any model of those categorical semantics, and in particular in any model in $\infty$-toposes. However, in order to make this fact a fully rigorous corollary, we need to have a definition of a ``model of $\LTT_\L$'' in terms of some yet-to-be-defined categorical semantics for $\LTT_\L$. We plan on giving such definitions in a follow-up to this paper. With those definitions in hand Theorem \ref{FOLisocorrectness} establishes that the term model (assuming it exists) of any HoTT$^+$ will contain a model of (the categorical semantics for) $\LTT_\L$, for any $\FOLiso$-signature $\L$.
\end{remark}

Corollaries \ref{FOLDSsoundnesscor} and \ref{FOLDSsoundnesscor2} 
%and Proposition \ref{formulasaremerepropositions} 
now carry over to $\FOLiso$ verbatim. 
Similarly, the notion of extension, realization, satisfaction, model and type of models in Definitions \ref{FOLDSextension}, \ref{FOLDSsatisfaction} and \ref{FOLDStypeofmodels} carry over immediately to $\FOLiso$, as does the notion of a theory and (classical) entailment in Definitions \ref{FOLDStheory} and \ref{entailment}.

That said, we now must record some results that are important in the setting of HoTT that do not necessarily arise in the more traditional setting of MLTT.

\begin{cor}\label{formulasaremerepropositions}
For any formula-in-context $\Gamma.\phi$ and any $\L$-structure $\M$ we have that $\phi^\M$ is a mere proposition (in context $\Gamma^\M$).
\end{cor}
\begin{proof}
Immediate by the definition of the interpretation of formulas, since the only type-formers that do not preserve $h$-level are $\Sigmatype, +$ and we truncate them.
\end{proof}

%Finally, we note that once again our semantics entail that we lose no generality if we work only with theories over replete $\FOLiso$-signatures, as the following proposition records.

\begin{cor}\label{hsigsameasfolisosig}
Let $\L$ be an $h$-signature.
For any $\L$-theory $\T$ let $\T^{\cong}$ be the same theory over $\L^{\cong}$. Then $\Mode{\T} = \Mode{\T^{\cong}}$
\end{cor}
\begin{proof}
Immediate from Proposition \ref{hstrucsameasfolisostruc}.
\end{proof}

%%%Initiality stuff
%\begin{terminology}
%Let $TT$ be some version of MLTT and $C(TT)$ a categorical semantics for $TT$, i.e. $C(TT)$ is a (usually essentially algebraic) first-order theory with the right structure to interpret the syntax and rules of $TT$.
%We say that \emph{$C(TT)$ is an initial categorical semantics for $TT$} if the initiality conjecture holds for $C(TT)$, i.e. the term model $M_{TT}$ for $TT$ is a model of $C(TT)$ and is initial among all other models of $C(TT)$. 
%Finally, we say that a given FOLDS signature $\L$ \emph{interprets into} a type theory $TT$ if $\LTT_\L$ interprets into $TT$.
%\end{terminology}
%
%\begin{problem}
%Given an initial categorical semantics $C(TT)$ for $HoTT$ and a model of $M$ of $C(TT)$ to construct an interpretation of any $\FOLiso$-signature $\L$ into $M$. 
%\end{problem}
%\begin{construction}
%By
%\end{construction}
%%%Initiality stuff

We conclude this section with several examples illustrating the semantics of $\FOLiso$.

%TODO: Remark concerning the h-level of extensions of formulas
%However, in this setting the extension $\Extension{\Gamma}{\M}$ may be a type of arbitrary $h$-level

\begin{exam}\label{LKone}
%Let $\L$ be the FOLDS signature defined by the category with one object (of dimension $-1$) and one identity morphism. 
Let $\L_{K,1}$ be the $h$-signature with one object of $h$-level $1$.
Then an $\L_{K,1}$-structure is simply a mere proposition $P \colon \textbf{Prop}_{\mathcal{U}}$. 
%This allows us to say that \textcolor{red}{$(1)$-logic} has the same expressive power as propositional logic.
%\end{exam}
%
%\begin{exam}
%Let $\L=\textbf{1}$, where \textbf{1} is the category with one object and one identity arrow. 
Similarly, if we write $\L_{K,n}$ for the $\FOLiso$-$n$-signature with one object of $h$-level $n$, a $\L_{K,n}$-structure is a type of $h$-level $n$.
%Thus the empty $\textbf{1}$-$n$-theory is simply the theory of $n$-groupoids (understood as $n$-types in HoTT).
%This provides a generalization of the fact that the empty theory over a signature with a single sort in first-order logic is simply the theory of sets, also known as the ``pure theory of identity''.
\end{exam}

\begin{exam}
An $\Lgraph$-structure consists of an $h$-set $O \colon \SetU$ and a mere relation $A  \colon O \rightarrow O \rightarrow \PropU$. 
%\end{exam}
%
%\begin{exam}
%More generally, the study of $h$-signatures of height 1 essentially coincides with traditional set-based model theory (over relational signatures). 
%This will becomes clearer after we have defined the interpretation of formulas.
\end{exam}

\begin{exam}
An $\Lrg^{321}$-structure $\M$ consists of the following data:
\begin{align*}
I^\M &\colon \underset{x \colon O^\M}{\Pi} \: A(x,x) \rightarrow \textbf{Prop}_{\mathcal{U}} \\
A^\M &\colon O^\M \rightarrow O^\M \rightarrow \textbf{Set}_{\mathcal{U}} \\
O^\M &\colon \textbf{Gpd}_{\mathcal{U}}
\end{align*}
Note that an $(\Lrg^{321})^{\cong}$-structure consists of exactly the same data since the logical sorts do not appear in $\HStruc{\Lrg}$ (since no non-logical sort depends on them).
%This is once again the crucial example to keep in mind as it will illustrate most of the concepts we introduce below.
\end{exam}

\begin{exam} 
In $\Lrg$ if we are given the formula $\phi \equiv \exists \tau \colon I(x,f). \top$ then its interpretation in some $\Lrg$-structure $\M$ will be given by 
$
\vert \vert \underset{\tau \colon I^\M (x,f)}{\Sigma} \textbf{1} \vert \vert
$
which is of course equivalent to $\vert \vert I^\M (x,f) \vert \vert$. 
%But this is not strictly speaking a type since $x$ and $f$ remain variables. Instead, we should take 
The interpertation of $\phi$ in $\M$ thus consists of the following judgement
\[
x\colon O^\M, f \colon A^\M(x,x) \vdash \vert \vert I^\M (x,f) \vert \vert  \TTisatype
\]
where we are abusing notation in using $I^\M (x,f)$ for what really is its first projection since strictly speaking $I^\M$ was defined as a dependent function into $\textbf{Prop}_{\mathcal{U}}$.
Its extension is
\[
\Extension{\Gamma}{\M}(\phi) = \underset{
\begin{subarray}
\: x \colon O^\M \\
f \colon A^\M(x,x)
\end{subarray}
}{\Sigmatype} \truncated{I^\M (x,f)}
\]
and $\M \models \phi[\mathbf{a}/\Gamma]$ iff there is a derivable term
\[
\mathbf{a} \colon \Extension{\Gamma}{\M}(\phi)
\]
%Since both $\textbf{0}$ and $\textbf{1}$ are types in the empty context, by the weakening rule for contexts we will get an analogous definition for any formula $\phi$ in any context, possibly larger than its context of free variables. 
%(For example, by weakening we will also have  that
%\[
%\Delta^\M, x\colon O^\M, g,f \colon A^\M(x,x) \vdash \vert \vert I^\M (x,f) \vert \vert \colon \mathcal{U}
%\]
%which can be taken as the interpretation of $\phi$ in the context $\lbrace x, f, g \rbrace$ which properly contains its canonical context.)
\end{exam}

\begin{exam}
Consider $\Lgraph^=$:
\[
\xymatrix{
1 &A \ar@/_/[d] \ar@/^/[d] &=_O \ar@/_/[ld] \ar@/^/[ld] \\
2 &O
}
\]
$\Lgraph$-formulas are then exactly (when suitably translated) the formulas of first-order logic with equality for a single-sorted signature $\Sigma$ with a single binary predicate $A$. 
Semantically, an $\Lgraph$-structure $\M$ consists of an $h$-set $O^\M \colon \textbf{Set}_{\mathcal{U}}$ and a dependent type $A^\M \colon O^\M \rightarrow O^\M \rightarrow \textbf{Prop}_{\mathcal{U}}$.
This is all entirely analogous to $\Sigma$-structures in traditional set-theoretic semantics. 
%We may say that $0$-logic has the same expressive power as first-order logic with equality and its semantics as defined here coincide with the usual set-theoretic semantics.
$\FOLeq$ can thus be thought of as the ``classical limit'' of $\FOLiso$.
\end{exam}

\begin{exam}\label{fullexhott}
Consider the following $\Lrg$-sentence
\[
\phi \equiv \forall x \colon O \exists f \colon A(x,x). I(f)
\]
and let $\M$ be the $\Lrg$-1-structure given by the following data
\[
\langle \textbf{Set}_{\mathcal{U}}, \lambda x.\lambda y.x \rightarrow y, \lambda x. \lambda f. \text{Id}_{x\rightarrow x} (f,1_x) \rangle
\]
Then $\M$ is a model of $\phi$. To see this, observe that
\[
\M \models \forall x \colon O \exists f \colon A(x,x). I(f)
\]
since
\[
\varnothing \vdash \lambda x.(1_x, \mathtt{refl}_{1_x}) \colon \underset{x \colon \textbf{Set}_{\mathcal{U}}}{\Pi} \vert \vert \underset{f \colon x \rightarrow x}{\Sigma} \text{Id}_{x \rightarrow x}(f,1_x) \vert \vert
\] 
is derivable in HoTT and since the interpretation of $\phi$ is given by
\[
(\forall x \colon O \exists f \colon A(x,x). I(f))^\M \equiv \underset{x \colon \textbf{Set}_{\mathcal{U}}}{\Pi} \vert \vert \underset{f \colon x \rightarrow x}{\Sigma} \text{Id}_{x \rightarrow x}(f,1_x) \vert \vert
\]
\end{exam}

\begin{exam}
If we define $\Tcircle$ as the $\Lcircle$-theory which consists of the single axiom
\[
\forall x \colon O. \forall \tau \colon \mathtt{base} (x). \exists l \colon x=_Ox. \mathtt{loop}(l, \tau)
\]
we can easily deduce that
\[
\Mod{\Tcircle} \simeq \Sigmatype \: (O \colon \U) \: (\mathtt{base} \colon O) \: \Idtype_O (\mathtt{base},\mathtt{base})
\]
The higher inductive type corresponding to the circle $S^1$ can then be understood as an initial object in an appropriate precategory structure for $\Mod{\Tcircle}$ as has been documented in \cite{Sojakova}. 
More generally, the existence of initial models for $\FOLiso$-theories will be equivalent to the existence of certain higher inductive types.
This example thus illustrates the essence of the connection of $\FOLiso$ with the problem of finding a general specification of higher inductive types, a connection which seems to us worth pursuing.
\end{exam}

%What remains now is to define a proof system for $\FOLiso$ and prove a soundess theorem for it with respect to %our homotopy semantics. We take this task up in the next Section.

\section{Deductive System for $\FOLiso$}\label{nlogicProof}

We will now describe a proof system $\Diso$ for $\FOLiso$ and prove a soundess theorem for it with respect to our homotopy semantics.
The rules of $\Diso$ will be the rules of $\DFOLDS$ supplemented with an axiom $\Eqintro$ that postulates the existence of the ``trivial isomorphism'' in terms of the reflexivity sort, an axiom $\trrule$ that postulates that ``transporting along the trivial isomorphism does nothing'', and a rule $\Jrule$ governing the isomorphism sorts that can be understood as a ``first-order'' version of the $\Idtype$-elimination rule of type theory.

\begin{notation}
%In our presentation above we have assumed that every sequent is well-formed and that the rules are only instantiated for appropriate sorts (e.g. (Eq-intro) is instantiated only for $K$ with $d(K) < -1$).
% i.e. only for those $K$ such that there is an $r_K$). 
In the presentation below, we will assume all the sequents are well-formed. In particular, this will eliminate all ambiguity regarding which variables depend on which others, e.g. in the $\trrule$ rule below.
We will generally suppress (implicit) variable dependencies, e.g. writing simply $x=_Ky$ without making explicit the variables on which $K$ might depend.
We will also drop subscripts for equality, reflexivity and transport sorts when the sort to which they refer is clear from the context, and write simply $\eqsort, \rsort$ and $\tsort$.
%We will also use the notation $\phi(\rsort_K(q),\mathbf{x})$ to denote the conjunct $\rsort(q) \wedge \phi (q, \mathbf{x})$, where $\mathbf{x}$ are any other variables that may appear in $\phi$.
%Furthermore, we have suppressed all variable dependencies that are not directly relevant to the rule in question and have also assumed that sorts $K$ as they appear could themselves denote isomorphism sorts (in which case we understand $=_K^1$ as an abbreviation for $=_{K'}^{i+1}$ where $K$ is $=_{K'}^i$).
For a formula $\phi$ in context $\Gamma, x \colon K, y \colon K, p \colon x =_K y$ we will use the notation $\phi[x,x,q]$ (e.g. in the (J) rule below) to denote the formula $\delta (\phi)$ obtained by substitution along the ``contraction'' context morphism $$\delta \colon \big( \Gamma, x \colon K, q \colon x =_K x \big) \conmorphism \big( \Gamma, x \colon K, y \colon K, p \colon x =_K y \big)$$
% described in Example \ref{contraction}, 
%and we will combine this notation with our previous one, writing e.g. $\phi[x,x,\rsort(q)]$ for $\rsort(q) \wedge \phi[x,x,q]$.
%Finally, we will use the useful and suggestive notation $\truncated{\phi(x \colon K)}$ for any formula $\exists x \colon K. \phi$ or simply $\truncated{\phi(x)}$ when the sort of the bound variable is clear from the context, and once again combine this notation with our above conventions.
%Finally, let us note that in the definition of the formal system $\Diso$ below we will assume that a given rule is included only if the specific signature $\L$ over whose well-formed sequents $\Diso$ is defined contains the logical sorts that these rules refer to, e.g. the rule $\Eqintro$ below will appear only if $\rsort_K$ appears in $\L$ (for some $K \in \L$).
\end{notation}

%\begin{terminology}
%In a $\FOLiso$-signature $\L$ we will say that an isomorphism sort $=_K$ is \emph{top-level} if is not the codomain of any arrow in $\L$. 
%\end{terminology}
%
%\begin{remark}
%An isomorphism sort $=_K$ can be top-level in a $\FOLiso$-signature only if $h(K)=2$.
%\end{remark}

\begin{defin}
The \textbf{proof system $\Diso$ for $\FOLiso$} for any given $\FOLiso$-signature $\L$ consists of the rules of (intuitionistic or classical, regular, coherent or full) $\DFOLDS$ together with the following three rules:
\[
\inferrule
{
\!
}
{
\Gamma, x \colon K \: \vert \: \theta \seqimplies \exists q \colon x\eqsort x. \rsort (q)
%\exists p \colon x=_K x. \rsort_K (p)
}
\quad  \Eqintro %(\text{Eq}_1^{K,j})
\]
%\quad\quad
%\[
%\inferrule
%{
%\!
%}
%{
%\Gamma, x \colon K \: \vert \: \theta \vdash \exists p \colon x=_K^{d(K)+1} x. \top
%}
%\quad  \text{(Eq-intro-top)}
%\]

%\bigskip

%(1b)
%$
%\inferrule
%{
%\!
%}
%{
%\exists p \colon x=_K^{d(K)+1} x. \top
%}
%\quad  \text{ ``$\mathtt{Id}$-intro'' for top-level equalities } %(\text{Eq}_1^{K,j})
%$
%
%\bigskip

%(2)
%$
%\inferrule
%{
%\:
%}
%{
%\exists p \colon x=^j_K y. \top \rightarrow \exists q \colon y=^j_K x. \top
%}
%\quad  %\text{ where } K \in \L, d(K) \leq n (\text{Eq}_2^{K,j})
%$
%
%\bigskip

%(2b)
%$
%\inferrule
%{
%\!
%}
%{
%x=^n_K x
%}
%\quad  %\text{ where } K \in \L, d(K) \leq n (\text{Eq}_2^{K,j})
%$
%
%\bigskip

%(3)
%$
%\inferrule
%{
%\:
%}
%{
%(\exists p \colon x=_K^j y. \top) \wedge \phi \rightarrow \widehat{\phi}[y/x]
%}
%%\quad \text{ where } K \in \L, \: d(K) = n-1
%$

%\bigskip
%
%(3b)
%$
%\inferrule
%{
%\exists p \colon x=_K^j y. \top \\ \phi
%}
%{
%\phi[y/x]
%}
%\quad \text{ where } K \in \L, \: d(K) < n-1
%$

\bigskip

%(4)

%``$\mathtt{Id}$-elim'' of ``$J$-rule'' 

%
%(5a)
%$
%\inferrule
%{
%r_K^j(p) \wedge r_K^j(q)
%}
%{
%p=^{j+1}_K q
%}
%\quad \text{ where } K \in \L^=_m, \: d(K) = n-1
%$

%(5)
\[
\inferrule
{
\:
}
{
\Gamma, x \colon K, \alpha \colon A, q \colon x \eqsort x \: \vert \: \rsort(q) \wedge \theta \seqimplies \tsort  (q, \alpha, \alpha)
}
\quad \trrule
\]

\bigskip

\[
\inferrule
{
\Gamma, x \colon K, q \colon x \eqsort x \: \vert \: 
%\exists q_1 \colon x=_Kx. 
\rsort(q) \wedge \theta[x,x,q] \seqimplies 
%\exists q_2 \colon x=_Kx. 
\phi[x,x,q]
%\Gamma, x \colon K \: \vert \: \exists q_1 \colon x=_K^1x.\exists \tau \colon r_K^1 (q_1,x).\theta[x/y,q_1/p] \vdash \exists q_2 \colon x=_K^1x.\exists \tau \colon r_K^1 (q_2,x).\phi[x/y,q_2/p]
}
{
\Gamma, x \colon K, y \colon K, p \colon x\eqsort y \: \vert \:  \theta \seqimplies \phi
}
\quad  \Jrule
\]

\bigskip

If we also include the law of the excluded middle (LEM) as an axiom then we denote the corresponding proof system by $\Disocl$. 
\end{defin}

\def\FOLisoentails{\vdash_{\cong}}
\def\FOLisoentailscl{\vdash^{\text{cl}}_{\cong}}
\def\FOLisoentailsboth{\vdash^{(\text{cl})}_{\cong}}

\def\modelsHoTT{\models_{h}}
\def\modelsHoTTcl{\models^{\text{cl}}_{h}}
\def\modelsHoTTboth{\models^{(\text{cl})}_{h}}

\begin{notation}
When what we say applies to both $\Diso$ and $\Disocl$ we will use the notation $\Diso^{(\text{cl})}$.
Analogous to Definition \ref{entailment} we write $\FOLisoentails$ (resp. $\FOLisoentailscl$) for entailment in $\Diso$ (resp. $\Disocl$) and when what we have to say applies to both systems we write $\FOLisoentailsboth$).
\end{notation}

\begin{remark}
The intuition behind the new rules is as follows.
The $\Eqintro$ rule encodes the fact that for any term $x$ of a sort $K$ there is an inhabitant $p$ of the isomorphism sort $x \eqsort x$ for which the ``reflexivity predicate'' can be asserted, i.e. there is always a trivial isomorphism from an object to itself.
The $\trrule$ rule encodes the fact that transporting along the trivial isomorphism ``does nothing'', i.e. it is the same as applying the identity ``function(al relation)''.
The $\Jrule$ rule is an adaptation of the $\Idtype$-elimination rule of type theory, and it can be understood (and thereby also justified at an intuitive level) as saying that if we wish to prove a statement about two objects and an isomorphism between them then it suffices to prove the same statement for one of these objects and the trivial isomorphism. The remarkable consequence of this rule is that it ensures that every statement is ``invariant under isomorphism'' even as it allows us to consider multiple (distinct) isomorphisms between objects, and in particular non-trivial ones.
\end{remark}

\begin{remark}
We note that the $\Jrule$ rule above corresponds to what in MLTT would be called \emph{strong} $\mathtt{Id}$-elimination since $\theta$ behaves like a contextual parameter that may itself depend on the variables $x,y,p$. In the presence of $\Pitype$-types, strong $\Idtype$-elimination is equivalent to the usual form (without a contextual parameter). Similarly, in the presence of universal quantification our rule $\Jrule$ is equivalent to the perhaps more recognizable form

\bigskip

\[
\inferrule
{
\Gamma, x \colon K, q \colon x=x \: \vert \: \rsort(q) \seqimplies
%\exists q_2 \colon x=_Kx. 
\phi[x,x,q]
}
{
\Gamma, x \colon K, y \colon K, p \colon x=_K y \: \vert \:  \top \seqimplies \phi
}
\quad  \wJrule
\]

\bigskip
\end{remark}

\begin{theo}[Soundness for homotopy semantics]\label{nlogSoundness}
Let $\T$ be a $\FOLiso$ $\L$-theory. If $\T \FOLisoentailsboth \tau$ then $\T \modelsHoTTboth \tau$.
\end{theo}
\begin{proof}
All rules of $\Diso$ that are in $\DFOLDS$ and do not involve $\exists$ or $\vee$ have been shown to be sound in Theorem \ref{soundness}.
For the rules involving $\exists$ and $\vee$ we must now use the universal property of the propositional truncation in a straightforward way.
We will do the case of the $(\exists)$ rule
\[
\myinferrule{\Gamma \v \exists x \colon K. \phi \seqimplies \psi}{\Gamma, x \colon K \v \phi \seqimplies \psi} \quad (\exists)
\]
as an illustration.
%where $y$ does not appear free in $\psi$ and where we are suppressing the variable dependencies of $K$. 
If the interpretation of the top line in $(\exists)$ is true (in some $\M$) this means that we have derived a term  
\[
 \eta \colon \underset{\begin{subarray} g\mathbf{x} \colon \Gamma^\M \\ y \colon K^\M \end{subarray}}{\Pi} \phi^\M \rightarrow \psi^\M
\]
%where $\Gamma$ is some context appropriate to both $\phi$ and $\psi$ and which does not contain $y$.
Then we can define the following term
\[
 \xi \eqdef \lambda \mathbf{x}. (\lambda \langle y, p \rangle. \eta(\mathbf{x},y)(p)) \colon \underset{\mathbf{x}\colon \Gamma^\M}{\Pi}  \:  \underset{y \colon K^\M}{\Sigma} \phi^\M  \rightarrow \psi^\M
\]
But since $\psi^\M$ will be a mere proposition for any substitution instance of its free variables, by the universal property of the propositional truncation and the fact that $\Pi$-types preserve mere propositions we obtain a term
\[
\proptrunc{\xi} \eqdef \lambda \mathbf{x}. \vert \vert (\lambda \langle y, p \rangle. \eta(\mathbf{x},y)(p)) \vert \vert \colon \underset{\mathbf{x}\colon \Gamma^\M}{\Pi}  \:  \vert \vert \underset{y \colon K^\M}{\Sigma} \phi^\M \vert \vert  \rightarrow \psi^\M
\]
This is exactly the translation of the (satisfaction of the) bottom sequent in ($\exists$). The soundness of the rest of the rules involving $\vee$ and $\exists$ follows straightforwardly, making use of the universal property of the propositional truncation when needed in the manner just sketched.

So it remains to prove soundness for the new rules $\Eqintro, \Jrule$ and $\trrule$. As before, we fix an arbitrary $\L$-structure $\M$. For $\Eqintro$ we need to show that $\M$ satisfies the sequent
\begin{equation}\label{eqintronts}
\Gamma, x \colon K \: \vert \: \theta \seqimplies \exists q \colon x=x. \rsort (q)
\end{equation}
We have 
\begin{equation}
\big(\Gamma, x \colon K \: \vert \: \theta \seqimplies \exists q \colon x=x. \rsort (q))^\M = \Pitype \: (\Gamma^\M, x \colon K^\M) \:\: \theta^\M \rightarrow \truncated{\underset{q \colon \text{Id}_{K^\M}(x,x)}{\Sigma} \text{Id}_{\text{Id}_{K^\M}(x,x)}(q,\mathtt{refl}_x)}
\end{equation}
But then we have that the term
\begin{equation}
%\Delta^\M \vdash 
\truncated{\langle \refl_x, \mathtt{refl}_{\mathtt{refl}_x} \rangle} \colon \vert \vert \underset{p \colon \text{Id}_{K^\M}(x,x)}{\Sigma} \text{Id}_{\text{Id}_{K^\M}(x,x)}(p,\mathtt{refl}_x) \vert \vert
\end{equation}
is derivable in context $\Gamma^\M, x \colon K^\M$ which means exactly that in the same context we can derive a term
\begin{equation}
\lambda y. \truncated{\langle \refl_x, \mathtt{refl}_{\mathtt{refl}_x} \rangle} \colon \theta^\M \rightarrow \truncated{\underset{p \colon \text{Id}_{K^\M}(x,x)}{\Sigma} \text{Id}_{\text{Id}_{K^\M}(x,x)}(p,\mathtt{refl}_x)} 
\end{equation}
The derivability of this term means exactly that the required sequent (\ref{eqintronts}) is satisfied by $\M$.
%follows immediately by Id-intro since if we have $\Gamma \vdash K^\M \colon \mathcal{U}$ then by Id-intro we get 
%\[
%\Gamma, x \colon K^\M \vdash \mathtt{refl}_x \colon \text{Id}_{K^\M}(x,x)
%\]
%and by $\Pi$-intro
%\[
%\Gamma \vdash \lambdax. \mathtt{refl}_x \colon \underset{x \colon K^\M}{\Pi} \text{Id}_{K^\M} (x,x)
%\]
%which by definition gives
%is derivable which gives us the required result. 
%(2) and (3) follow immediately from standard properties of identity types. 

For $\trrule$ we need to show that $\M$ satisfies the sequent
\begin{equation}\label{trulents}
\Gamma, x \colon K, q \colon x \eqsort x, \alpha \colon A \: \vert \: \rsort(q) \seqimplies \tsort  (q, \alpha, \alpha)
\end{equation}
Given that the sequent is well formed we know we have a type family
\begin{equation}
\Gamma^\M, x \colon K^\M \vdash A^\M \colon \TTisatype
\end{equation}
%We can write
%\[
%A \colon K^\M \rightarrow \typesofhlevel{h(A)}
%\]
%for the associated type family (noting our abuse of notation in considering $\typesofhlevel{h(A)}$ as a universe of types rather than as a $\Sigmatype$-type whose first projection is a type).
Then by the definition of transport in type theory we have:
% for any $x \colon K^\M$:
\begin{equation}
\Gamma, x \colon K^\M \vdash \trans_{\refl_x}^{A} \equiv  \lambda \alpha. \alpha \colon A(x) \rightarrow A(x)  
\end{equation}
which means that 
%for any $\alpha \colon A(x)$ we have
\begin{equation}
\Gamma, x \colon K^\M, \alpha \colon A(x) \vdash \trans_{\refl_x}^{A} (\alpha) \equiv \alpha
\end{equation}
and therefore also
\begin{equation}
\Gamma, x \colon K^\M, \alpha \colon A(x) \vdash \Idtype_{A(x)} (\alpha, \trans_{\refl_x}^{A} (\alpha))
\end{equation}
%is inhabited in context $\Gamma, x \colon K^\M, \alpha \colon A(x)$.
But the inhabitation of this last type is exactly equivalent to the inhabitation of the type
\begin{equation}
\big( \Idtype_{\Idtype_{K^\M}(x,x)}(q,\refl_x) \big) \rightarrow \Idtype_{A^\M(x)} (\alpha, \trans_{q}^{A} (\alpha))
\end{equation}
in context $\Gamma^\M, x \colon K^\M, q \colon \Idtype_{K^\M}(x,,x), \alpha \colon A^\M$, which means exactly that the required sequent (\ref{trulents}) is satisfied in $\M$.

For $\Jrule$, given the availability of $\Pitype$-types in type theory, we will prove $\wJrule$ for simplicity since the main idea of the argument is exactly the same.
So assume we know that the sequent
\[
\Gamma, x \colon K, q \colon x=x \: \vert \: \rsort(q) \seqimplies \phi[x,x,q]
\]
is satisfied by $\M$.
This means that we have a term
\begin{equation}
\epsilon \colon \underset{\begin{subarray}{1 }\mathbf{z} \colon \Gamma^\M \\ x \colon K^\M \\ q \colon \Idtype_{K^\M}(x,x) \end{subarray}}{\Pi} \: \: \Idtype_{\Idtype_{K^\M}(x,x)} (q, \mathtt{refl}_x) \rightarrow \phi^\M[x,x,q] 
\end{equation}
in type theory.
Hence we can define a term
\begin{equation}
\eta \eqdef \lambda \mathbf{z} \colon \Gamma^\M. \lambda x. \epsilon_{\mathbf{z},x} (\refl_{\refl_x}) \colon \underset{\begin{subarray}{1 }\mathbf{z} \colon \Gamma^\M \\ x \colon K^\M \end{subarray}}{\Pi} \phi^\M[x,x,\refl_x]
\end{equation}
But then by the elimination rule for identity types we immediately get a term of
\begin{equation}
%\mathtt{ind}_= (\bar{\eta}) \colon 
%\underset{\begin{subarray}{1 }\mathbf{w} \colon 
\Pitype \:(\Gamma^\M, x,y \colon K^\M, p \colon \text{Id}_{K^\M} (x,y)) \: \phi^\M
\end{equation}
which is of course equivalent to
\begin{equation}\label{wJrulelast}
%\mathtt{ind}_= (\bar{\eta}) \colon 
%\underset{\begin{subarray}{1 }\mathbf{w} \colon 
\Pitype \:(\Gamma^\M, x,y \colon K^\M, p \colon \text{Id}_{K^\M} (x,y)) \: \onetype \rightarrow \phi^\M
\end{equation}
But the inhabitation of (\ref{wJrulelast}) means exactly that the sequent below the line in $\wJrule$ is satisfied by $\M$, as required. Finally, we note that the validity of LEM in HoTT+LEM is immediate, essentially by definition, since $\phi^\M$ is a mere proposition for any $\L$-formula $\M$.
\end{proof}

\begin{remark}
We note that a ``propositional'' $\trrule$-rule in which transporting a term along reflexivity is only propositionally equal to itself would also be sound since in the proof of soundness above the judgmental equality $\trans_{\refl_x}^{A} (\alpha) \equiv \alpha$ is used to extract a propositional equality. For example, this would ensure that a soundness theorem for $\Diso$ can also be proven for a homotopy semantics in cubical type theory.
\end{remark}

\begin{remark}
%Given all the truncations that we've introduced 
One might wonder whether the class of (homotopy) models we are considering for $\FOLiso$ is too wide. 
%One might wonder if it is possible to prove completeness for a much narrower class of models? 
In particular, 
%since we appear not to be making any use of the full structure of identity types 
one might wonder whether $\FOLiso$ can take semantics where \emph{every} type is interpreted as a set (i.e. a $0$-type) much as in Makkai's original formulation of the semantics of FOLDS.
%in \cite{MFOLDS}. 
This is not the case. 
%Interpreting our ground sorts as $n$-types is essential. This is
$\FOLiso$ does in fact have the expressive power to force a theory to have only models whose ground sorts are $n$-types. As the simplest possible illustration, take the $h$-signature $\L_{K,1}$ as in Example \ref{LKone}.  Consider the (full first-order) $\L_{K,1}$-theory $\T_O$ consisting of the single axiom
\[
\phi \eqdef \forall x,y \colon O. \exists p, q \colon x\eqsort_O y. \neg (p\eqsort_O q)
\]
with the obvious abbreviations.
Every model of $\T_O$ where $O$ is interpreted as an $h$-set (resp. discrete groupoid) falsifies $\phi$. But $\T_O$ is satisfiable: simply take an $\L$-structure where $O$ is interpreted as a (proper) $1$-type. Given the soundness theorem above this means that $\phi$ cannot be disproved by $\Diso$ even though it is not satisfied in any set-model of $\T_O$. As such, set models are not sufficient to describe provability for $\FOLiso$. 
\end{remark}

We now indicate how the rules of the proof system have the consequences that our homotopy semantics demands.

\begin{prop}[``Substitution Salva Veritate'']\label{transportprop}
For any $\L$-formula $\phi$ in context $\Gamma, x \colon K$ the following sequent is derivable
\[
\Gamma, x \colon K, y \colon K \: \vert \: x=y \wedge \phi \seqimplies \phi[y/x]
\]
\end{prop}
\begin{proof}
Note that since $x,y$ are declared last in the context there can be no free variable in $\phi$ which depends on either, which means that the substitution $\phi[y/x]$ is well-defined. This means in particular that no $q \colon x =x$ can appear free in $\phi$ which implies that we have 
\begin{equation}\label{phiyxaequivphi}
\phi[y/x][x,x,q] \equiv \phi[x,x,q]
\end{equation}
With this in mind we have the following derivation, starting from $\idenrule$:
\[
\inferrule
{
\Gamma, x \colon K, q \colon x=x \: \vert \: \phi[x,x,q] \seqimplies \phi[x,x,q]
}
{
\inferrule
{
\Gamma, x \colon K, q \colon x=x \: \vert \: \rsort(q) \wedge \phi[x,x,q] \seqimplies \phi[x,x,q]
}
{
\inferrule
{
\Gamma, x \colon K, q \colon x=x \: \vert \: \rsort (q) \wedge \phi[x,x,q] \seqimplies \phi[y,x][x,x,q]
}
{
\inferrule
{
\Gamma, x \colon K, y \colon K, p \colon x = y \: \vert \: \phi \seqimplies \phi[y/x]
}
{
\Gamma. x \colon K, y \colon K \: \vert \: x= y \wedge \phi \seqimplies \phi[y/x]
}
\quad (\exists)
}
\quad \Jrule
}
\quad (\ref{phiyxaequivphi})
}
\quad \text{($\wedge$-intro)}
\]
where ($\wedge$-intro) is the obvious derived rule.
\end{proof}
%\begin{proof}
%%Analogous to Proposition \ref{topprop}, and omitted.
%%By (iden) we know that 
%%The sequent 
%%$
%%\Gamma, x \colon K \: \vert \:  \phi[x/y] \vdash \phi[y/x][x/y]
%%$
%%is derivable (from no premises) since $\phi[x/y] \equiv \phi[y/x][x/y]$. But then we can just repeat the (first three steps of) the derivation in the proof of Proposition \ref{topprop} with $\theta \equiv \phi[x/y]$ and $\phi \equiv \phi[y/x]$ to get the desired result.
%\end{proof}

%\begin{remark}
%Proposition \ref{transportprop} have not made use of the transport relations $\tsort$, but the rule $(\tsort\rsort)$ could be used to provide alternative proofs, more similar in style to the usual proofs in MLTT.
%\end{remark}

%We can also obtain in $\D$ the expected result that the transport relations $\tau_f$ are indeed functional relations that define equivalences.

Together with $\Eqintro$, Proposition \ref{transportprop} gives us the usual rules for identity in any first-order sequent calculus (e.g. as it is presented in \cite{Elephant}), as the following record.

\begin{cor}[``Mere equality is an equivalence relation'']\label{standardequality}
The formula 
\[
x\eqsort_Ky \: \eqdef \: \exists p \colon x\eqsort_Ky. \top
\]
is reflexive, symmetric and transitive.
\end{cor}

But one could now wonder whether Proposition \ref{transportprop} and Corollary \ref{standardequality} generalize to sorts of higher $h$-level, i.e. whether we can prove that the inhabitants of isomorphism sorts also behave like the isomorphisms of (higher) groupoids (rather than merely recording the usual extensional equality).
In particular we can ask whether logical sorts equip non-logical sorts of $h$-level $m$ with the structure of an $m$-groupoid. 
This is indeed the case with the transport structure providing a notion of composition that is invisible at lower $h$-levels.
We will here prove the case for sorts $K$ of $h$-level $3$, with the basic idea being the following:
\begin{itemize}
\item $K$ is the sort of objects of the groupoid
\item $\eqsort_K$ is the sort of isomorphisms
\item $\eqsort_{\eqsort_K}$ is the equality on isomorphisms
\item $\rsort_K$ is the predicate that picks out the trivial isomorphism
\item $\tsort_{s_K}$ is the composition (functional) relation
\end{itemize}
We will now make this basic idea precise.

\begin{prop}[``Transport is an equivalence'']\label{isequivtransport}
The transport sorts $\tau (p,\alpha,\beta)$ 
%for any $f \colon A \rightarrow K$ 
in a $\FOLiso$-signature $\L$ define a functional relation that is also an equivalence, in the sense that the following sequents are derivable:
\begin{enumerate}[(1)]
\item $\Gamma, p \colon x \eqsort y, \alpha \colon A(x, \mathbf{w}) \:\vert\: \top \seqimplies \exists \beta \colon A(y, \mathbf{w}). \tsort (p, \alpha, \beta)$
\item $\Gamma, \alpha \colon A(x, \mathbf{w}), p \colon x \eqsort y, \beta, \beta' \colon A(y, \mathbf{w}) \:\vert\: \tsort (p, \alpha, \beta) \wedge \tsort (p, \alpha, \beta') \seqimplies  \beta \eqsort \beta' $
\item $\Gamma, p \colon x \eqsort y, \alpha, \alpha' \colon A(x, \mathbf{w}), \beta \colon A(y, \mathbf{w}) \:\vert\: \tsort (p, \alpha, \beta) \wedge \tsort (p, \alpha', \beta) \seqimplies \alpha \eqsort \alpha'$
\item $\Gamma, p \colon x \eqsort y, \beta \colon A(y, \mathbf{w}) \:\vert\: \top \seqimplies \exists \alpha \colon A(x,\mathbf{w}). \tsort (p, \alpha, \beta)$
\end{enumerate}
\end{prop}
\begin{proof}
All four sequents follow by a straightforward application of $\trrule$ and $\Jrule$. We do (1) as an illustration and leave the rest to the reader:
\[
\inferrule
{
\:
}
{
\inferrule
{
\Gamma, q \colon x \eqsort x, \alpha \colon A(x, \mathbf{w}) \:\vert\: \rsort (q) \seqimplies \tau(q,\alpha,\alpha)
}
{
\inferrule
{
\Gamma, q \colon x \eqsort x, \alpha \colon A(x, \mathbf{w}) \:\vert\: \rsort (q) \seqimplies \exists \beta \colon A(x,\mathbf{w}). \tau(q,\alpha,\beta)
}
{
\Gamma, p \colon x \eqsort y, \alpha \colon A(x, \mathbf{w}) \:\vert\: \top \seqimplies \exists \beta \colon A(y, \mathbf{w}). \tsort (p, \alpha, \beta)
}
\Jrule
}
(\exists\text{-intro})
}
\trrule
\] 
Here once again $(\exists\text{-intro})$ signifies the obvious derived rule of existential instantiation.
\end{proof}

\begin{lemma}\label{isolemma}
Let $\L_{K,3}$ be the following $h$-signature:
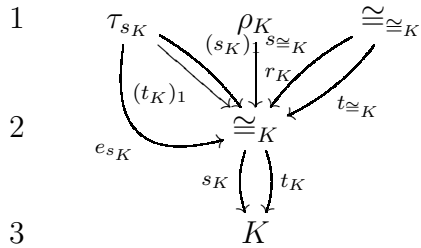
\[
\xymatrix{
1 &\tsort_{s_K} \ar@/^5pt/[rd]^{(s_K)_1} \ar[rd]_{(t_K)_1} \ar@/_30pt/[rd]_{e_{s_K}} & \rsort_K \ar[d]^{\rmor_K} & \eqsort_{\eqsort_K} \ar@/_/[ld]_{s_{\eqsort_K}} \ar@/^/[ld]^{t_{\eqsort_K}} \\
 2 & & \eqsort_K \ar@/_/[d]_{s_K} \ar@/^/[d]^{t_K} \\
 3 & & K \\
}
\]
Then $\L_{K,3}$ is isomorphic, as an $h$-signature, to $\Lcat$.
\end{lemma}
\begin{proof}
Let $I \colon \L_{K,3} \rightarrow \Lcat$ be the functor (indeed $h$-morphism) that on objects sends $K \mapsto O$, $=_K \mapsto A$, $\tsort_{s_K} \mapsto \circ$, $\rsort_K \mapsto I$, $=_{=_K} \mapsto =_A$ and on arrows sends $(s_K)_1 \mapsto t_2$, $(s_K)_2 \mapsto t_1$ and $e_{s_K} \mapsto t_0$.
By inspection one can immediately check that the relations between arrows imposed by the logical sorts in $\L_{K,3}$ are exactly the same as the stipulated relations in $\Lcat$ under this mapping, thus making $I$ a full and faithful functor bijective on objects.
\end{proof}

\begin{terminology}
Lemma \ref{isolemma} allows us to speak of the \emph{FOLDS $\L_{K,3}$-theory of groupoids} whose axioms are those of $\Tgpd$ but expressed over the isomorphic FOLDS signature $\L_{K,3}$ rather than over $\Lcat$ in the obvious way.
We can then take the $\FOLiso$-signature $\repletion{\L_{K,3}}$ associated to $\L_{K,3}$ and thus obtain both the FOLDS $\repletion{\L_{K,3}}$-theory of groupoids (over $\DFOLDS$) and the empty $\repletion{\L_{K,3}}$-theory (over $\Diso$).
We say that two theories (understood as a set of axioms together with a deductive system) are \emph{logically equivalent} if there is an isomorphism between their signatures such that a sentence is derivable from one theory if and only if it is derivable from the other.
%\textcolor{red}{$\L_{K,3}$-translates} of their $\Lcat$ counterparts under the \textcolor{red}{translation} $I$.
\end{terminology}

With the above terminology in mind we can now make precise the sense in which sorts of $h$-level $3$ in $\FOLiso$ are groupoids.

\begin{prop}[``Sorts of $h$-level $3$ are groupoids'']\label{groupoidprop}
The FOLDS $\repletion{\L_{K,3}}$-theory of groupoids (over $\DFOLDS$) is logically equivalent to the empty $\repletion{\L_{K,3}}$-theory (over $\Diso$).
\end{prop}
\begin{proof}
We will use the notation $\cong$ for $=_K$ and $=$ for $=_{=_K}$ in $\L_{K,3}$ in order to hint at the intended interpretation.
% and write simply $s$ for $s_K$.
Furthermore, for $p \colon x \cong y, u \colon x \cong z, q \colon y \cong z$ we will use the notation $$\circ(p,q,u) \eqdef \tsort_{s} (p,u,q)$$ in order to understand $u$ as the composite of $p$ and $q$ (even though $q$ is the transport along $p$ of $u$ in the intended interpretation of $\tsort_s$).

First, we check that all the axioms for the $\L_{K,3}$-theory of groupoids are derivable using only the rules of $\D$:
\begin{enumerate}[(1)]

\item(Equality is an equivalence relation satisfying substitution salva veritate)

As already noted, this follows from Proposition \ref{transportprop}.

\item(Existence of identities) 

$\forall x \colon K. \exists i \colon x \cong x. \rsort (i)$ 
%\item $\forall x,y \colon O \forall f \colon A(x,y) \exists \epsilon \colon =_A(f,f,x,y).\top$
%\item $\forall x,y \colon O \forall f,g \colon A(x,y) \forall \epsilon \colon =_A(f,g,x,y) \exists \epsilon_2 \colon =_A(g,f,x,y).\top$
%\item $\forall x,y \colon O \forall f,g,h \colon A(x,y) \forall \epsilon \colon =_A(f,g,x,y) \forall \epsilon_2 \colon =_A(g,h,x,y) \exists \epsilon_3 \colon =_A(f,h,x,y).\top$

This follows immediately from $\Eqintro$.

\item(Functionality of composition-1)

$\forall x,y,z \colon K. \forall f \colon x \cong y. \forall g \colon y \cong z. \exists h \colon x \cong z.  \circ(f,g,h)$ 

This is a special case of Proposition \ref{isequivtransport}.

\item (Functionality of Composition-2)
%\begin{align*}

$\forall x,y,z \colon K. \forall f \colon x \cong y. \forall g \colon y \cong z. \forall h,h' \colon x \cong z. 
 \circ (f,g,h) \wedge \circ (f,g,h') \rightarrow h=h'$
 % \\ &\exists \epsilon \colon =_A(h,h',x,z)
 %&((\exists \tau_1 \colon \circ (f,g,h,x,y,z). \exists \tau_2 \colon \circ (f,g,h',x,y,z).\top) \rightarrow h=h')
%\end{align*}

This is once again a special case of Proposition \ref{isequivtransport}.

\item (Right unit)

$\forall x,y \colon K. \forall i \colon x \cong x. \forall g \colon x \cong y. \rsort(i) \rightarrow \circ (i,g,g)$ 

This is exactly the $\trrule$ rule.

\item (Left unit) 

$\forall x,y \colon K. \forall i \colon y \cong y. \forall f \colon x \cong y. \rsort(i) \rightarrow \circ(f,i,f)$ 

This follows by a straightforward application of the $\Jrule$ rule.

\item(Uniqueness of identity)

$\forall x \colon K. \forall i,j \colon x \cong x. \rsort(i) \wedge \rsort(j) \rightarrow i=j$ 
%\item $\forall x \colon O \forall i,j \colon A(x,x) \forall \phi \colon I(i,x,x) \forall \epsilon \colon =_A(i,j,x,x) \exists \psi \colon I(j,x,x). \top$

By Proposition \ref{isequivtransport}.(2) it suffices to show 
\begin{equation}
\forall x \colon K. \forall i, j \colon x \cong x. \rsort(i) \wedge \rsort (j) \rightarrow \circ(i,j,j) \wedge \circ(i,j,i)
\end{equation}
But this follows from the left and right unit axioms above.

\item(Associativity)
\[
\begin{split}
&\forall x,y,z,w \colon K. \forall f \colon x \cong y. \forall g \colon y \cong z. \forall h \colon z \cong w. \forall i \colon x \cong z. \forall j \colon x \cong w. \\ &\forall k \colon y \cong w. \circ(f,g,i) \wedge \circ(i,h,j) \wedge \circ (g,h,k) \rightarrow \circ (f,k,j)
\end{split}
\]

By the $\Jrule$ rule it suffices to prove 
\[
\circ(\rsort(f),g,g) \wedge \circ(g,\rsort(h),g) \wedge \circ (g,\rsort(h),g) \rightarrow \circ (\rsort(f),g,g)
\]
which follows immediately from the (already proven) left and right unit axioms.
%\item 
%\[
%\begin{split}
%&\forall x,y,z \colon O \forall f,f' \colon A(x,y) \forall \epsilon_1 \colon =_A(f,f',x,y) \forall g, g' \colon A(y,z) \forall \epsilon_2 \colon =_A(g,g',x,y) \\ &\forall h, h' \colon A (x,z) \forall \tau_1 \colon \circ (f,g,h,x,y,z) \forall \tau_2 \colon \circ (f', g' ,h', x,y,z) \exists \epsilon_3 \colon =_A(h,h',x,z)
%\end{split}
%\]
%\item \[
%\begin{split}
%&\forall x,y,z \colon O \forall f,f' \colon A(x,y) \forall \epsilon_1 \colon =_A(f,f',x,y) \forall g, g' \colon A(y,z) \forall \epsilon_2 \colon =_A(g,g',x,y) \\ &\forall h, h' \colon A (x,z) \forall \tau_1 \colon \circ (f,g,h,x,y,z)  \forall \epsilon_3 \colon =_A(h,h',x,z) \exists \tau_2 \colon \circ (f', g' ,h', x,y,z)
%\end{split}
%\]

\item (Every arrow is an isomorphism)

\newcommand{\isiso}[1]{\mathtt{isiso}(#1)}

$\forall x,y \colon K. \forall p \colon x \cong y. \exists q \colon y \cong x. \exists u \colon x \cong x. \exists v \colon y \cong y. \circ(p,q,\rsort(u)) \wedge \circ (q,p, \rsort(t))$

Let us abbreviate the formula that begins with the existential quantifier as $\isiso{p}$. Then we have:
\[
\isiso{p}[x,x,q'] \equiv \exists q'', u' \colon x \cong x. \exists v' \colon x \cong x. \tsort_s(\rsort(u'),q'',q') \wedge  \tsort_s(\rsort(v'), q', q'')
\]
By the $\Jrule$ rule it now suffices to show that $\isiso{p}[x,x,q']$ is derivable. Indeed, we have:
\[
\inferrule
{
\rsort(q) \rightarrow \exists u' \colon x \cong x. \exists v' \colon y \cong y. \rsort(u') \wedge \rsort(v')
}
{
\inferrule
{
\rsort(q) \rightarrow \exists u',q'' \colon x \cong x. \exists v' \colon y \cong y.\rsort(v') \wedge \tsort_s (\rsort(u'),\rsort(q''),q)
}
{
\inferrule
{
\exists q'', u' \colon x \cong x. \exists v' \colon x \cong x. \tsort_s(\rsort(u'),q'',q') \wedge  \tsort_s(\rsort(v'), q', q'')
}
{
\isiso{p}[x,x,q']
}
%\quad (\exists\text{-intro})
}
\quad \trrule, \ref{isequivtransport}
}
\quad \trrule, \Eqintro, \ref{isequivtransport}
%\quad \Eqintro
\]

\end{enumerate}

Conversely, we must show that the rules $\Eqintro, \trrule,\Jrule$ of $\Diso$ are derivable from the $\L_{K,3}$-theory of groupoids over $\DFOLDS$. But note that the axioms $\Eqintro$ and $\trrule$ are direct translations of axioms of the theory of groupoids. And the $\Jrule$ rule follows from the well-known result of Makkai in \cite{MFOLDS} that the FOLDS $\Lcat$-formulas are precisely the formulas that are invariant under isomorphism of objects in a category.
\end{proof}

We expect the higher analogues of Proposition \ref{groupoidprop} to also be true.
 %and in particular we expect that the empty theory over an analogously defined $\L_{K,\infty}$ to be logically equivalent to a (suitable) theory of $\infty$-groupoids. 
 In particular, we expect that the empty $\FOLiso$-theory over the signature $\repletion{\L_{K,\infty}}$ associated to the $h$-signature with only one object $K$ of $h$-level $\infty$ to be logically equivalent to a (suitable) theory of $\infty$-groupoids.
Making this a precise statement and proving it is left for future work. 
But even as things stand, $\Diso$ allows us to formulate a new \emph{definition} of an $\infty$-groupoid: it is a model of the empty $\FOLiso$-theory over $\repletion{\L_{K,\infty}}$.
%%%TODO if it can be made cogent%%%
%But for now let us remark that if one takes HoTT to be such a ``suitable theory of $\infty$-groupoids$'' (as one certainly should) then the Soundness Theorem for $\FOLiso$ already shows that any derivable $\FOLiso$-formula over $\L_{K,\infty}$ will be true in HoTT.
%%%TODO if it can be made cogent%%%

%\section{Soundness}\label{nlogSandC}

%At this point we have obtained all the components traditionally required of a logic: a syntax, a semantics and a proof system. 
%We may thus begin to investigate the relationship between these components.
%We will now prove soundness for the rules of $\Diso$ with respect to the homotopy semantics of Section \ref{nlogSemantics}. 

%\begin{notation}
%For any $\L$-theory $\T$ we write $\T \vdash \phi$ (resp. $\T \vdash_{\text{cl}} \phi$) to denote that $\phi$ is derivable in $\Diso$ (resp. $\Disocl$) from (a finite subset of)
% the sentences of $\T$. We write $\T \models \phi$ (resp. $\T \models_{\text{cl}} \phi$) to denote that $\phi$ is true in all (homotopy) models of $\T$.
%Since we are going to prove soundness of our system in a type theory that has $\Pitype$-types we will interpret sequents in the manner outlined in Remark \ref{pisequent} and for any context $\Gamma^\M$ we will write $\underset{\mathbf{x} \colon \Gamma^\M}{\Pitype}$ and $\underset{\mathbf{x} \colon \Gamma^\M}{\Pitype}$ to indicate the binding of all the variables in the context (in the appropriate order).
%\end{notation}

\section{Examples and Applications}\label{nlogExamples}

%As explained in Section 1 of Chapter \ref{philcomp},
FOLDS was invented as a systematic way of avoiding the use of equalities that are irrelevant for the structures of interest, e.g. equality between objects when we care about categories only up to equivalence. 
% as a language-first approach to invariance: we first cut down the expressive power of a given language and then define the relevant notion of identity (the notion that preserves only sentences in our invariant language).
On the other hand, $\FOLiso$ represents a partial reversal of this idea, since we are re-introducing equalities as logical sorts with a fixed interpretation, together with transport function(al relation)s and reflexivities. 
In this final section we examine applications and interesting examples that indicate that this partial reversal is useful in the setting of the Univalent Foundations.

%Now, l
Let $\Lprecat$ be the $h$-signature described in Example \ref{Lprecateg}.
%\[
%\xymatrix{
%1 &\circ \ar@/^5pt/[rd]^{t_0} \ar[rd]_{t_1} \ar@/_20pt/[rd]_{t_2} & I \ar[d]^{i} & =_A \ar@/_/[ld]_{e_1} \ar@/^/[ld]^{e_2} \\
% 2 & & A \ar@/_/[d]_{d_0} \ar@/^/[d]^{d_1} \\
% \infty & & O \\
%}
%\]
%subject to the relations
%\[
%d_0t_0=d_0t_2, d_1t_1=d_1t_2, d_0t_1=d_1t_0
%\]
%\[
%d_0i=d_1i
%\]
%\[
%d_0e_1=d_0e_2, d_1e_1=d_1e_2
%\]
%and with $d(O)=\infty$, $d(A) = 0$ and $d(I)=d(\circ)=-1$.
The $\Lprecat$-\emph{theory of precategories} $\mathbb{T}_{\text{precat}}$ consists of the axioms laid out in Example \ref{Tcat}.
A homotopy $\Lprecat$-structure thus consists of the following data:
\begin{itemize}
\item A type $O \colon \mathcal{U}$
\item A dependent function $A \colon O \rightarrow O \rightarrow \textbf{Set}_{\mathcal{U}}$
\item A dependent function $I \colon \underset{x\colon O}{\Pi} A(x, x) \rightarrow \textbf{Prop}_{\mathcal{U}}$
%Applying the relation $d_0i=d_1i$ given to us in $\Lcat$, this means that
%\[
%\underset{x_{d_0i}, x_{d_1i} \colon O}{\Pi} A(x_{d_0i}, x_{d_1i}) \equiv \underset{x \colon O}{\Pi} A(x, x)
%\]
%Therefore $I$ is actually a term of type .
\item A dependent function
$\circ \colon \underset{x,y,z \colon O}{\Pi} A(x, y) \rightarrow A(y,z) \rightarrow A(x, z) \rightarrow \textbf{Prop}_{\mathcal{U}}$
%where the number of variables has been reduced by a similar consideration of the relations that hold in $\Lcat$.
%\item We define $=_A$ to be the identity type on $A(x,y)$ for any $x,y \colon O$. 
%So the expression $E(f,g,a,b)$ where $f,g \colon A (a,b)$ is to be translated to $f=_{A(a,b)} g$. A usual, we will suppress the subscript on the equality sign.
\end{itemize}

%One can probably already see how some of the axioms in Definition \ref{cataxioms} will be automatically satisfied by virtue of our translation scheme. 

%So let us now 
We can now translate the axioms of $\Tprecat$ into types in HoTT for an arbitrary model $\M$ of $\Tcat$. We will list them in order, writing $=$ for the identity type on $A^\M$ and omitting $\M$ from superscripts for readability:
%$T_i$ for the $\Lcat$-translate of each axiom (i) and we omit superscripting every piece of data with $\M$ and also omit the subscript $A$ in $=_A$:
\begin{multicols}{2}
\begin{enumerate}
\item[($T_1$)] $\underset{x \colon O}{\Pi} \truncated{\underset{i \colon A (x,x)}{\Sigma} I(i,x)}$
%\item[($T_2$)] $ \underset{\begin{subarray}{1} x, y \colon O \\ f \colon A(x,y) \end{subarray}}{\Pi} f=f$
%\item[($T_3$)] $ \underset{\begin{subarray}{1} x, y \colon O \\ f,g \colon A(x,y) \\ \epsilon \colon f=g \end{subarray}}{\Pi} g=f$
%\item[($T_4$)] $ \underset{\begin{subarray}{1} x, y \colon O \\ f, g, h \colon A(x,y) \\ \epsilon_1 \colon f=g \\ \epsilon_2 \colon g=h  \end{subarray}}{\Pi} f=h$ 
%\vert \vert \circ(f,g,h,x,y,z) \times \circ(f,g,h',x,y,z) \vert \vert  \rightarrow 
\item[($T_2$)] $ \underset{\begin{subarray}{1} x, y, z \colon O \\ f \colon A(x,y) \\ g \colon A(y,z) \end{subarray}}{\Pi} || \underset{h \colon A(x,z)}{\Sigma} \circ(f,g,h,x,y,z)||$  
\item[($T_3$)] \hspace{-0.5cm} $\underset{\begin{subarray}{1}  \hspace{0.8cm} x,y,z \colon O \\ \hspace{0.8cm} f \colon A(x,y) \\ \hspace{0.8cm} g \colon A(y,z) \\ \hspace{0.5cm} h,h' \colon A(x,z) \\ \tau_1 \colon \circ(f,g,h,x,y,z) \\ \tau_2 \colon \circ(f,g,h',x,y,z)\end{subarray}}{\Pi} \:  \hspace{-0.3cm} h=h'$
\item[($T_4$)] $ \underset{\begin{subarray}{1} x, y, z, w \colon O \\ f \colon A(x,y) \\ g \colon A(y,z) \\ h \colon A(z,w) \\ i \colon A(x,z) \\ j \colon A(x,w) \\ k \colon A(y,w)  \end{subarray}}{\Pi} \: \underset{\begin{subarray}{1} \tau_1 \colon \circ(f,g,i,x,y,z) \\ \tau_2 \colon \circ(i,h,j,x,z,w) \\ \tau_3 \colon \circ(g,h,k,y,z,w)  \end{subarray}}{\Pi} \circ(f,k,j,x,y,w)$
\end{enumerate}
\begin{enumerate}
\item[($T_5$)] $\underset{\begin{subarray}{1} x \colon O \\ i,j \colon A(x,x) \\ \phi \colon I(i,x) \\ \psi \colon I(j,x) \end{subarray}}{\Pi} i=j$
%\item[($T_8$)] $\underset{\begin{subarray}{1} x \colon O \\ i,j \colon A(x,x) \\ \phi \colon I(i,x,x) \\ \epsilon \colon i=j \end{subarray}}{\Pi} I(j,x,x) $
\item[($T_6$)] $\underset{\begin{subarray}{1} x,y \colon O \\ i \colon A(x,x) \\ g \colon A(x,y) \\ \phi \colon I(i,x) \end{subarray}}{\Pi} \circ(i,g,g,x,x,y) $
\item[($T_{7}$)]  $\underset{\begin{subarray}{1} x,y \colon O \\ i \colon A(y,y) \\ f \colon A(x,y) \\ \phi \colon I(i,y) \end{subarray}}{\Pi} \circ(f,i,f,x,y,y) $
%\item[($T_{11}$)] $\underset{\begin{subarray}{1} x,y,z \colon O \\ f,f' \colon A(x,y) \\ g,g' \colon A(y,z) \\ \epsilon_1 \colon f=f' \\  \epsilon_2 \colon g=g' \end{subarray}}{\Pi} \: \underset{\begin{subarray}{1} h,h' \colon A(x,z) \\ \tau_1 \colon \circ(f,g,h,x,y,z) \\ \tau_2 \colon \circ(f',g',h',x,y,z)  \end{subarray}}{\Pi} h=h'$
%\item[($T_{12}$)] $\underset{\begin{subarray}{1} x,y,z \colon O \\ f,f' \colon A(x,y) \\ g,g' \colon A(y,z) \\ \epsilon_1 \colon f=f' \\  \epsilon_2 \colon g=g' \end{subarray}}{\Pi} \: \: \underset{\begin{subarray}{1} h,h' \colon A(x,z) \\ \tau_1 \colon \circ(f,g,h,x,y,z) \\ \epsilon_3 \colon h=h'  \end{subarray}}{\Pi} \circ(f',g',h',x,y,z) $
\end{enumerate}
\end{multicols}

%We now make the following observations. $T_2,T_3,T_4$ are always inhabited in MLTT by the standard properties of identity types that ensure that it is an equivalence relation. 
%For instance, the term $\lambda (x,y,f).\mathtt{refl}_f$ will always be a term of $T_2$, regardless of our choice of $A$ and $O$. 
%Similarly, $T_8$ and $T_{12}$ will always be inhabited. 
%because we can transport $\phi$ along $\epsilon$ producing a term of type $I(j,x,x)$ and similarly we can transport $\tau_1$ along $\epsilon_1$, $\epsilon_2$ and $\epsilon_3$ to produce a term of type $\circ(f',g',h',x,y,z)$. 
%Thus we will only consider axioms $T_1,T_5,T_6,T_7,T_9,T_{10}, T_8$ and $T_{11}$ and define
We thus obtain:
\[
%\begin{split}
\textbf{Mod}(\mathbb{T}_{\text{precat}}) \equiv \sum_{\substack{O \colon \mathcal{U} \\ A \colon O \rightarrow O \rightarrow \textbf{Set}_{\mathcal{U}} \\  I \colon \underset{x\colon O}{\Pi} A(x, x) \rightarrow \textbf{Prop}_{\mathcal{U}} \\  \circ \colon \underset{x,y,z \colon C}{\Pi} A(x, y) \rightarrow A(y,z) \rightarrow A(x, z) \rightarrow \textbf{Prop}_{\mathcal{U}}}} T_1 \times T_2 \times T_3 \times T_4 \times T_5 \times T_6 \times T_7 
%\times T_{11}
%\end{split}
\]
%Let us call a term of type $\textbf{Mod}(\mathbb{T}_{\text{cat}})$ a $\mathbb{T}_{\text{cat}}$-model. 
On the other hand in \cite{HTT} a \emph{precategory} is defined by the following data:
\begin{enumerate}[(1)]
\item A type $C \colon \mathcal{U}$  %(``objects'') 
\item A dependent function $\text{Hom}_{\C} \colon C \rightarrow C \rightarrow \textbf{Set}_{\mathcal{U}}$ %(``Hom-sets'')
\item 
A dependent function 
$
1 \colon \underset{a \colon C}{\Pi} \text{Hom}_{\C} (a, a)
$
% (``identity'')
%- we write, naturally, $1_a$ for the value of $1$ at a term $a$, i.e. for $\mathtt{app}(1,a)$.) 
\item A dependent function
$
\circ \colon \underset{a,b,c \colon C}{\Pi} \text{Hom}_{\C} (a, b) \rightarrow \text{Hom}_{\C} (b, c) \rightarrow \text{Hom}_{\C} (a, c)
$ 
%(``composition'')
\item A dependent function
$
\mathtt{assoc} \colon \underset{a,b,c,d \colon C}{\Pi} \: \underset{\begin{subarray}{c} 
f \colon \text{Hom}_{\C} (a, b) \\
g \colon \text{Hom}_{\C} (b, c) \\
h \colon \text{Hom}_{\C} (c, d)
\end{subarray}}{\Pi} h \circ (g \circ f) = (h \circ g) \circ f
$
%which witnesses (strict) associativity
\item A dependent function
$
\mathtt{ident} \colon \underset{a,b \colon C}{\Pi} \underset{f \colon \text{Hom}_{\C} (a, b)}{\Pi} (f \circ 1_a = f) \times (1_b \circ f = f)
$
%which witnesses right and left cancellability of identity maps.
\end{enumerate}

%We are now in a position to ask whether any $\mathbb{T}_{\text{cat}}$-model contains enough data to produce a precategory in UF. 
We can now show that precategories are ``$\FOLiso$-elementary'' in the sense  that they are axiomatizable, up to equivalence, by $\Tcat$ over the $\FOLiso$-signature $\Lcat$.
In what follows below we will be making free use of the HoTT version of the Axiom of Unique Choice (AUC) (\cite{HTT}, Corollary 3.9.2). 
%Proposition \ref{modtoprecat} below
%may be understood as saying that the type of 
%says that precategories is ``$\infty$-elementary'' in the sense that they are axiomatizable, up to equivalence, by a theory over an $\infty$-signature.

%\[
%\colon \underset{a,b,c,d \colon C}{\Pi} \underset{\begin{subarray}{c} 
%f \colon \text{Hom}_{\C} (a, b) \\
%g \colon \text{Hom}_{\C} (b, c) \\
%h \colon \text{Hom}_{\C} (c, d)
%\end{subarray}}{\Pi} h \circ (g \circ f) = (h \circ g) \circ f
%\]

\begin{prop}\label{modtoprecat}
$\textbf{\emph{PreCat}} \simeq \textbf{\emph{Mod}}(\mathbb{T}_{\text{\emph{precat}}})$
%Each $C \colon \textbf{\emph{Mod}}(\mathbb{T}_{\text{cat}})$ carries the structure of a precategory.
\end{prop}
\begin{proof}
The proof boils down to proving, using AUC, that an axiomatization of a category in terms of a relation of composition is equivalent to the usual axiomatization in terms of an operation of composition.
% where the operation is defined from the relation in a standard manner. 
First we define a function
\begin{equation}
p \colon \textbf{Mod}(\mathbb{T}_{\text{precat}}) \rightarrow \textbf{PreCat}
\end{equation}
So let $C  \colon \textbf{Mod}(\mathbb{T}_{\text{precat}})$ and write $t_i$ for the inhabitants of each axiom $T_i$ that is part of the data of $C$. 
We need to provide the data for conditions (1)-(6) in the definition of precategories. 
We are given $O$ and $A$ and those immediately take care of conditions (1) and (2).
 For condition (3) we first observe that $\underset{i \colon A (x,x)}{\Sigma} I(i,x)$ is a mere proposition for any $x \colon O$. For suppose that $( i,\phi )$ and $( j, \psi )$ are two terms of type $\underset{i \colon A (x,x)}{\Sigma} I(i,x,x)$. 
 %By Theorem 2.7.2 in \cite{HTT} 
To show that $ ( i,\phi ) = ( j,\psi )$ it suffices to show that there is $p \colon i=j$ and that $p_* (\phi) = \psi$. By applying $t_5$ to the data $ \langle x,i,j,\phi,\psi \rangle$ we get a proof that $i=j$, i.e. a term $p \colon i=j$. Clearly, since $I(i,x,x)$ and $I(j,x,x)$ are mere propositions, we also get that $\psi = p_* (\phi)$ and therefore we get  that $(i,\phi ) = (j,\psi)$ and therefore that the type $\underset{i \colon A (x,x)}{\Sigma} I(i,x,x)$ is a mere proposition. By AUC and $t_1$ we get a term 
\begin{equation}
u \colon \underset{x \colon O}{\Pi} \:\: \underset{i \colon A (x,x)}{\Sigma} I(i,x,x)
\end{equation}
Thus we can define, for each $x \colon O$, the following term
%the following term of type $A(x,x)$:
\begin{equation}
1_x =_{\text{def}} \TTproj_1 (u_x) \colon A(x,x)
\end{equation}
and thus we obtain a term
\begin{equation}
1^C =_{\text{def}} \lambda x. 1_x \colon \underset{x \colon O}{\Pi} A(x,x)
\end{equation}
as required by condition (3).
Condition (4) follows similarly and we omit the details. 
%By an identical argument from ${11}$ we get that $ \underset{h \colon A(x,z)}{\Sigma} \circ(f,g,h,x,y,z)$ is a mere proposition and therefore that $T_5$ is equivalent to 
%\[
%\underset{\begin{subarray}{1} x, y, z \colon O \\ f \colon A(x,y) \\ g \colon A(y,z) \end{subarray}}{\Pi}  \underset{h \colon A(x,z)}{\Sigma} \circ(f,g,h,x,y,z)
%\]
%Thus from $t_5$ we get a term $c \colon \underset{\begin{subarray}{1} x, y, z \colon O \\ f \colon A(x,y) \\ g \colon A(y,z) \end{subarray}}{\Pi}  \underset{h \colon A(x,z)}{\Sigma} \circ(f,g,h,x,y,z)$ and we use it to define an operation of composition for any $f \colon A(x,y)$ and $g \colon a(y,z)$ as follows:
%\[
%\circ^C =_{\text{def}} \lambda(x,y,z).(\lambda f. \lambda g. \mathtt{pr}_1(c_{x,y,z,f,g})) \colon  \underset{x,y,z \colon C}{\Pi} A(x, y) \rightarrow A(y,z) \rightarrow A(x, z)
%\]
%thus satisfying condition (4).
For condition (5), let $x,y,z,w \colon O$ and $f \colon A(x,y), g\colon A(y,z)$ and $h \colon A(z,w)$. Now let 
\begin{equation}
%\begin{split}
r =_{\text{def}} (t_4)_{x,y,z,w,f,g,h,g \circ^C f, h \circ^C (g \circ^C f), h \circ^Cg} \colon \underset{\begin{subarray}{1} \tau_1 \colon \circ(f,g,g \circ^C f,x,y,z) \\ \tau_2 \colon \circ(g \circ^C f,h,h \circ^C (g \circ^C f),x,z,w) \\ \tau_3 \colon \circ(g,h,h \circ^C g,y,z,w)  \end{subarray}}{\Pi} \hspace{-1.1cm}\circ(f,h \circ^C g,h \circ^C (g \circ^C f), x,y,w)
%\end{split}
\end{equation}
We can then define
\begin{align*}
p_1 &=_{\text{def}} \TTproj_2 (c_{x,y,z,f,g}) \\
p_2 &=_{\text{def}} \TTproj_2 (c_{x,z,w,g \circ^Cf,h}) \\
p_3 &=_{\text{def}} \TTproj_2 (c_{y,z,w,g,h}) \\
\end{align*}
and thus $r_{p_1,p_2,p_3} \colon \circ(f,h \circ^C g,h \circ^C (g \circ^C f), x,y,w)$. But by definition we have a term $\pi \colon \circ(f,h \circ^C g,(h \circ^C g) \circ^C f, x,y,w)$ and therefore we get a term 
\begin{equation}
(t_{11})_{x,y,w,f,f,h \circ^C g,h \circ^C g, \mathtt{refl}_f, \mathtt{refl}_{h \circ^C g}, h \circ^C (g \circ^C f), (h \circ^C g) \circ^C f,  r_{p_1, p_2, p_3}, \pi} \colon h \circ^C (g \circ^C f)  = (h \circ^C g) \circ^C f
\end{equation}
Thus we can define
\begin{equation}
\mathtt{assoc}^C_{x,y,z,w,f,g,h} =_{\text{def}} (t_{4})_{x,y,w,f,f,h \circ^C g,h \circ^C g, \mathtt{refl}_f, \mathtt{refl}_{h \circ^C g}, h \circ^C (g \circ^C f), (h \circ^C g) \circ^C f,  r_{p_1, p_2, p_3}, \pi}
\end{equation}
and this gives us the section 
\begin{equation}
\mathtt{assoc}^C \colon  \underset{a,b,c,d \colon C}{\Pi} \underset{\begin{subarray}{c} 
f \colon \text{Hom}_{\C} (a, b) \\
g \colon \text{Hom}_{\C} (b, c) \\
h \colon \text{Hom}_{\C} (c, d)
\end{subarray}}{\Pi} h \circ (g \circ f) = (h \circ g) \circ f
\end{equation}
as required by condition (5).  
%Finally, to verify condition (6) we proceed by an exactly analogous argument to condition (5). 
Condition (6) follows similarly.
%We omit the details. 
%using $t_9$ and $t_{10}$ in place of $t_{11}$. We will call the sections produced $\mathtt{idl}^C$ and $\mathtt{idr}^C$ for the left and right unit laws respectively and leave the details of the verification to the reader.
%\[
%f \circ g =_{\text{def}} \mathtt{pr}_1(c_{f,g}) \colon A(x,z)
%\]
%\end{proof}
%Given any $\mathbb{T}_{\text{cat}}$-model $C$ we will 
So we can now write $p(C)$ for the precategory given by the data 
\begin{equation}
(O,A, 1^C, \circ^C, \mathtt{assoc}^C, \mathtt{idl}^C, \mathtt{idr}^C)
\end{equation}

Conversely, we need to define a function
\begin{equation}
q \colon \textbf{PreCat} \rightarrow \textbf{Mod}(\mathbb{T}_{\text{precat}})
\end{equation}
So let $C$ be a precategory given by the data
\begin{equation}
(C,\text{Hom}, 1, \circ, \mathtt{assoc}, \mathtt{idl}, \mathtt{idr})
\end{equation}
Given $1$ we know that for each $x \colon C$, $\text{Hom}(x,x)$ is inhabited since $1_x \colon \text{Hom}(x,x)$. Thus we can define
\begin{equation}
I_C =_{\text{def}} \lambda x.(\lambda f.(f=1_x)) \colon \underset{x \colon C}{\Pi} \text{Hom}(x,x) \rightarrow \textbf{Prop}_{\mathcal{U}}
\end{equation}
where we know that $f=1_x$ is a mere proposition since $\text{Hom}(x,x)$ is an $h$-set.
%\[
%I_C =_{\text{def}} \lambda x. !_x \colon \underset{x \colon C}{\Pi} \text{Hom}(x,x) \rightarrow \textbf{Prop}_{\mathcal{U}}
%\]
Exactly analogously, since $\circ$ ensures that each type $\text{Hom}(x,z)$ will be inhabited given $f \colon \text{Hom}(x,y)$ and $g \colon \text{Hom}(y,z)$ we define 
\begin{equation}
\circ_C \colon \underset{x,y,z \colon C}{\Pi} \text{Hom}(x,y) \rightarrow \text{Hom}(y,z) \rightarrow \text{Hom}(x,z) \rightarrow \textbf{Prop}_{\mathcal{U}}
\end{equation}
The verification of the axioms for this data is then entirely straightforward and we omit the details.
%$T_1$ is inhabited since $!_x(1_x) \colon I(1_x,x,x)$ for any $x$. Similarly for $T_5$. For $T_6, f9, T_{10}$ we use $\mathtt{assoc}, \mathtt{inl}$ and $\mathtt{inr}$ in the obvious manner. Finally, $T_7$ and $T_{11}$ follow immediately from the transitivity of $=$ given the definitions of $I_C$ and $\circ_C$. 
%\end{proof}
So we can now write $q(C)$ for the $\mathbb{T}_{\text{precat}}$-model given by the data
\begin{equation}
(C,\text{Hom},I_C,\circ_C, t^C_1, t^C_2, t^C_3, t^C_4, t^C_5, t^C_6, t^C_7)
\end{equation}
%where $t^C_i$ stands for the proof of $T_i$ as obtained in the proof of Theorem \ref{precattomod}. Clearly, as in the case of $p$, we get an arrow
%\[
%q \colon \textbf{PreCat} \rightarrow \textbf{Mod}(\mathbb{T}_{\text{cat}})
%\]
It is straightforward to check that $p$ and $q$ are quasi-inverses which gives us the required equivalence.
\end{proof}
%We have thus established the following.

%\begin{cor}\label{precatax}
%$\textbf{\emph{PreCat}} \simeq \textbf{\emph{Mod}}(\mathbb{T}_{\text{cat}})$
%\end{cor}

%Proposition \ref{modtoprecat} 
%%may be understood as saying that the type of 
%says that precategories is ``$\infty$-elementary'' in the sense that they are axiomatizable, up to equivalence, by a theory over an $\infty$-signature.

\def\Lstrcat{\L_{\text{strcat}}}
\def\Tstrcat{\T_{\text{strcat}}}
\def\Tstrcate{\T_{\text{\emph{strcat}}}}
\def\IdOM{\text{Id}_{O^\M}}

A \emph{strict category} (\cite{HTT}, Definition 9.6.1) is a precategory in which the type of objects is an $h$-set. 
We write \textbf{StrCat} for the type of strict categories.
Let $\Lstrcat$ be the signature whose category part is the same as $\Lcat$ but with $h(O)=2$.
Let $\T_{\text{strcat}}$ be the $\Lstrcat$-theory that contains the same axioms as $\Tcat$.
%together with the following axiom which expresses that $O$ is an $h$-set:
%\begin{enumerate}
%\setcounter{enumi}{7}
%\item[(8)] $\forall x,y \colon O. \exists p \colon x= y. \forall q \colon x=y. p=q$
%\end{enumerate}
%We can now show that strict categories are also ``$\FOLiso$-elementary.''
We now immediately obtain the following.

\begin{cor}\label{strcatprop}
$\textbf{\emph{StrCat}} \simeq \textbf{\emph{Mod}}(\Tstrcate)$
\end{cor}
A \emph{univalent category} (\cite{HTT}, Definition 9.1.6) is a precategory satisfying the following additional condition:
%, which expresses the fact that the canonical map $\mathtt{idtoiso}_{a,b} \colon a=b \rightarrow a \cong b$ is an equivalence for all $a,b$:
\begin{enumerate}
\item[(7)] $\mathtt{cat} \colon \underset{a,b \colon C}{\Pi} \mathtt{isequiv} ( \mathtt{idtoiso}_{a,b})$
\end{enumerate}
We write \textbf{Unicat} for the type of univalent categories.
%To see this, 
Now let $\Lucat$ be the following (associated) $\FOLiso$-signature
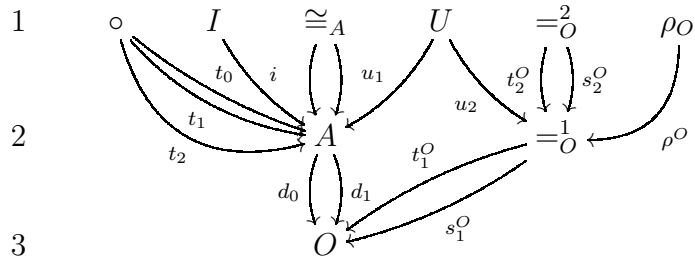
\[
\xymatrix{
1 &\circ \ar@/_5pt/[rrd]^{t_0} \ar@/_10pt/[rrd]_{t_1} \ar@/_25pt/[rrd]_{t_2} & I \ar@/_5pt/[rd]^{i}  &\eqsort_A \ar@/^/[d] \ar@/_/[d] &U \ar@/^/[ld]_{u_1} \ar@/_/[rd]_{u_2} &=^2_O \ar@/^/[d]^{s^O_2} \ar@/_/[d]_{t^O_2}  &\rsort_O \ar@/^20pt/[ld]^{\rho^O} \\
2& & & A \ar@/_/[d]_{d_0} \ar@/^/[d]^{d_1} & &=_O^1 \ar@/^/[lld]^{s_1^O} \ar@/_/[lld]_{t^O_1} \\
 3 & & & O \\
}
\]
%Thus $\Lucat$ is obtained from $\Lcat$ (but with $d(O)=1$) by first globularly completing, and the extending by the two-place relation $U$ 
subject to all the same relations as $\Lcat$ as well as the additional relation $t_1^O u_1 = s^O_1 u_2$. 
%(Note that $d(U)=-1$ necessarily.)
%($\Lucat$ contains no extra relations than those added in the globular completion stage.)
%Univalent categories can now be axiomatized as an $\Lucat$-theory.
We can then define $\T_{\text{ucat}}$ as the $\Lucat$-theory given by the axioms of $\T_{\text{cat}}$ together with the following extra axioms:
%some extra axioms that ensure that $U$ behaves like a functional relation and, crucially, 
%, i.e. the ``invariant'' version of category theory in UF. 
%To do so one considers the axioms of $\Tcat$ (as displayed above) and adds the following ``univalent'' axiom which we will call 
%he following ``Axiom U'':
\begin{enumerate}[(1)]
\setcounter{enumi}{7}
\item $\forall x,y \colon O. \forall f \colon A(x,y). \text{Iso}(f) \rightarrow (\exists ! p \colon x=_O^1y. U(f,p))$
\item $\forall x,y \colon O. \forall f \colon A(x,y). \forall p \colon x=_O^1y. U(f,p) \rightarrow \text{Iso}(f)$
\item $\forall x \colon O. \forall f \colon A(x,x). \forall p \colon x=_O^1x.(I(f,x) \wedge U(f,p) \rightarrow r_O^1(p,x))$
\item $\forall x,y \colon O. \forall f \colon A(x,y). \forall p,q \colon x=_O^1 y. (U(f,p) \wedge U(f,q) \rightarrow p=_O^2q)$
%\item $
\end{enumerate}
%In the above, Iso$(f)$ is an abbreviation for
where we have used the abbreviations
\[
\text{Iso}(f) \eqdef \exists g \colon A(x,y) \exists h_1 \colon A(x,x) \exists h_2 \colon A(y,y).
%\\ & \: \: \: \: \: \: \: \: \: \:
\circ(f,g,h_1) \wedge \circ(g,f,h_2) \wedge I(h_1) \wedge I(h_2) 
\]
and
\[
\exists ! p \colon x=_O^1y. U(f,p) \eqdef \exists p \colon x=_O^1y.(U(f,p) \wedge (\forall q \colon x=_O^1y.(U(f,q) \rightarrow p=_O^2q)))
\]
Thus, axioms (8)-(10) express that $U$ is a bijective relation between isomorphisms and paths that sends identity to reflexivity and axiom (11) expresses that $U$ a functional relation.
%We can then define $\T_{\text{ucat}}$ as the $\Lucat$-theory given by the axioms of $\T_{\text{cat}}$ together with Axiom U (as well as some extra book-keeping axioms). 
%If we write \textbf{Unicat} for the type of univalent categories in UF, the claim is then the following:
We can now show that univalent categories are also ``$\FOLiso$-elementary''.

\begin{prop}\label{unicatprop}
$\textbf{\emph{UniCat}} \simeq \textbf{\emph{Mod}}(\T_{\text{\emph{ucat}}})$
\end{prop}

%To prove it, we require a choice principle and so we must assume a stronger metatheory than we have so far been assuming. The reason we need choice is in order to change the interpretation of $U$ from a functional relation (indeed bijection) to an actual function. Thus, the proof -- though uncomplicated -- must be taken to take place in HoTT+AC (and therefore, by Diaconescu's theorem, in a classical setting) where AC is understood as in \cite{HTT}, Definition .

\def\Tucat{\mathbb{T}_{\text{ucat}}}

\begin{proof}
From Proposition \ref{modtoprecat} we can assume that the data for a model of $\Tucat$ is given by the same data as that of for a precategory, together with the interpretation of $U$. Given AUC, for any given $\Tucat$-model $\M$ we can extract from $U^\M$ a section
\begin{equation}
u^\M \colon \underset{x,y \colon O^\M}{\Pi} \mathtt{Id}_{O^\M} (x,y) \rightarrow \mathtt{Iso}^\M(x,y)
\end{equation}
where 
\begin{equation}
\mathtt{Iso}^\M(x,y) \equiv \underset{f \colon A^\M(x,y)}{\Pi} \mathtt{isiso}(f)
\end{equation}
such that
\begin{equation}
\pi \colon \underset{x,y \colon O^\M}{\Pi} \mathtt{isequiv}(u_{x,y})
\end{equation}
Thus we get
\begin{equation}
\textbf{Mod}(\Tucat) \simeq \underset{\begin{subarray}{1} \hspace{1.3cm} C \colon \textbf{Precat} \\ u \colon \underset{x,y \colon O^C}{\Pi} \mathtt{Id}_{O^C} (x,y) \rightarrow \mathtt{Iso}^C(x,y) \end{subarray}}{\Sigma} \: \underset{x,y \colon O^C}{\Pi} \mathtt{isequiv}(u_{x,y})
\end{equation}
We can now take \textbf{Unicat} to be the type
\begin{equation}
\underset{C \colon \textbf{Precat}}{\Sigma} \: \underset{x,y \colon O^C}{\Pi} \mathtt{isequiv}(\mathtt{idtoiso}_{x,y})
\end{equation}
There is then a natural map $f$ from \textbf{UniCat} to $\textbf{Mod}(\Tucat)$ which sends $\langle D, \mathtt{univ} \rangle$ to $\langle D, \mathtt{idtoiso}, \mathtt{univ} \rangle$. For a given 
\begin{equation}
\langle D, u, \pi \rangle \colon \textbf{Mod}(\Tucat)
\end{equation}
the homotopy fiber of $f$ over $\langle D, u, \pi \rangle$ is given by
\begin{equation}
\mathtt{hfib}_f (\langle D, u, \pi \rangle) \equiv \underset{\langle C ,\mathtt{univ} \rangle \colon \textbf{UniCat}}{\Sigma} \langle C, \mathtt{idtoiso}, \mathtt{univ} \rangle = \langle D, u, \pi \rangle 
\end{equation}
To show that $\mathtt{hfib}_f (\langle D, u, \pi \rangle)$ is contractible it clearly suffices to show that for all $x,y \colon O$, $u_{x,y} = \mathtt{idtoiso}_{x,y}$ and by function extensionality this reduces to giving an inhabitant of
\begin{equation}
\underset{x,y,p}{\Pi} u_{x,y} (p) = \mathtt{idtoiso}_{x,y}(p)
\end{equation}
which by path induction reduces to providing an inhabitant of
\begin{equation}
\underset{x}{\Pi} u_{x,x} (\mathtt{refl}_x) = \mathtt{idtoiso}_{x,x}(\mathtt{refl}_x)
\end{equation}
But by the axioms of univalent categories and of $\Tucat$ we get that both sides of the equation are (propositionally) equal to the (unique) identity map on $x$. Thus $f$ is an equivalence and we are done.
\end{proof}

%\begin{remark}
%Clearly the axioms of $\Tucat$ employ full first-order logic. It is unclear whether \textbf{Unicat} can also be axiomatized in the coherent fragment over $\Lucat$.
%\end{remark}

Proposition \ref{unicatprop} illustrates the kind of result that $\FOLiso$ was designed to tackle.
Namely we want to use $\FOLiso$ as a tool for answering traditional model-theoretic questions (e.g. of elementariness) 
but about structures defined on homotopy types rather than sets.
%phrased in the setting of homotopy types instead of sets. 
From this point of view, many future projects and questions suggest themselves. We list a few, in descending order of priority and ascending order of speculation:
\begin{enumerate}
%\item Extending the syntax of $n$-logic to the case $n = \infty$ in such a way as to make it possible to define semantics as in Definition \ref{HLstructure}. (This is related to the well-known open problem of managing infinite chains of coherence data in HoTT.)
%The main difficulty here is that the trick of splitting the rules for isomorphism sorts will no longer work.
\item Defining the categorical semantics of $\FOLiso$ (possibly in terms of Reedy-fibrant diagrams as in \cite{shulman2015reedy}). In particular, to define an initial categorical semantics for $\LTT_\L$ for each $\L$. 
\item Proving a general completeness theorem for the homotopy semantics with respect to the deductive system $\Diso$. 
%The main obstruction here is to package the syntax in a sufficiently efficient manner so that the actual constructions can be written out in full.
\item Defining a $\FOLiso$-signature $\L_{UF}$ and an $\L_{UF}$-theory $\T_{UF}$ that can serve as an axiomatization of the universe of $\infty$-groupoids, i.e. a $\FOLiso$-axiomatization of $\infty$-groupoids.
\item Characterizing $\FOLiso$-elementary types in general, i.e. proving a Lo\'{s} Theorem for $\FOLiso$ along the lines of \cite{CK}, Theorem 4.1.12. 
%(This is likely to be non-trivial even in the case of $n=1$.)
\item Characterizing ``homotopy categorical'' theories, i.e. finding necessary and sufficient conditions on a $\FOLiso$-theory $\T$ such that $\Mod{\T}$ is contractible.
\end{enumerate}

Finally, we ought to make a few remarks concerning the setting in which the investigation of $\FOLiso$ can take place. Everything that we have done in this paper can be done either in (some) HoTT or in set theory. By this we mean that, in addition to defining $\FOLDS(\Ob,\Mor)$, $\hSig$, $\FOLiso$ etc., for each fixed $\FOLiso$-signature $\L$ we are able to define $\HStruc{\L}$, $\LTT_\L$ and the interpretation of $\LTT_\L$ into MLTT either in type theory or in set theory.
But as we mentioned in Remark \ref{functoriality} we should ideally like to regard $\HStruc{-}$ as an ``interpretation functor'' $\hSig \rightarrow \U$ which takes each $\L$ to the ``correct'' $\Sigma$-type in $\U$. 
In particular, having such a functor would allow us to consider ``infinitary'' constructions on $\hSig$ which we could then reflect onto $\U$.
For example, as has been outlined in \cite{TsemWeaverTTIabstract}, by defining a function $sst \colon \mathbb{N} \rightarrow \hSig$ picking out for each $n$ the signature for the $n$-truncated semi-simplicial types and then taking the (homotopy) limit over $\HStruc{sst(-)}$ in $\U$ would give us a definition of the type of semi-simplicial types.
Unfortunately it is unlikely that such an ``interpretation functor'' can be defined in the standard HoTT, due to the well-known coherence problems. 

One option would would be to move to a so-called ``two-level type theory'' (cf. \cite{2LTT, Altenkirch16, HTS}) and, e.g., define inverse categories whose relations are satisfied strictly, building on what is done in \cite{2LTT}, Definition 10. 
Our preferred option, however, is to work in a HoTT with a \emph{postulated} interpretation function, as has been announced in \cite{TsemWeaverTTIabstract}, under the name of TT+I, but which perhaps should more appropriately be called HoTT+I. We believe this is the right setting in which to develop and apply $\FOLiso$, and we believe HoTT+I to be in many ways a more correct overall formalization of the Univalent Foundations since it is a univalent type theory (in the sense that there is only one, univalent, identity type) but which can still allow us to carry out constructions that are conjectured to be impossible to carry out in standard HoTT, e.g. defining semi-simplicial types. 

We believe a HoTT along the lines of HoTT+I is the right setting in which to study $\FOLiso$ natively within the Univalent Foundations.
Such a study should be understood as a model theory within the Univalent Foundations.

\vspace{0.3cm}

\noindent\textbf{Acknowledgements.} I thank Benedikt Ahrens, Steve Awodey, Andrej Bauer, John Burgess, Thierry Coquand, Harry Crane, Nicola Gambino, Dan Grayson, Chris Kapulkin, Peter LeFanu Lumsdaine, Anders M\"{o}rtberg, Paige North, Mike Shulman, Vladimir Voevodsky and Matthew Weaver for helpful and stimulating conversations and emails.  I would also like to single out in thanks an anonymous referee for very extensive and helpful comments that significantly improved and substantially extended the scope of the paper.

\bibliographystyle{shortalphabetic}

\bibliography{nlogicrefs}

\end{document}